\documentclass[a4paper]{amsart}

\usepackage[leqno]{amsmath}
\usepackage{amsthm}
\usepackage{amssymb}
\usepackage{graphicx}
\usepackage{calc} 
\usepackage{txfonts}
\usepackage{enumitem}
\usepackage{ifpdf} 
\usepackage{ifthen}
\usepackage{braket}
\usepackage{xcolor}
\usepackage[all]{xy}
\ifpdf
  \usepackage[pdftex]{hyperref}
  \hypersetup{plainpages=false,unicode=true,pdffitwindow=true}
\else
  \usepackage[dvips]{hyperref}
\fi
\usepackage{dsfont}

\theoremstyle{plain}
\newtheorem{theorem}{Theorem}[section]
\newtheorem{lemma}[theorem]{Lemma}

\newtheorem{corollary}[theorem]{Corollary}

\theoremstyle{definition}
\newtheorem{definition}[theorem]{Definition}

\newtheorem{remark}[theorem]{Remark}

\newcommand{\R}{\mathbb{R}}

\newcommand{\f}[2]{\frac{#1}{#2}}
\newcommand{\pa}{\partial}
\newcommand{\half}{\f{1}{2}}

\newcommand{\abs}[1]{\left\lvert#1\right\rvert} 
\newcommand{\norm}[1]{\left\lVert#1\right\rVert} 
\newcommand{\ip}[2]{\left\langle #1\,,#2\right\rangle} 

\DeclareMathOperator*{\argmin}{arg\, min}

\DeclareMathOperator{\var}{var}

\DeclareMathOperator{\trace}{trace}
\DeclareMathOperator{\erfc}{erfc}

\usepackage{todonotes}
\usepackage{mathrsfs}

\newcommand{\pvar}{p-\mathrm{var}}

\newcommand{\E}{E}
\DeclareMathOperator{\spn}{span}
\renewcommand{\doteq}{\coloneqq}
\renewcommand{\hat}{\widehat}

\numberwithin{equation}{section}

\author[C. Bayer]{Christian Bayer}
\address{Weierstrass Institute, Mohrenstr.~39, 10117 Berlin, Germany}
\email{christian.bayer@wias-berlin.de}

\author[D. Belomestny]{Denis Belomestny}
\address{Universit\"{a}t Duisburg-Essen, Fachbereich Mathematik,
  Thea-Leymann-Str.~9, 45127 Essen, Germany}
\email{denis.belomestny@uni-due.de}

\author[M. Redmann]{Martin Redmann}
\address{Weierstrass Institute, Mohrenstr.~39, 10117 Berlin, Germany}
\email{martin.redmann@wias-berlin.de}

\author[S. Riedel]{Sebastian Riedel}
\address{TU Berlin, Institut f\"{u}r Mathematik, Stra\ss{}e des 17. Juni 136,
  10623 Berlin, Germany}
\email{riedel@math.tu-berlin.de}

\author[J. Schoenmakers]{John Schoenmakers}
\address{Weierstrass Institute, Mohrenstr.~39, 10117 Berlin, Germany}
\email{john.schoenmakers@wias.berlin.de}

\title{Solving linear parabolic rough partial differential equations}

\thanks{C.B., M.R., S.R., and J.S.~gratefully acknowledge support from the DFG
through the research unit FOR2402.}



\begin{document}

\begin{abstract}
  We study linear rough partial differential equations in the setting of [Friz
  and Hairer, Springer, 2014, Chapter 12]. More precisely, we consider a
  linear parabolic partial differential equation driven by a deterministic
  rough path $\mathbf{W}$ of H\"{o}lder regularity $\alpha$ with $1/3 < \alpha
  \le 1/2$.  Based on a stochastic representation of the solution of the rough
  partial differential equation, we propose a regression Monte Carlo algorithm
  for spatio-temporal approximation of the solution.  We provide a full
  convergence analysis of the proposed approximation method which essentially
  relies on the new bounds for the higher order derivatives of the solution in
  space. Finally, a comprehensive simulation study showing the applicability
  of the proposed algorithm is presented.
\end{abstract}

\maketitle

\section{Introduction}
\label{sec:introduction}

We consider linear rough partial differential equations in the setting of
Friz and Hairer~\cite[Chapter 12]{FH14}, see also Diehl, Oberhauser,
Riedel~\cite{DOR15} and Diehl, Friz and Stannat~\cite{DFS17}, i.e.,
\begin{gather*}
  -du = L(u) \, dt + \sum_{k=1}^d \Gamma_k(u) d\mathbf{W}^k,\\
  u(T, \cdot) = g,
\end{gather*}
where the differential operators $L$ and $\Gamma = (\Gamma_1, \ldots,
\Gamma_d)$ are defined by
\begin{gather*}
  Lf(x) = \half \trace\left( \sigma(x) \sigma(x)^\top D^2 f(x) \right) +
  \ip{b(x)}{Df(x)} + c(x) f(x),\\
  \Gamma_kf(x) = \ip{\beta_k(x)}{Df(x)} + \gamma_k(x) f(x),
\end{gather*}
see Section~\ref{sec:stoch-repr} for more details. We stress here that
$\mathbf{W}$ is a \emph{deterministic} rough path (of H\"{o}lder regularity
$\alpha$ with $1/3 < \alpha \le 1/2$), i.e., the PDE above is considered as a
deterministic, not a stochastic equation. (This does not, of course, preclude
choosing individual trajectories produced by a stochastic process, say a
fractional Brownian motion.)

The goal of this paper is to provide a numerical algorithm for solving the
above rough partial differential equation together with a proper numerical
analysis of the algorithm and numerical examples. More precisely, we want to
approximate the function $(t,x) \mapsto u(t,x)$ as a linear combination of
some easily computable basis functions depending on \(x\) with time dependent
coefficients. Such approximations can then be, for example, used to solve
optimal control problems for linear rough PDEs. In this respect, our approach
can be viewed as an alternative to the space-time Galerkin proper orthogonal
decomposition method used to solve optimal control problems for the standard
linear parabolic PDEs (see, e.g. \cite{baumann2016space} and references
therein).  We analyze the corresponding approximation error which turns out to
depend on smoothness properties of the solution \(u.\) As a by-product of this
analysis, we also proved regularity of the solution $u$ in $x$ of degree
larger than $1$ under suitable conditions.

Rough partial differential equations of the above kind appear under the name
Zakai equation in filtering, see, for instance,
Pardoux~\cite{pardoux1991filtrage}. There, $W$ corresponds to an actually
observed path, and we would like to get information on an underlying
unobserved path driving the dynamics of $W$. Hence, by its nature, $W$ is a
deterministic path as far as the Zakai equation is concerned. Note that,
generally speaking, the Zakai equation is only of the form of a partial
differential equation if the underlying model is a diffusion type, which
essentially implies that $W$ is (one trajectory of) a Brownian motion. But
even in this case the use of rough path theory makes sense, as it guarantees
continuity of the solution in terms of the driving noise, and hence allows for
a pathwise solution. If, for instance, the underlying system for the filtering
problem is a fractional Brownian motion, then the corresponding Zakai equation
will be path-dependent, see, for instance, Coutin and
Decreusefond~\cite{coutin1999} for the case of Hurst index $H > \half$. The
case $H < \half$ still seems to be open.

\subsection{Literature review}
\label{sec:literature-review}

Terry Lyons'~\cite{lyons98} theory of rough paths provides a deterministic,
pathwise analysis of stochastic ordinary differential equations. This has
interesting consequences both from a theoretical point of view---often based
on the continuity of the solution w.r.t.~the driving noise (which is not true
in the classical stochastic analysis framework)---and from a practical point
of view---see, for instance, \cite{lyons2014rough}. We refer to
\cite{FV10,FH14} for background information on rough path theory.

Nonetheless, the seemingly obvious step from rough ODEs the rough PDEs turns
out to be quite difficult, mainly because of two essential limitations of
standard rough path theory: regularity of the vector fields driving the
differential equation (lacking in the case of (unbounded) partial differential
operators), and the restriction to paths, i.e., functions parametrized by a
one-dimensional variable. While not relevant for this paper, we should mention
that the second restriction was overcome by seminal work of
Hairer~\cite{hairer13}, thereby allowing space-time noise.

Despite those difficulties, rough partial differential equations (driven by a
true path, i.e., a ``noise'' component only depending on time, but not space)
have become a thriving field in the last few years, and several approaches
have been developed to extend rough path analysis to rough PDEs (RPDEs). Most
approaches are based on transformations of the problem separating the
roughness of the drivers from the non-regularity of the differential
operators. A series of papers by Friz and co-authors derives existence and
uniqueness results for some classes of RPDEs by applying a flow-transformation
to a classical PDE (with random coefficients), for instance see
\cite{CFO11,FGLS17}. Other works in this flavour are based on mild
formulations of the RPDE, e.g., Deya, Gubinelli and Tindel~\cite{DGT12}.

Some more recent works have focused on more intrinsic formulations of rough
PDEs, trying to extend classical PDE techniques to the rough PDE context. This
paper is based on the Feynman-Kac approach of Diehl and co-authors
\cite{DFS17,DOR15,FH14}, which is presented in more detail in
Section~\ref{sec:stoch-repr}. In a quite different vein, Deya, Gubinelli,
Hofmanov\'{a} and Tindel~\cite{DGHT16} have provided a rough Gronwall lemma,
which makes classical approaches to weak solutions of PDEs accessible.

Despite the increasing interest in rough PDEs, so far no numerical schemes
have been suggested to the best of our knowledge. In this context, let us
again mention \cite{DGHT16}, which could open up the field to finite element
methods, as it provides variational techniques for some classes of rough
PDEs. Of course, an abundance of numerical methods exist for classical
stochastic PDEs and PDEs with random coefficients, see, for instance,
\cite{Kruse14}.

In this work, we will use the stochastic representation of~\cite{DFS17} in
order to build a regression based approximation of the solution $u(t, \cdot)$
of the rough PDE, a technique that has been successfully applied both to stochastic PDEs
by Milstein and Tretyakov~\cite{MT12} and to PDEs
with random coefficients by Anker et al.~\cite{Ank2017}.

\subsection{Outline of the paper and main results}
\label{sec:outline-paper}

Diehl, Friz and Stannat~\cite{DFS17} provide a solution theory to the rough
partial differential equation above by means of a 
\emph{stochastic representation}, i.e., they construct a stochastic process
$X$ which is driven by a stochastic rough path $\mathbf{Z}$ constructed from a
Brownian motion $B$ and our rough path $\mathbf{W}$ driving the rough PDE. The
solution $u$ of the rough PDE then is given as a conditional expectation of a
functional of $X$, see Section~\ref{sec:stoch-repr} and, in particular,
\eqref{eq:stoch_rep} for more details.

For the numerical approximation of $u(t, \cdot)$, it is very important to
understand the regularity of this map. Note that the regularity in $t$ quite
clearly corresponds to the regularity of the driving path $\mathbf{W}$,
whereas regularity in space alone can be much better depending on the
coefficients and the terminal data $g$. In the theoretical work~\cite{DFS17},
spacial regularity of $u$ is obtained from regularity of the stochastic
process $X$ in its initial value $X_0 = x$, which is well understood for rough
differential equations. In order to show regularity of $u$ one, however, needs
to interchange differentiation with expectation, and the required
integrability conditions were only available for the first derivative (see
Cass, Litterer, Lyons~\cite{CLL13}), but not for higher derivatives. In
Section~\ref{sec:regularity-solution} we extend these results to
higher derivatives, which enables us to show the following theorem, see
Corollary~\ref{cor:main}:
\begin{theorem}
  \label{thr:main-regularity}
  Let $u(t,x)$ be as above. Assume that $g$ is $k$-times differentiable with
  $g$ and its derivatives having at most exponential growth.  Assume further
  that $b$ and $\beta$ are bounded, $(2 + k)$-times continuously
  differentiable with bounded derivatives, $\sigma$ is bounded with
  $(3+k)$-times continuously differentiable with bounded derivatives, and
  $\gamma$ and $c$ are bounded, $k$-times continuously differentiable with
  bounded derivatives. Then $u(t, \cdot)$ is $k$ times continuously
  differentiable and we provide explicit bounds on the derivatives.
\end{theorem}
For any fixed \(t>0,\) the spacial resolution of the function $x \mapsto
u(t,x)$ can be approximated using regression with respect to properly chosen
basis functions $\psi_1, \ldots, \psi_K$, $K \in \mathbb{N}$. More precisely,
given a specific probability measure $\mu$ on the state space $\R^n$, we try
to minimize the error in the sense of $L^2(\mu)$, i.e., we would like to find
\begin{equation*}
  \argmin_{\tilde{v} \in \spn\{\psi_1, \ldots, \psi_K\}} \int_{\R^n}
  \abs{u(t,x) - \tilde{v}(x)}^2 \mu(dx).
\end{equation*}
The above loss function can, however, only serve as a guiding principle, since
$u(t,x)$ is not available to us. Instead, we replace the above loss function
by a proper Monte Carlo approximation: denoting the actual stochastic
representation of $u$ by $\mathcal{V} = \mathcal{V}(t,x,\omega)$ in the sense
that $u(t,x) = E[\mathcal{V}(t,x)]$, we consider samples $\mathcal{V}^{(m)}$
of $\mathcal{V}$ obtained by
\begin{enumerate}
\item sampling initial values $x^{(m)}$ according to the distribution $\mu$;
\item sampling the solution of $X$ started at $X_t = x^{(m)}$ driven by \label{samplingprocedure}
  independent (of each other and of $x^{(m))}$) samples of the Brownian
  motion. 
\end{enumerate}
Finally, we construct an approximation $\tilde{v}$ of $u(t,\cdot)$ by
(essentially) solving the least squares problem
\begin{equation*}
  \argmin_{\tilde{v} \in \spn\{\psi_1, \ldots, \psi_K\}} \sum_{m=1}^M \f{1}{M}
  \abs{\mathcal{V}^{(m)} - \tilde{v}(x^{(m)})}^2.
\end{equation*}
(In addition to the ``true'' stochastic regression as indicated here, we actually
also use a ``pseudo-regression'' proposed in~\cite{Ank2017} for PDEs with random
coefficients and presented in detail in Section~\ref{sec:regression}.) We
obtain (cf.~Theorem~\ref{psth}):
\begin{theorem}
  \label{thr:main-2}
  Under some boundedness conditions on the solutions and its stochastic
  representation, there is a constant $C > 0$ (which can be made explicit)
  such that
  \begin{equation*}
    E\left[ \int_{\R^n} \abs{u(t,x) - \tilde{v}(x)}^2 \mu(dx) \right] \le C
    \f{K}{M} + \inf_{w \in \spn\{\psi_1, \ldots, \psi_K\}} \int_{\R^n}
    \abs{u(t,x) - w(x)}^2 \mu(dx).
  \end{equation*}
\end{theorem}

In order to find an approximation \(u(t,x)\) on a time grid \(0<t_1<\ldots<t_L<T\)
the entire procedure (\ref{samplingprocedure}) has to be repeated for every time step, i.e., we generate
samples of $X$ starting in $x^{(m)}$ at the respective initial time $t_l$. In Section~\ref{seq:temp-spat} we develop an alternative regression type algorithm which
allows for approximating the solution \(u(t,x)\) on \(0<t_1<\ldots<t_L<T\) using only one set of trajectories
of the process \(X\) with $X_0 = x^{(m)}$.

Under somewhat more restrictive assumptions, our
approximation \(\overline{u}\) satisfies\begin{eqnarray*}
\label{eq:upper_bound_time}
E\left[\int_{\R^n}|\overline{u}(t_l,x)-u(t_l,x)| \, \mathrm{P}%
_{X_{t_l}}(dx)\right]&\leq&
C\left[\frac{K\,\log(M)}{M}\right.
\\
&& \left.+\inf_{w \in \spn\{\psi_1, \ldots, \psi_K\}} \int_{\R^n}
    \abs{u(t_l,x) - w(x)}^2 \mathrm{P}%
_{X_{t_l}}(dx)\right],
\end{eqnarray*}
for \(l=1,\ldots,L.\)

In order to  bound the corresponding approximation errors 
\[\inf_{f\in\text{span}\{\psi_{1},...,\psi_{K}\}}\Vert
u(t_l,\cdot)-f(\cdot)\Vert_{L^{2}(\varrho)}^{2},\quad l=1,\ldots,L,
\] 
where the measure $\varrho$ is either $\mu$ or $\mathrm{P}_{X_{t_l}}$, one needs to specify the basis functions.

In Section~\ref{seq:app_err} we show that in the case of piecewise polynomial
basis functions, the approximation error can be bounded (up to a constant) by
\(K^{-\varkappa}\) with \(\varkappa=\frac{2\nu(q+1)}{n(\nu+2(q+1))},\)
provided that each function \(u(t_l,\cdot)\) is \(q+1\) times differentiable,
irrespective of the chosen regression type. Now the smoothness of
\(u(t_l,\cdot)\) follows from the smoothness of the coefficients of the
underlying PDE and \(g\) via Theorem~\ref{thr:main-regularity}.

A last puzzle piece is still missing if we want to provide a fully
implementable approximation scheme, since we still need to solve the rough
differential equation describing $X$. Here we employ a (simplified) Euler-type
scheme including approximations of the needed signature terms by polynomials
of the path itself, see Bayer, Friz, Riedel and Schoenmakers~\cite{BFRS16} for
details. The scheme is recalled in Section~\ref{sec:Euler}. In the current
context, the rate of convergence of the scheme is (almost) $2\alpha - 1/2$,
see Theorem~\ref{thr:Euler_rate} for details.
Finally, we give several numerical examples in
Section~\ref{sec:numerical-examples}. 

\begin{remark}
  \label{rem:stochastic_extension}
  The scope of this paper is solving deterministic rough PDEs, i.e., PDEs
  driven by a deterministic but rough path $\mathbf{W}$. What happens when
  $\mathbf{W}$ is instead assumed to be random -- implying randomness of u? If
  we apply the same algorithm as above, but with the samples of $X$ based on
  i.i.d.~samples of $\mathbf{W}$, then the regression based approximation is
  an estimate for $E[u(t,x)]$ (see~\cite{Ank2017} for the case of regular
  random noise). Of course, we can also use a regression approach for solving
  the full random solution $u(t,x;\omega)$. In this case, we need to choose
  basis functions in both $x$ and $\omega$. This means, proper basis functions
  need to be found in $\omega$ -- or rather, in $\mathbf{W}$. The
  \emph{signature} of $\mathbf{W}$ provides a useful parametrization for
  purposes of regression, see, for instance Lyons~\cite{lyons2014rough}. We
  will revisit this question in future works.
\end{remark}


\section{Stochastic representation}
\label{sec:stoch-repr}


We consider rough partial differential equations in the setting studied
\cite{DFS17,DOR15,FH14}. Given a $d$-dimensional $\alpha$-H\"{o}lder
continuous geometric rough path $\mathbf{W} = (W, \mathbb{W})$, $\f{1}{3} <
\alpha \le \half$, we consider the backward problem on $\R^n$
\begin{subequations}
  \label{eq:Cauchy-RPDE}
  \begin{gather}
    \label{eq:Cauchy-RPDE-a}
    -du = L(u) \, dt + \sum_{k=1}^d \Gamma_k(u) d\mathbf{W}^k,\\
    \label{eq:Cauchy-RPDE-b}
    u(T, \cdot) = g,
  \end{gather}
\end{subequations}
where the differential operators $L$ and $\Gamma = (\Gamma_1, \ldots,
\Gamma_d)$ are defined by
\begin{gather}
  \label{eq:def-L}
  Lf(x) = \half \trace\left( \sigma(x) \sigma(x)^\top D^2 f(x) \right) +
  \ip{b(x)}{Df(x)} + c(x) f(x),\\
  \label{eq:def-Gamma}
  \Gamma_kf(x) = \ip{\beta_k(x)}{Df(x)} + \gamma_k(x) f(x),
\end{gather}
for a suitable test function $f:\R^n \to \R$ and given functions $\sigma:\R^n
\to \R^{n \times m}$, $b:\R^n \to \R^n$, $c:\R^n \to \R$, $\beta_k : \R^n \to
\R^n$, $\gamma_k:\R^n \to \R$, $k=1, \ldots, d$. All functions are ``smooth
enough''.

A function $u = u(t,x; \mathbf{W})$ is called a ``regular''
solution\footnote{\cite{DFS17} also provide a weak notion of solution. In what
follows, the construction for both notions of solutions is the same, but weak
solutions can be established under weaker regularity conditions on the coefficients.}
to~(\ref{eq:Cauchy-RPDE}) if $u \in \mathcal{C}^{0,2}$ and
\begin{equation*}
  u(t,x) = g(x) + \int_t^T Lu(r,x) dr + \sum_{k=1}^d \int_t^T \Gamma_k u(r,x) d\mathbf{W}^k_r,
\end{equation*}
where the integral is understood in the rough path sense requiring
$\Gamma_ku, \Gamma_j\Gamma_ku$ to be controlled by $\mathbf{W}$ as functions in $t$.

Solutions to the above rough PDE in the above sense are constructed by
Feynman-Kac representations. We introduce an $m$-dimensional Brownian motion
$B$, which will essentially be used to construct a diffusion process with
generator $L$. Specifically, let
\begin{equation}
  \label{eq:SDE}
  dX_t = \sigma(X_t) dB_t + b(X_t) dt + \beta(X_t) d\mathbf{W}_t,
\end{equation}
where the $dB$-integral is understood in the It\^o sense. More precisely,
\eqref{eq:SDE} is understood as a random (via $B$) rough ordinary differential
equation with respect to a $(m+d)$-dimensional rough path $\mathbf{Z} =
\left(Z, \mathbb{Z}\right)$ defined by
\begin{equation}\label{eq:joined_rp}
  Z_t \coloneqq
  \begin{pmatrix}
    B_t\\
    W_t
  \end{pmatrix},\quad
  \mathbb{Z}_{s,t} \coloneqq
  \begin{pmatrix}
    \mathbb{B}^{\text{It\^o}}_{s,t} & \int_s^t W_{s,r} \otimes dB_r\\
    \int_s^t B_{s,r} \otimes dW_r & \mathbb{W}_{s,t}
  \end{pmatrix}.
\end{equation}

The following existence and uniqueness theorem is (part of) \cite[Theorem 2.8]{DFS17}.
\begin{theorem}
  \label{thr:existence+uniqueness+cauchy}
  Assume that the coefficients satisfy $\sigma_i, \beta_j, \gamma_k \in
  \mathcal{C}^6_b(\R^n)$, $c, g \in \mathcal{C}^4_b(\R^n)$. Define
  \begin{equation}
    \label{eq:stoch_rep}
    u(t,x; \mathbf{W}) \coloneqq E^{t,x}\left[ g(X_T) \exp\left( \int_t^T
        c(X_r) dr + \int_t^T \gamma(X_r) d\mathbf{W}_r \right) \right], \quad
    (t,x) \in [0,T] \times \R^n.
  \end{equation}
  Then $u \in \mathcal{C}^{0,4}_b\left([0,T] \times \R^n\right)$ solves the
  problem~\eqref{eq:Cauchy-RPDE} in the regular sense. The solution is unique
  among all $\mathcal{C}^{0,4}_b\left([0,T] \times \R^n\right)$ ``which are controlled
  by $\mathbf{W}$''. Moreover, if $g$ additionally has exponential decay, then
  the same is true for $u$.
\end{theorem}

\begin{remark}
  \label{rem:regularity}
  The authors of this paper are in doubt to what extent
  Theorem~\ref{thr:existence+uniqueness+cauchy} was indeed proved
  in~\cite{DFS17}, in particular with respect to regularity. Differentiability
  of $u$ in space is obtained by the corresponding differentiability of the
  solution map $x = X_0 \mapsto X_t$ of the mixed stochastic/rough
  differential equation~\eqref{eq:SDE}. Cass, Litterer and Lyons~\cite{CLL13}
  (see also \cite{FR13}) have proved the existence \emph{and integrability} of
  the first variation of RDEs like~\eqref{eq:SDE}, i.e., the first derivative,
  which extends to the statement that $u \in C^{0,1}$ in the above
  theorem. However, to the best of our knowledge, this result has not been
  extended to higher order derivatives in the literature before. We fill this
  gap in Section~\ref{sec:regularity-solution}, see
  Theorem~\ref{thm:bounds_flow_deriv} for the result on regularity of the flow
  of an RDE and Corollary~\ref{cor:main} for the extended version of
  Theorem~\ref{thr:existence+uniqueness+cauchy}
  above.  
\end{remark}

\begin{remark}
  \label{rem:gen-cauchy}
  It is possible to consider the problem (\ref{eq:Cauchy-RPDE}) for slightly
  more general operators $L$ and $\Gamma,$ by adding to $L$ and $\Gamma$ an
  autonomous term, say $h(x)$ and $\eta(x)\in\R^d,$ respectively.  This will
  result in an extended stochastic representation
\begin{equation}
  \label{eq:stoch_rep_extended0}
  E\left[  g(X_{T}^{t,x})Y_{T}^{t,x,1}+Z_{T}^{t,x,1,0}\right],  \quad t\leq T, \quad x\in \mathbb R^n,
\end{equation}
for the solution of~\eqref{eq:Cauchy-RPDE}, where 
\begin{align}
Y_{s}^{t,x,1}  & :=\exp\left(  \int_{t}^{s}c(X_{r}^{t,x})dr+\int_{t}^{s}%
\gamma(X_{r}^{t,x})dW_{r}\right)  \label{Ystdef}, \quad \text{and}\\
Z_{T}^{t,x,1,0}  & :=\int_{t}^{T}Y_{r}^{t,x,1}\left(  h(X_{r})dt+\eta^{\top
}\left(  X_{r}\right)  dW_{r}\right). \nonumber
\end{align}
\smallskip
\end{remark}

\begin{remark}
\label{sec:variance-reduction}
By defining a mean-zero process
$\widetilde{Z}_{\cdot}^{t,x,1,0}$ as the solution to 
\begin{align*}
d\widetilde{Z}_{s}  =Y_{s}F^{\top}(s,X_{s})dB_{s},\text{ \ \ }\widetilde{Z}_{t}=0
\end{align*}
for an arbitrary column vector function $F(s,y)\in\mathbb{R}^{m},$ $y\in\mathbb{R}^{n},$ and $Y_{s}$ given in (\ref{Ystdef}), we obtain another
modification of the standard stochastic
representation,~\eqref{eq:stoch_rep}, which provides a
stochastic representation with a free parameter that has smaller (point-wise) variance if this parameter is chosen accordingly.
Indeed, from Theorem \ref{thr:existence+uniqueness+cauchy} it is a
trivial observation that 
\begin{equation}
  \label{eq:stoch_rep_extended}
  E\left[  g(X_{T}^{t,x})Y_{T}^{t,x,1}+\widetilde{Z}_{T}^{t,x,1,0}\right],  \quad t\leq T, \quad x\in \mathbb R^n,
\end{equation}
is a stochastic representation to the solution of~\eqref{eq:Cauchy-RPDE}. 
In fact, via the chain rule for geometric rough paths it is possible to show that the variance of the random variable
$$
g(X_{T}^{t,x})Y_{T}^{t,x,1}+\widetilde{Z}_{T}^{t,x,1,0}
$$
vanishes if $F$ satisfies $\sigma^{\top}Du+F$ $=$ $0.$ (Cf.  Milstein and Tretyakov~\cite{mils/tret04} for this result in the standard SDE setting.) 
Of course such an ``optimal'' $F$ involves the solution of the problem itself, and as such
is not directly available.  A comprehensive study of constructing ``good'' variance reducing parameters $F$ in the present context is deferred to subsequent work. 
\smallskip

\end{remark}


\begin{remark}
  \label{rem:Dirichlet-problem}
  From a regression point of view, it might be simpler to consider the
  Dirichlet problem on a domain $D \subset \R^n$, i.e.,
  \begin{gather*}
    -du = L(u) \, dt + \sum_{k=1}^d \Gamma_k(u) d\mathbf{W}^k,\\
    u(T, x) = g(x), \ x \in D,\quad
    u(\cdot, x) = f(x), \ x \in \pa D.
  \end{gather*}
  There are a few challenges here:
  \begin{itemize}
  \item A new existence and uniqueness theorem  following the lines
    of~\cite{DFS17} is required. In particular, the
    Feynman-Kac representation in terms of stopped processes has to be derived.
  \item Numerical schemes for stopped rough differential equations have, to
    the best of our knowledge, not yet been considered.
  \end{itemize}
\end{remark}

\section{Regularity of the solution}
\label{sec:regularity-solution}

In order to understand the convergence of the regression based approximation
to $u$ as a function of the input data (including the rough path
$\mathbf{W}$), we need to control the derivative $\partial_x u(t,x)$ and
higher order derivatives explicitly in terms of the data.

We start with an $e$-dimensional weakly geometric rough path $\mathbf{Z}$ with
finite $p$-variation norm ($2 \leq p < 3$).\footnote{Note that $\mathbf{Z}$ as
  defined in~\eqref{eq:joined_rp} is \emph{not} weakly geometric. As outlined
  below, we have to transform the equation~\eqref{eq:SDE} for $X$ into
  Stratonovich form first.}  Recall that standard stability estimates for
solutions of rough differential equations driven by $\mathbf{Z}$ lead to
estimates of the form
\begin{equation*}
  \exp\left( \norm{\mathbf{Z}}_{\pvar} \, \vee \, \norm{\mathbf{Z}}_{\pvar}^p \right),
\end{equation*}
see, for instance, \cite[Theorem 10.38]{FV10}. If we replace $\mathbf{Z}$ by a
Brownian rough path, we see that terms of the above form are not integrable,
due to the $p$th power. Hence, these estimates, which are sufficient (and
sharp) in the deterministic setting, are impractical in the stochastic
setting. Cass, Litterer and Lyons~\cite{CLL13} were able to derive alternative
estimates for the first derivative of the solution flow induced by a rough differential equation, which retain integrability in (most) Gaussian contexts, cf. also \cite{FR13}. In the following section, we extend their results to higher order derivatives.

\subsection{Higher order derivatives of RDE flows}


Consider the rough differential equation
\begin{align}\label{eqn:ex_rough_DE}
	X_t^{x} = x + \int_0^t V(X^{x}_s)\, d\mathbf{Z}_s \in \R^n,
\end{align}
where $V \in \mathcal{C}(\R^n, L(\R^e, \R^n))$. Formally, the derivative
$X^{(1)} \coloneqq D_{x} X^{x}$ should solve the equation
\begin{align*}
	X^{(1)}_t = \operatorname{Id} + \int_0^t D V(X^{x}_s)(d\mathbf{Z}_s)  X^{(1)}_s\,  \in \R^{n \times n}
\end{align*}
with
\begin{align*}
 D V \colon \R^n \to L(\R^n, L(\R^e, \R^n)) \cong L(\R^e, L(\R^n, \R^n)).
\end{align*}
The higher order derivatives of the vector field $V$ are functions
\begin{align*}
	D^k V \colon \R^n \to L((\R^n)^{\otimes k}, L(\R^e, \R^n)) \cong L( \R^e, L((\R^n)^{\otimes k}, \R^n)),
\end{align*}
and the $k$-th derivative of the flow $X^{(k)} \coloneqq D^k_{x} X^{x}$ should be a function
\begin{align*}
	D^k X^{\cdot}_t \colon \R^n \to L((\R^n)^{\otimes k}, \R^n).
\end{align*} 
Taking formally the second derivative in \eqref{eqn:ex_rough_DE}, we obtain the equation
\begin{align*}
	X^{(2)}_t = \int_0^t D^2 V(X^{x}_s)( d\mathbf{Z}_s)(X^{(1)}_s \otimes X^{(1)}_s) + \int_0^t D V(X^{x}_s)(d\mathbf{Z}_s) X^{(2)}_s\, ,
\end{align*}
and for the third derivative,
\begin{align*}
 X^{(3)}_t &= \int_0^t D^3 V(X^{x}_s)( d\mathbf{Z}_s) (X^{(1)}_s \otimes X^{(1)}_s \otimes X^{(1)}_s) + 2 \int_0^t D^2 V(X^{x}_s)(X^{(1)}_s \otimes X^{(2)}_s)  \\
 &\quad +  \int_0^t D^2 V(X^{x}_s) (d\mathbf{Z}_s) (X^{(2)}_s \otimes X^{(1)}_s) + \int_0^t D V(X^{x}_s)( d\mathbf{Z}_s) X^{(3)}_s.
\end{align*}
  These formal calculations can be performed for any order $k$. The forthcoming theorem justifies these calculations. Moreover, it provides estimates for the solution which are especially useful for tail estimates when the equation is driven by a Gaussian process. For given $0 \le s < t\le T$, these estimates are based on the following
sequence of times $\tau_i$, iteratively defined by $\tau_0 = s$ and
\begin{equation*}
  \tau_{i+1} \coloneqq \inf \Set{\tau_i < u < t | \norm{\mathbf{Z}}_{\pvar;[\tau_i,u]}^p
    \ge \alpha} \wedge t,
\end{equation*}
where $\alpha$ is a positive parameter.  
Define
\begin{equation}
  \label{eq:N}
  N_{\alpha}( \mathbf{Z} ; [s,t]) \coloneqq \max\Set{n | \tau_n < t}.
\end{equation}
For $\alpha = 1$, we will omit the parameter and simply write $N$ instead of $N_1$. The important insight of~\cite{CLL13} was that $\norm{\mathbf{Z}}_{\pvar}^p$
can often be replaced by $N$ in rough path estimates, and that $N$ does have
Gaussian tails when $\mathbf{Z}$ is replaced by Gaussian processes respecting
certain regularity assumptions. But for now we remain in a purely
deterministic setting.

Next, we state the main theorem of this section. Since the proof is a bit lengthy, we decided to give it in the appendix, cf.~page \pageref{proof:flow_deriv_est}.
  
  \begin{theorem}\label{thm:bounds_flow_deriv}
   Fix $s \in [0,T]$ and let $\mathbf{Z}$ be a weakly geometric $p$-rough path for $p \in [2,3)$. Let $V \in \mathcal{C}^{2 + k}_b(\R^n, L(\R^e, \R^n))$ for some $k \geq 1$. Consider the unique solution $X^{s,x}$ to
   \begin{align}\label{eqn:RDE_main_theorem}
    X_t^{s,x} = x + \int_s^t V(X^{s,x}_u)\, d\mathbf{Z}_u \in \R^n, \quad t \in [s,T].
   \end{align}
  Then for every fixed $t \in [s,T]$, the map $x \mapsto X_{t}^{s,x}$ is $k$-times differentiable. Moreover, the $k$th derivative $X^{(k)}_t = D^k_x X^{s,x}_t$ solves a rough differential equation which is obtained by formally differentiating \eqref{eqn:RDE_main_theorem} $k$-times with respect to $x$. Setting $\omega(u,v) \coloneqq \|\mathbf{Z}\|_{p-\text{var};[u,v]}^p$, we have the bounds
  \begin{align}
    \| X^{(k)} \|_{p-\omega} \coloneqq  \| X^{(k)} \|_{p-\omega;[s,T]} &\leq C \|V\|_{\mathcal{C}^{2 + k}_b} \exp \left( C \|V\|_{\mathcal{C}^{2 + k}_b}^p (N( \mathbf{Z} ; [s,T]) + 1) \right) \quad \text{and} \label{eqn:bounds_derivatives} \\
    \| X^{(k)} \|_{\infty} \coloneqq \| X^{(k)} \|_{\infty;[s,T]} &\leq |X^{(k)}_s| + C \|\mathbf{Z}\|_{p-\text{var};[s,T]} \|V\|_{\mathcal{C}^{2 + k}_b} \exp \left( C \|V\|_{\mathcal{C}^{2 + k}_b}^p (N( \mathbf{Z} ; [s,T]) + 1) \right) \label{eqn:bounds_derivatives_sup}
  \end{align}
  where $C$ depends on $p$ and $k$.
  \end{theorem}

\subsection{Bounding higher variations of~\eqref{eq:SDE}}
\label{sec:bound-first-vari}

In order to apply Theorem \ref{thm:bounds_flow_deriv} to the rough stochastic differential equation~\eqref{eq:SDE}, we first rewrite it in Stratonovich form which leads to the following hybrid Stratonovich-rough differential equation
\begin{align}  \label{eq:rough_Stra-SDE}
  \begin{split}
  dX_t &= \left[b(X_t)-a(X_t) \right] dt +\sigma(X_t)\circ dB_r  + \beta(X_t) d\mathbf{W}_t \\
  &= \hat{b}(X_t) dt +\sigma(X_t)\circ dB_r  + \beta(X_t) d\mathbf{W}_t
  \end{split}
\end{align}
where $a(\cdot)=\frac{1}{2} \sum_{i=1}^m D\sigma_i(\cdot)\cdot\sigma_i(\cdot)$ is the It\^o-Stratonovich correction, $\sigma_i$ is the $i$th column
of $\sigma$ and $\hat{b}(\cdot):=b(\cdot)-a(\cdot)$. Equation \eqref{eq:rough_Stra-SDE} is now indeed of the form \eqref{eqn:ex_rough_DE} if we set $V = (\hat{b}, \sigma, \beta) \colon \R^n \to L(\R^{1 + m + d}, \R^n)$ and consider the joint (geometric) rough path lift $\mathbf{Z}$ of $t \mapsto (t,B_t, W_t) =: (\tilde{B}_t,W_t)$ obtained from the Stratonovich rough path lift of the Brownian motion $\tilde{B}$ and $\mathbf{W}$ (recall that $d$ is the dimension of $\mathbf{W}$, $m$ is the dimension of the Brownian motion $B$). The lift $\mathbf{Z}$ is given as in (\ref{eq:joined_rp}), but $\mathbb{B}^{\text{It\^o}}$ is substituted by the Stratonovich integral $\mathbb{B}^{\text{Str}}$, cf.~\cite{DOR15} for further details.

Applying the bounds of Theorem \ref{thm:bounds_flow_deriv} 
to the solution of equation~\eqref{eq:rough_Stra-SDE} with initial condition $X_t = x$, we see that the expected value of the
norm of the $k$th variation $D^k X$ is bounded in terms
of the moment generating function of $N(\mathbf{Z} ; [t,T])$. In~\cite{DOR15}, such bounds are provided considering the moment generating function of
$\norm{\mathbf{Z}}_{\pvar}^2$. The following is a version
of~\cite[Corollary 23]{DOR15}, which differs in two respect: first, we
consider the moment generating function of $N(\mathbf{Z} ; [t,T])$
instead of $N(\mathbf{Z} ; [t,T])^2$, and secondly we try to make the
constants explicit (instead of only providing the existence of the
exponential moment).

\begin{lemma}
  \label{lem:mgf-N}
  Given $\f{1}{\alpha} < p < 3$ and $\delta > 0$, we let
  \begin{equation*}
    \kappa_p(\delta, \mathbf{W}) \coloneqq E\left[ \exp\left( \delta
        \norm{\mathbf{Z}}_{\pvar;[t,T]}^2 \right) \right], 
  \end{equation*}
  assuming that $\delta$ is small enough such that $\kappa < \infty$. Then for
  all $\lambda > 0$ we have the bound
  \begin{equation}
    \label{eq:mgf-N-2}
    E\left[\exp\left( \lambda N(\mathbf{Z} ; [t,T]) \right) \right] \le
    \exp\left( 2^{1/p} \lambda \left[ \f{\log(2 \kappa_p(\delta,
          \mathbf{W})}{\delta} \right]^{p/2} \right) + \sqrt{2\pi} \lambda
    \sigma e^{2\lambda^2 \sigma^2},
  \end{equation}
  where $\sigma \coloneqq \sqrt{T-t}$.
\end{lemma}
\begin{proof}
  Choose $K > \sqrt{\f{\log \kappa_p(\delta, \mathbf{W})}{\delta}}$ and define
  \begin{equation*}
    r_0 \coloneqq 2^{1/p} K^p, \quad a \coloneqq 1 - \f{\kappa_p(\delta,
      \mathbf{W})}{\exp(\delta K^2)}, \quad \alpha \coloneqq \Phi^{-1}(a),
  \end{equation*}
  where $\Phi$ denotes the c.d.f.~of the standard normal distribution and we
  note that $0 < a < 1$ by our conditions. 
  The result follow from the Fernique type estimate in \cite[Theorem
  17]{DOR15}, which shows that
  \begin{equation*}
    P\left(  N(\mathbf{Z} ; [t,T]) > r \right) \le 1 - \Phi\left( \alpha
      + \f{r}{2\sigma} \right), \quad r \ge r_0.
  \end{equation*}
  (The choice of constants follows from \cite[Lemma 19, Theorem 20, Lemma 22, proof of
  Corollary 23]{DOR15}, for $q = 1$.) Using this estimate, the
  integration by parts formula
  \begin{equation*}
    E\left[\exp(\lambda N(\mathbf{Z} ; [t,T]) \right] = \int_0^\infty
    P\left( N(\mathbf{Z} ; [t,T]) > \f{1}{\lambda} \log x\right) dx
  \end{equation*}
  and the estimate $1-\Phi(x) \le \f{1}{2} e^{-x^2/2}$ (where the
  Gaussian tail estimate applies, i.e., for
  $r = \f{1}{\lambda} \log x \ge r_0$) together with the trivial estimate of
  any probability by $1$ (where the estimate does not apply, i.e., for
  $r = \f{1}{\lambda} \log x < r_0$), directly gives
  \begin{multline}
    \label{eq:mgf-N}
    E\left[\exp\left( \lambda N(\mathbf{Z} ; [t,T]) \right) \right] \le
    e^{\lambda r_0} + \sqrt{\f{\pi}{2}} \sigma \lambda \exp\left( 2 \sigma
      \lambda (\sigma\lambda - \alpha) \right) \times \cdots \\
    \cdots \times \erfc\left( \f{1}{\sqrt{2}} \left( \alpha - 2 \sigma \lambda
      + \f{r_0}{2\sigma}\right) \right).
  \end{multline}

  Note that $K^2 \ge \f{1}{\delta} \log(2 \kappa_p)$ implies $\alpha \ge 0$.
  Using the trivial bound $\erfc \le 2$, we further obtain
  \begin{equation*}
    E\left[\exp\left( \lambda  N(\mathbf{Z} ; [t,T]) \right) \right] \le
    e^{2^{1/p} \lambda K^p} + \sqrt{2 \pi} \sigma \lambda e^{2 \lambda^2 \sigma^2}.
  \end{equation*}
  The right hand side is now minimized by $K = \sqrt{\f{1}{\delta}
    \log(2 \kappa_p)}$, which gives~\eqref{eq:mgf-N-2}.
\end{proof}

Finally, we consider bounds for the derivatives of the solution $u(t,x)$
of~\eqref{eq:Cauchy-RPDE}. For ease of notation, we will \emph{formally} only
consider the case $c \equiv 0$, $\gamma \equiv 0$, i.e.,
\begin{equation*}
  u(t,x;\mathbf{W}) = E^{t,x}\left[ g(X_T) \right].
\end{equation*}
However, note that we do allow $g$ and its derivatives to have exponential growth in what follows,
and it is thus easy to incorporate the general setting by extending the state
space. For this, we just need to add an additional component $Y_t$ solving
\begin{equation*}
  dY_s = c(X_s) ds + \gamma(X_s) d\mathbf{W}_s, \quad Y_t = 0,
\end{equation*}
and consider
\begin{equation*}
  u(t,x; \mathbf{W}) = E^{t,x}\left[ g(X_T) \exp(Y_T) \right].
\end{equation*}

\begin{corollary}
\label{cor:main}
 Let $u(t,x,\mathbf{W})$ be as above. Assume that $g$ is $k$-times differentiable and that there are constants $\zeta_1$, $\zeta_2 \geq 0$ such that
 \begin{align*}
  |D^l g(x)| \leq \zeta_1 e^{\zeta_2 |x|}
 \end{align*}
 for all $x \in \R^n$ and $l = 1,\ldots, k$. Assume that $\hat{b}$, $\sigma$ and $\beta$ are bounded, $(2 + k)$-times differentiable with bounded derivatives, and let $K > 0$ be a bound for their norms, i.e.
 \begin{align*}
  \|\hat{b}\|_{\mathcal{C}_b^{2+k}} \vee \|\sigma\|_{\mathcal{C}_b^{2+k}} \vee  \|\beta \|_{\mathcal{C}_b^{2+k}} \leq K.
 \end{align*}
 Then there are constants $C = C(p,k)$ and $C_1 = C_1(p)$ such that
 \begin{align*}
  |\partial^k_x u(t,x,\mathbf{W}) | &\leq C \zeta_1 \lambda e^{\zeta_2(|x| + 1)} \left( 1 + \sqrt{(T-t)} (\zeta_2 C_1 + C) K^p \right) \\
   &\quad \times \exp \left((\zeta_2 C_1 + C) K^p \left( 1 + \left( \f{\log\left(2 \kappa_p(\delta, \mathbf{W}) \right)}{\delta} \right)^{p/2} + (T-t) (\zeta_2 C_1 + C) K^{p} \right) \right)
 \end{align*}
 where we use the same notation as in Lemma \ref{lem:mgf-N}.
\end{corollary}

\begin{proof}
  Recall that the solution $X_v^{t,x}$ to
  \begin{align*}
   dX^{t,x}_v = \hat{b}(X^{t,x}_v)\, dv + \sigma(X^{t,x}_v)\,\circ dB_v +  \beta(X^{t,x}_v)\, d \mathbf{W}_v; \quad X^{t,x}_t = x
  \end{align*}
  equals the solution to
  \begin{align*}
   dX^{t,x}_v = V(X^{t,x}_v)\, d\mathbf{Z}_v; \quad X^{t,x}_t = x
  \end{align*}
  where $V = (\hat{b}, \sigma, \beta) \colon \R^n \to L(\R^{1 + m + d}, \R^n)$ and $\mathbf{Z}$ denotes the joint geometric rough path lift of $v \mapsto (v, B_v, W_v)$.
 Iterating the chain rule, we see that
 \begin{align*}
  \partial^k_x g(X^{t,x}_v) = \sum_{l = 1}^k \sum_{i_1 + \ldots + i_l = k}  \lambda_{i_1, \ldots, i_k} (D^l g) (X^{t,x}_v) (D_x^{i_1} X^{t,x}_v \otimes \ldots \otimes D_x^{i_l} X^{t,x}_v) 
 \end{align*}
 where $1 \leq i_1, \ldots i_l \leq k$ and $\lambda_{l_1, \ldots, i_k}$ are nonnegative integers which can be calculated explicitly (using e.g. Fa\`a di Bruno's formula). Thus we obtain
 \begin{align*}
  | \partial^k_x u(t,x) | = |E( \partial^k_x g(X_T^{t,x}))| \leq \sum \lambda_{i_1, \ldots, i_l} \zeta_1 E \left( \exp({\zeta_2|X^{t,x}_T|}) |D_x^{i_1} X^{t,x}_T|\cdots | D_x^{i_l} X^{t,x}_T| \right). 
 \end{align*}
 As in the proof of Theorem \ref{thm:bounds_flow_deriv}, one can see that there is a constant $C_1$ depending only on $p$ such that
 \begin{align*}
  |X^{t,x}_T| \leq |x| + N(X^{t,x} ; [t,T]) + 1 \leq |x| + C_1 \|V\|_{\mathcal{C}_b^{2+k}}^p ( N( \mathbf{Z} ; [t,T]) + 1) + 1.
 \end{align*}
 The bounds of Theorem \ref{thm:bounds_flow_deriv} imply that
 \begin{align*}
  |D_x^{i_1} X^{t,x}_T|\cdots | D_x^{i_l} X^{t,x}_T| &\leq \left( 1 + C \|\mathbf{Z}\|_{p-\text{var};[t,T]} \|V\|_{\mathcal{C}^{2 + k}_b} \exp \left( C \|V\|_{\mathcal{C}^{2 + k}_b}^p( N( \mathbf{Z} ; [t,T]) + 1) \right) \right)^l \\
  &\leq 2^{k-1} + C \|\mathbf{Z}\|_{p-\text{var};[t,T]}^l \|V\|_{\mathcal{C}^{2 + k}_b}^l \exp \left( C \|V\|_{\mathcal{C}^{2 + k}_b}^p( N( \mathbf{Z} ; [t,T]) + 1) \right)
 \end{align*}
 for a constant $C$ depending on $p$ and $k$. Therefore, we see that there is a constant $\lambda$ depending on $k$ only such that
 \begin{align*}
  | \partial^k_x u(t,x) | 
  &\leq \zeta_1 \lambda e^{\zeta_2(|x| + 1)} E \left( \exp \left( \zeta_2 C_1 \|V\|_{\mathcal{C}^{2 + k}_b}^p( N( \mathbf{Z} ; [t,T]) + 1) \right) \right) \\
  &\quad + C \zeta_1 \lambda e^{\zeta_2(|x| + 1)} \max_{l = 1,\ldots, k} \|V\|_{\mathcal{C}^{2 + k}_b}^l E \left( \|\mathbf{Z}\|_{p-\text{var};[t,T]}^l \exp \left( (\zeta_2 C_1 + C) \|V\|_{\mathcal{C}^{2 + k}_b}^p( N( \mathbf{Z} ; [t,T]) + 1 )\right) \right).
 \end{align*}
 We use \cite[Lemma 4, Lemma 1 and Lemma 3]{FR13} to see that for every $l = 1, \ldots, k$,
 \begin{align*}
   (\|V\|_{\mathcal{C}^{2 + k}_b} \|\mathbf{Z}\|_{p-\text{var};[t,T]})^l \leq \exp( k \|V\|_{\mathcal{C}^{2 + k}_b}^p (2 N( \mathbf{Z} ; [t,T]) + 1) ).
 \end{align*}
 This implies that
 \begin{align*}
  | \partial^k_x u(t,x) | \leq  C \zeta_1 \lambda e^{\zeta_2(|x| + 1)} E \left( \exp \left( (\zeta_2 C_1 + C) \|V\|_{\mathcal{C}^{2 + k}_b}^p( N( \mathbf{Z} ; [t,T]) + 1) \right) \right).
 \end{align*}
 Now we use Lemma \ref{lem:mgf-N} to obtain the bound 
 \begin{align*}
   &\ C \zeta_1 \lambda e^{\zeta_2(|x| + 1)} E \left( \exp \left( (\zeta_2 C_1 + C) \|V\|_{\mathcal{C}^{2 + k}_b}^p( N( \mathbf{Z} ; [t,T]) + 1) \right) \right)  \\  
   &\leq C \zeta_1 \lambda e^{\zeta_2(|x| + 1) + (\zeta_2 C_1 + C) \|V\|_{\mathcal{C}^{2 + k}_b}^p}  E \left( \exp \left( (\zeta_2 C_1 + C) \|V\|_{\mathcal{C}^{2 + k}_b}^p( N( \mathbf{Z} ; [t,T]) ) \right) \right) \\
   &\leq  C \zeta_1 \lambda e^{\zeta_2(|x| + 1)} \left( 1 + \sqrt{2\pi (T-t)} (\zeta_2 C_1 + C) \|V\|_{\mathcal{C}^{2 + k}_b}^p \right) \\
   &\quad \times \exp \left((\zeta_2 C_1 + C) \|V\|_{\mathcal{C}^{2 + k}_b}^p\left( 1 + 2^{1/p} \left( \f{\log\left(2 \kappa_p(\delta, \mathbf{W}) \right)}{\delta} \right)^{p/2} + 2(T-t) (\zeta_2 C_1 + C) \|V\|_{\mathcal{C}^{2 + k}_b}^{p} \right) \right)
 \end{align*}
 and our claim follows.
\end{proof}




\section{Regression}
\label{sec:regression}

From the numerical point of view it is desirable to have a functional approximation for the solution   $u,$ i.e., to have an approximation of the form
\begin{eqnarray}
\label{eq: func_app}
u(t,x)\approx\sum_{k=0}^K a_k(t)\psi_k(x),
\end{eqnarray}
for some natural \(K>0,\) where \((\psi_k(x))\) are some simple basis functions and the coefficients \((a_k(t))\) depend only on \(t.\)
Such an approximation can be then used to perform integration, differentiation and optimization of \(u(t,x)\) in a fast way. In this section we are going to use nonparametric regression to construct approximations of the type  \eqref{eq: func_app}.
First we turn to the problem of approximating  \(u(t,x)\) for a fixed \(t>0\) and then consider   approximation of the solution \(u\) in space and time. While the first problem can  be solved using a simplified version of linear regression called pseudo-regression, for the second task  we need to use  general nonparametric regression algorithms.

\subsection{Spacial resolution obtained by regression}
\label{pseudo-res}
The representation~\eqref{eq:stoch_rep} implies,%
\begin{equation}
u(t,X_{t}^{0,x})={E}_{\mathcal{F}_{t}}\left[  g(X_{T}^{t,X_{t}^{0,x}%
})Y_{T}^{t,X_{t}^{0,x},1}\right]  ,\label{impl}%
\end{equation}
where $\left(  \mathcal{F}_{s}\right)  _{0\leq s\leq T}$ denotes the
filtration generated by $B.$ (Recall the notation introduced in
Remark~\ref{rem:gen-cauchy}.) From~\eqref{eq:stoch_rep_extended} we observe
that for $s\geq t,$
\begin{align}
Y_{s}^{t,X_{t}^{0,x},1} &  =\exp\left[  \int_{t}^{s}c\left(  X_{r}%
^{t,X_{t}^{0,x}}\right)  dr+\gamma^{\top}\left(  X_{r}^{t,X_{t}^{0,x}}\right)
d\mathbf{W}_{r}\right]  \nonumber\\
&  =\exp\left[  \int_{t}^{s}c\left(  X_{r}^{0,x}\right)  dr+\gamma^{\top
}\left(  X_{r}^{0,x}\right)  d\mathbf{W}_{r}\right]  \nonumber\\
&  =\frac{\exp\left[  \int_{0}^{s}c\left(  X_{r}^{0,x}\right)  dr+\gamma
^{\top}\left(  X_{r}^{0,x}\right)  d\mathbf{W}_{r}\right]  }{\exp\left[  \int_{0}%
^{t}c\left(  X_{r}^{0,x}\right)  dr+\gamma^{\top}\left(  X_{r}^{0,x}\right)
d\mathbf{W}_{r}\right]  }=\frac{Y_{s}^{0,x,1}}{Y_{t}^{0,x,1}}.\label{y0}%
\end{align}
Due to (\ref{y0}), (\ref{impl}) yields%
\begin{equation}
u(t,X_{t}^{0,x})={E}_{\mathcal{F}_{t}}\left[  g(X_{T}^{0,x})\frac
{Y_{T}^{0,x,1}}{Y_{t}^{0,x,1}}\right]  .\label{repu}%
\end{equation}
\bigskip

We now aim at estimating  $u(t,x)$ for a fixed $t,$ $0\leq t\leq T,$
globally in $x\in\mathbb{R}^{n},$ based on the stochastic representation
(\ref{repu}). Let us consider a random variable $\mathcal{U}$ ranging over
some domain $\mathcal{D}\subset\mathbb{R}^{n},$ with distribution $\mu.$ Given
$\mathcal{U},$ we then consider the random trajectory%
\[
\left(  X_{s}^{0,\mathcal{U}},Y_{s}^{0,\mathcal{U},1}\right)  _{0\leq s\leq
T},
\]
which is understood in the sense that the Brownian trajectory $B$ is
independent of $\mathcal{U}$. At time $s=0$ we sample i.i.d.~copies
$\mathcal{U}_{1},...,\mathcal{U}_{M}$ of $\mathcal{U}.$ We then construct a
collection of ``training paths'' $\mathcal{D}_{M}^{tr},$ consisting of
independent realizations
\begin{equation}
\mathcal{D}_{M}^{tr} \coloneqq \left\{ \left(  X_{s}^{0,\mathcal{U}_{m};m}%
,Y_{s}^{0,\mathcal{U}_{m},1;m}\right)  _{0\leq s\leq T} \mid m=1,\ldots,M \right\},\label{sa}
\end{equation}
again based on independent realizations of the Brownian motion $B$. Next
consider for a fixed time $t,$ $0\leq t\leq T,$ the vector
$\mathcal{Y}^{(t)}\in\mathbb{R}^{M},$ where%
\begin{equation}
\mathcal{Y}_{m}^{(t)} \coloneqq g\left(  X_{T}^{0,\mathcal{U}_{m};m}\right)  \frac
{Y_{T}^{0,\mathcal{U}_{m},1;m}}{Y_{t}^{0,\mathcal{U}_{m},1;m}}.\label{rany}%
\end{equation}
Now let $\psi_{1},...,\psi_{K}$ be a set of basis functions on $\mathbb{R}%
^{n}$ and define define a matrix  $\mathcal{M}^{(t)}
\in\mathbb{R}^{M\times K}$ by
\[
\mathcal{M}_{mk}^{(t)} \coloneqq \psi_{k}\left(  X_{t}^{0,\mathcal{U}_{m};m}\right),
\]
In the next step we solve the least squares problem
\begin{equation}
\label{reg}
\widehat{\gamma}^{(t)} \coloneqq \underset{\gamma\in\mathbb{R}^{K}}%
{\arg\min}\frac{1}{M}\sum_{m=1}^{M}\left(  \mathcal{Y}_{m}^{(t)}-\sum
_{k=1}^{K}\mathcal{M}_{mk}^{(t)}\gamma_{k}\right)  ^{2}
=\left(  \left(  \mathcal{M}^{(t)}\right)  ^{\top}\mathcal{M}^{(t)}\right)
^{-1}\left(  \mathcal{M}^{(t)}\right)  ^{\top}\mathcal{Y}^{(t)}.
\end{equation}
This gives an approximation%
\begin{equation}
\widehat{u}(t,x)=\widehat{u}(t,x;\mathcal{D}_{M}^{tr})\coloneqq \sum_{k=1}%
^{K}\text{\ }\widehat{\gamma}_{k}^{(t)}\psi_{k}(x)\label{fullstoch}%
\end{equation}
of $u$.
Thus, with one and the same sample (\ref{sa}) we may so get for different
times $t$ and states $x$ an approximate solution $\widehat{u}(t,x).$
Let us first consider the particular case $t=0,$ where we have%
\[
\mathcal{M}_{mk}^{(0)}:=\psi_{k}\left(  \mathcal{U}_{m}\right)  ,\text{
\ \ }\mathcal{Y}_{m}^{(0)}=g\left(  X_{T}^{0,\mathcal{U}_{m};m}\right)
Y_{T}^{0,\mathcal{U}_{m},1;m}%
\]
and then (\ref{reg}) reads%
\begin{equation}
\text{\ }\widehat{\gamma}^{(0)}=\frac{1}{M}\left(  \frac{1}{M}\left(
\mathcal{M}^{(0)}\right)  ^{\top}\mathcal{M}^{(0)}\right)  ^{-1}\left(
\mathcal{M}^{(0)}\right)  ^{\top}\mathcal{Y}^{(0)}%
.\label{eq:semistochastic-coefficients2}%
\end{equation}
Instead of the inverted random matrix in
(\ref{eq:semistochastic-coefficients2}) we may turn over to a so called
pseudo-regression estimator where the matrix entries%
\[
\left[  \frac{1}{M}\left(  \mathcal{M}^{(0)}\right)  ^{\top}\mathcal{M}%
^{(0)}\right]  _{kl}=\frac{1}{M}\sum_{m=1}^{M}\psi_{k}\left(  \mathcal{U}%
_{m}\right)  \psi_{l}\left(  \mathcal{U}_{m}\right)
\]
are replaced by their limits as $M \to \infty$, i.e., by the scalar products%
\[
\mathcal{G}_{kl}:=\langle\psi_{k},\psi_{l}\rangle:=\int_{\mathcal{D}}\psi
_{k}\left(  z\right)  \psi_{l}\left(  z\right)  \mu(dz).
\]
That is, we may also consider the estimate%
\begin{align}
\tilde{u}(0,x) &  :=\tilde{u}(0,x;\mathcal{D}_{M}^{tr}):=\sum_{k=1}%
^{K}\text{\ }\tilde{\gamma}_{k}^{(0)}\psi_{k}(x)\text{ \ \ with}%
\label{semstoch}\\
\tilde{\gamma}^{(0)} &  :=\frac{1}{M}\mathcal{G}^{-1}\left(
\mathcal{M}^{(0)}\right)  ^{\top}\mathcal{Y}^{(0)}.\label{gass}%
\end{align}
The interesting point is that in (\ref{semstoch}) we may freely choose both the initial
measure, and the set of basis functions. So by a suitable choice
of basis functions $\left(  \psi_{k}\right)  $ and initial measure $\mu,$ we
may arrange the matrix $\mathcal{G}$ to be known explicitly, or even that
$\mathcal{G}=\mathrm{Id}$ (the identity matrix), thus simplifying the regression
procedure significantly from a computational point of view. Indeed, the cost
of computing (\ref{reg}) in (\ref{fullstoch}) is of order $MK^{2}$ while the
cost of computing (\ref{gass}) is only of order $MK.$

It should be emphasized that the function estimates (\ref{fullstoch}) and
(\ref{semstoch}) are random as they depend on the simulated training paths
(\ref{sa}). In the next section we study mean-squares-estimation errors in a
suitable sense for the particular case (\ref{semstoch}), and for the general
case (\ref{fullstoch}), respectively.

\subsubsection{Error analysis}
\label{sec:spac-resol-obta} For the error analysis of the pseudo-regression
method (\ref{semstoch}) we could basically refer to Anker et
al.~\cite{Ank2017}, where pseudo regression is applied in the context of
global solutions for random PDEs. For the convenience of the reader, however,
let us here recap the analysis in condensed form, consistent with the present
context and terminology. For the formulation of the theorem and its proof
below, let us abbreviate (cf.~(\ref{rany}) and (\ref{semstoch}))%
\begin{align*}
\mathcal{V} &  \coloneqq g(X_{T}^{0,\mathcal{U}})Y_{T}^{0,\mathcal{U},1},\text{
\ \ }v(z)\coloneqq u(0,z),\text{ \ \ }\tilde{v}(z)\coloneqq\tilde{u}(0,z),\\
\mathcal{V}^{(m)} &  \coloneqq g(X_{T}^{0,\mathcal{U}^{(m)};m})Y_{T}^{0,\mathcal{U}%
^{(m)},1;m},\text{ \ \ }\mathcal{M}\coloneqq\mathcal{M}^{(0)},\text{ \ \ }%
\mathcal{Y}\coloneqq\mathcal{Y}^{(0)},\text{ \ \ }\tilde{\gamma}\coloneqq\tilde
{\gamma}^{(0)}.
\end{align*}

\begin{theorem}
\label{psth} Suppose that%
\begin{gather*}
\left\vert v(z)\right\vert \leq A\text{ \ \ and \ \ }\mathsf{Var}\left[
\mathcal{V}|\mathcal{U}=z\right]  <\sigma^{2},\text{ \ \ for all }%
z\in\mathcal{D},\\
0<\underline{\lambda_{\min}}\leq\lambda_{\min}\left(  \mathcal{G}^{K}\right)
\leq\lambda_{\max}\left(  \mathcal{G}^{K}\right)  \leq\overline{\lambda_{\max
}},\text{ \ \ for all }K=1,2,...,
\end{gather*}
where $\lambda_{\min}\left(  \mathcal{G}^{K}\right)  ,$ and $\lambda_{\max
}\left(  \mathcal{G}^{K}\right)  ,$ denote the smallest, respectively largest,
eigenvalue of the positive symmetric matrix $\mathcal{G}^K \coloneqq \left(
  \mathcal{G}_{ij} \right)_{1 \le i,j \le K}.$ Then it holds,
\begin{multline}\label{th1}
\E \int_{\mathcal{D}}\left\vert \tilde{v}(z)-v(z)\right\vert
^{2}\mu(dz) \leq\frac{\overline{\lambda_{\max}}}{\underline{\lambda_{\min}}}\left(
\sigma^{2}+A^{2}\right)  \frac{K}{M}+\\
+\underset{w \in \spn\{\psi
_{1},...,\psi_{K}\}}{\inf}\int_{\mathcal{D}}\left\vert w(z)-v(z)\right\vert
^{2}\mu(dz).
\end{multline}

\end{theorem}

\begin{proof}
Let $v^{K}$ be the projection of $v$ on to the linear span of $\psi
_{1},\ldots,\psi_{K},$ i.e.,%
\begin{equation}
v^{K}=\underset{w\,\in\,\text{span}\{\psi_{1},\ldots ,\psi_{K}\}}{\arg\inf}%
\int_{\mathcal{D}}\left\vert w(z)-v(z)\right\vert ^{2}\mu(dz).\label{pr}%
\end{equation}
Then, with $\gamma^{\circ}:=(\gamma_{1}^{\circ},...,\gamma_{K}^{\circ})^{\top
}\in\mathbb{R}^{K}$ defined by%
\begin{equation}
v^{K}=\sum_{k=1}^{K}\gamma_{k}^{\circ}\psi_{k},\label{uK}%
\end{equation}
and $\alpha\in\mathbb{R}^{K}$ defined by $\alpha_{k}:=\langle\psi_{k}%
,v\rangle,$ it follows straightforwardly by taking scalar products that%
\begin{equation}
\gamma^{\circ}=\mathcal{G}^{-1}\alpha.\label{ag}%
\end{equation}
By the rule of Pythagoras it follows that,%
\begin{equation}
\label{th11}
\E \int_{\mathcal{D}}\left\vert \tilde{v}(z)-v(z)\right\vert ^{2}%
\mu(dz)=
\E \int_{\mathcal{D}}\left\vert \tilde{v}(z)-v^{K}(z)\right\vert
^{2}\mu(dz)+\int_{\mathcal{D}}\left\vert v^{K}(z)-v(z)\right\vert ^{2}%
\mu(dz).
\end{equation}
With $\psi:=(\psi_{1},...,\psi_{K})^{\top}$ it holds by (\ref{ag}) that,%
\begin{align*}
\E\int_{\mathcal{D}}\left\vert \tilde{v}(z)-v^{K}(z)\right\vert
^{2}\mu(dz)&=\int_{\mathcal{D}}\E\left\vert \tilde{\gamma}^{\top
}\psi(z)-\gamma^{\circ\top}\psi(z)\right\vert ^{2}\mu(dz)\\
&  =\int_{\mathcal{D}}\E\left\vert \left(  \frac{1}{M}\mathcal{Y}%
^{\top}\mathcal{M}-\alpha^{\top}\right)  \mathcal{G}^{-1}\psi(z)\right\vert
^{2}\mu(dz)\\
&  =\int_{\mathcal{D}}\E\left[  \left(  \frac{1}{M}\mathcal{Y}^{\top
}\mathcal{M}-\alpha^{\top}\right)  \mathcal{G}^{-1}\psi(z)\psi^{\top
}(z)\mathcal{G}^{-1}\left(  \frac{1}{M}\mathcal{M}^{\top}\mathcal{Y}%
-\alpha\right)  \right]  \mu(dz)\\
&  =\E\left[  \left(  \frac{1}{M}\mathcal{Y}^{\top}\mathcal{M}%
-\alpha^{\top}\right)  \mathcal{G}^{-1}\left(  \frac{1}{M}\mathcal{M}^{\top
}\mathcal{Y}-\alpha\right)  \right]  ,
\end{align*}
since%
\[
\int_{\mathcal{D}}\left[  \psi(z)\psi^{\top}(z)\right]  _{kl}\mu
(dz)=\langle\psi_{k},\psi_{l}\rangle=\mathcal{G}_{kl}.
\]
We thus have that%
\begin{equation*}
0  \leq\E\int_{\mathcal{D}}\left\vert \tilde{v}(z)-v^{K}%
(z)\right\vert ^{2}\mu(dz)
  \leq\frac{1}{\underline{\lambda_{\min}}}\E\left\vert \frac{1}%
{M}\mathcal{M}^{\top}\mathcal{Y}-\alpha\right\vert ^{2}=\frac{1}%
{\underline{\lambda_{\min}}}\sum_{k=1}^{K}\mathsf{Var}\text{\thinspace}\left[
\frac{1}{M}\mathcal{M}^{\top}\mathcal{Y}\right]  _{k},
\end{equation*}
using that%
\begin{align*}
\E\left[  \frac{1}{M}\mathcal{M}^{\top}\mathcal{Y}\right]  _{k} &
=\frac{1}{M}\E\sum_{m=1}^{M}\psi_{k}(\mathcal{U}^{(m)})\mathcal{V}%
^{(m)}\\
&=\E\left(  \psi_{k}(\mathcal{U}^{(1)})\E\left[  \mathcal{V}%
^{(1)}|\mathcal{U}^{(1)}\right]  \right)  \\
&  =\langle\psi_{k},v\rangle=\alpha_{k}.
\end{align*}
Now, by observing that
\begin{align*}
\mathsf{Var}\text{\thinspace}\left[  \frac{1}{M}\mathcal{M}^{\top}%
\mathcal{Y}\right]  _{k} &  =\mathsf{Var}\text{\thinspace}\left(  \frac{1}%
{M}\sum_{m=1}^{M}\psi_{k}(\mathcal{U}^{(m)})\mathcal{V}^{(m)}\right)  \\
&  =\frac{1}{M}\mathsf{Var}\text{\thinspace}\left(  \psi_{k}(\mathcal{U}%
^{(1)})\mathcal{V}^{(1)}\right)  \\
&  =\frac{1}{M}\E\text{\thinspace}\mathsf{Var}\left[  \psi
_{k}(\mathcal{U}^{(1)})\mathcal{V}^{(1)}|\mathcal{U}^{(1)}\right]  +\frac
{1}{M}\mathsf{Var}\text{\thinspace}\E\left[  \psi_{k}(\mathcal{U}%
^{(1)})\mathcal{V}^{(1)}|\mathcal{U}^{(1)}\right]  \\
&  =\frac{1}{M}\E\text{\thinspace}\left(  \psi_{k}^{2}(\mathcal{U}%
^{(1)})\mathsf{Var}\left[  \mathcal{V}^{(1)}|\mathcal{U}^{(1)}\right]
\right)  +\frac{1}{M}\mathsf{Var}\text{\thinspace}\psi_{k}(\mathcal{U}%
^{(1)})v\left(  \mathcal{U}^{(1)}\right)  \\
&  \leq\frac{\sigma^{2}+A^{2}}{M}\mathcal{G}_{kk}^{K},
\end{align*}
one has%
\[
\frac{1}{\underline{\lambda_{\min}}}\sum_{k=1}^{K}\mathsf{Var}\text{\thinspace
}\left[  \frac{1}{M}\mathcal{M}^{\top}\mathcal{Y}\right]  _{k}\leq\frac
{\sigma^{2}+A^{2}}{M\underline{\lambda_{\min}}}\text{tr}\left(  \mathcal{G}%
^{K}\right)  \leq\frac{\sigma^{2}+A^{2}}{M\underline{\lambda_{\min}}%
}K\overline{\lambda_{\max}},
\]
and then (\ref{th1}) follows.
\end{proof}
\subsection{Spatio-temporal  resolution obtained by regression}
\label{seq:temp-spat}
If we want to approximate \(u(t,x)\) in space and time, we can perform regression on a given set of trajectories for different time points \(t\). Let us fix a time grid \((t_1,\ldots,t_L)\) with \(0<t_1<t_2<\ldots<t_L<T\) and consider regression problems
\begin{align}
\text{\ }\widehat{\gamma}^{(l)} & \coloneqq \underset{\gamma\in\mathbb{R}^{K}}%
{\arg\min}\frac{1}{M}\sum_{m=1}^{M}\left(  \mathcal{Y}_{m}^{(l)}-\sum
_{k=1}^{K}\mathcal{M}_{mk}^{(l)}\gamma_{k}\right)  ^{2}\label{reg_time}\\
&  =\left(  \left(  \mathcal{M}^{(l)}\right)  ^{\top}\mathcal{M}^{(l)}\right)
^{-1}\left(  \mathcal{M}^{(l)}\right)  ^{\top}\mathcal{Y}^{(l)},\quad l=1,\ldots, L,\nonumber
\end{align}
where
\[
\mathcal{M}_{mk}^{(l)} \coloneqq \psi_{k}\left(  X_{t_l}^{0,\mathcal{U}_{m};m}\right)
\]
and
\begin{equation*}
\mathcal{Y}_{m}^{(l)} \coloneqq g\left(  X_{T}^{0,\mathcal{U}_{m};m}\right)  \frac
{Y_{T}^{0,\mathcal{U}_{m},1;m}}{Y_{t_l}^{0,\mathcal{U}_{m},1;m}}, \quad l=1,\ldots, L.%
\end{equation*}
This would give us a decomposition
\begin{equation*}
\widehat{u}(t_l,x)=\widehat{u}(t_l,x;\mathcal{D}_{M}^{tr})\coloneqq \sum_{k=1}%
^{K}\text{\ }\widehat{\gamma}_{k}^{(l)}\psi_{k}(x), \quad l=1,\ldots, L.
\end{equation*}
of $u$. Furthermore, the coefficients \(\widehat{\gamma}_{k}^{(l)}\) can be interpolated to provide us with the approximation of the form \eqref{eq: func_app}.
The convergence analysis of   the estimates (\ref{reg_time}) is more involved and follows from the general theory of nonparametric regression, see Section 11 in \cite{gyorfi2002distribution}  Assume that
\begin{enumerate}
\item[(A1) ] $\,\max_{l=1,\ldots,L}\sup_{z\in\mathbb{R}^{n}}\mathsf{Var}\left[ \left. g(X_{T}%
^{0,\mathcal{U}})\frac{Y_{T}^{0,\mathcal{U},1}}{Y_{t_l}^{0,\mathcal{U},1}}\right |
X_{t_l}^{0,\mathcal{U}}=z\right]  \leq\sigma^{2}<\infty$,

\item[(A2) ] $\,\max_{l=1,\ldots,L}\sup_{x\in\mathbb{R}^{n}}|u(t_l,x)|\leq A<\infty,$
\end{enumerate}
for some positive constants $\sigma$ and $A.$ Then we denote by
$\overline{u}$ a truncated regression estimate, which is defined as follows:
\[
\overline{u}(t,x)\doteq T_{A}\widehat{u}(t,x)\doteq%
\begin{cases}
\widehat{u}(t,x) & \text{if }|\widehat{u}(t,x)|\leq A,\\
A\operatorname{sgn}\widehat{u}(t,x) & \text{otherwise.}%
\end{cases}
\]
Under (A1)--(A2) we have the following $L^{2}$-upper bound (see Theorem~11.3 in
\cite{gyorfi2002distribution})
\begin{eqnarray}
\label{eq:upper_bound_regr}
E\Vert\overline{u}(t_l,\cdot)-u(t_l,\cdot)\Vert_{L^{2}(\mathrm{P}%
_{X_{t_l}})}^{2}&\leq&\tilde{c}\left(  \sigma^{2}+A^{2}(\log M+1)\right)
\frac{K}{M}
\\
\nonumber
&&+8\inf_{f\in\text{span}\{\psi_{1},...,\psi_{K}\}}\Vert
u(t_l,\cdot)-f(\cdot)\Vert_{L^{2}(\mathrm{P}_{X_{t_l}})}^{2},
\end{eqnarray}
for all \(l=1,\ldots,L,\) where $\tilde{c}>0$ is a universal constant. Note that the use of the measure \(\mathrm{P}%
_{X_{t_l}}\) in \eqref{eq:upper_bound_regr} is essential and \(\mathrm{P}%
_{X_{t_l}}\) can not be in general replaced by an arbitrary measure \(\mu\) as in the case of pseudo-regression algorithm.
\par
Instead of linear regression, we could use a nonlinear one. Let us fix a nonlinear class  of functions \(\Psi_M\) and define
\begin{eqnarray*}
\widehat{u}(t_l,x)=\argmin_{\psi\in \Psi_M}\frac{1}{M}\sum_{m=1}^{M}\left(  \mathcal{Y}_{m}^{(t_l)}-\psi\left(  X_{t_l}^{0,\mathcal{U}_{m};m}\right)\right)  ^{2}.
\end{eqnarray*}
Under a stronger assumption that \(|\mathcal{Y}^{(t_l)}|\leq A\) with probability \(1\) for all \(l=1,\ldots,L\) and a constant \(A>0,\) we get (see Theorem~11.5 in \cite{gyorfi2002distribution})
\begin{eqnarray}
\label{eq:upper_bound_regr_nl}
E\Vert\overline{u}(t_l,\cdot)-u(t_l,\cdot)\Vert_{L^{2}(\mathrm{P}%
_{X_{t_l}})}^{2}&\leq&\left(  c_1+c_2\log M\right)
\frac{V_{\Psi_M}}{M}
\\
\nonumber
&&+2\inf_{f\in \Psi_M}\Vert
u(t_l,\cdot)-f(\cdot)\Vert_{L^{2}(\mathrm{P}_{X_{t_l}})}^{2},
\end{eqnarray}
for all \(l=1,\ldots,L,\) where the constants \(c_1,\) \(c_2\) depend on \(A_t,\)  \(V_{\Psi_M}\) is the Vapnik-Chervonenkis dimension of \(\Psi_M\) and \(\overline{u}\) is a truncated version of \(\widehat{u}.\)  The advantage of using nonlinear classes consists in  their ability to significantly reduce the approximation errors \(\inf_{f\in \Psi_M}\Vert
u(t_l,\cdot)-f(\cdot)\Vert_{L^{2}(\mathrm{P}_{X_{t_l}})}^{2},\) while keeping the
complexity  \(V_{\Psi_M}\) comparable to the linear classes. One popular
choice of \(\Psi_M\) is neural networks.

\subsection{Rates of convergence}

\label{seq:app_err}
\label{regr:piece_poly} There are several ways to choose the basis
functions $\psi_{1},\ldots,\psi_{K}$. In this section we consider the
so-called piecewise polynomial partitioning estimates and present $L^{2}%
$-upper bounds for the corresponding projection errors
\begin{equation}
\inf_{f\in\text{span}\{\psi_{1},...,\psi_{K}\}}\Vert u(t,\cdot)-f(\cdot
)\Vert_{L^{2}(\varrho)}^{2}=:\Vert u(t,\cdot)-\underline{u}(t,\cdot
)\Vert_{L^{2}(\varrho)}^{2}, \label{ppp}%
\end{equation}
for some fixed $t\geq0$ and some generic measure $\varrho$ on \(\mathbb{R}^n.\) For instance, in
\eqref{eq:upper_bound_regr} $t$ and $\varrho$ may taken to be $t_{l}$ and
$\mathrm{P}_{X_{t_{l}}},$ $l=1,...,L,$ respectively, and in \eqref{th1} we may
take $t=0$ and $\varrho$ equal to $\mu.$ The piecewise polynomial partitioning
estimate of $u$ works as follows: We fix some $q\in\mathbb{N}$ that denotes
the maximal degree of polynomials involved in our basis functions. Next fix
some $R>0$ and a uniform partition of $\left[  -R,R\right]  ^{n}$ into
$S^{n}$ cubes $C_{1},\ldots,C_{S^{n}}.$ That is, $[-R,R]$ is partitioned into $S$ subintervals with equal length. Further, consider the set of basis
functions $\psi_{j,1},\ldots,\psi_{j,c_{q,n}}$ with $j\in\left\{
1,\ldots,S^{n}\right\}  $ and $c_{q,n}:=\binom{q+n}{n}$ such that $\psi
_{j,1}(x),\ldots,\psi_{j,c_{q,n}}(x)$ are polynomials with degree less than or
equal to $q$ for $x\in C_{j},$ and $\psi_{j,1}(x)=\ldots=\psi_{j,c_{q,n}%
}(x)=0$ for $x\notin C_{j}$. Then we consider the least squares projection
estimate $\underline{u}(t,x)$ for $x\in\mathbb{R}^{n}$, based on
$K=S^{n}c_{q,n}=O(S^{n}q^{n})$ basis functions. Let us define the operator
$D^{\alpha}$ as
\[
D^{\alpha}f(x):=\frac{\partial^{\left\vert \alpha\right\vert }f(x)}{\partial
x_{1}^{\alpha_{1}}\cdots\partial x_{n}^{\alpha_{n}}},
\]
for any real-valued function $f$, $\alpha\in\mathbb{N}_{0}^{n}$ and
$\left\vert \alpha\right\vert =\alpha_{1}+\ldots+\alpha_{n}.$ For
$r\in\mathbb{N}_{0}$ and $L_f:$ \(\mathbb{R}^n\to\mathbb{R}_{+}\) we say that a function $f\colon$
$\mathbb{R}^{n}\rightarrow\mathbb{R}$ is \emph{$(r+1,L_f)$-smooth
w.r.t.\ the (Euclidian) norm }$|\cdot|$ whenever, for all $\alpha$
with $\left\vert \alpha\right\vert =\sum_{i=1}^{n}\alpha_{i}=r$ and all \(R>0\), we have
\[
\left\vert D^{\alpha}f(x)-D^{\alpha}f(y)\right\vert \leq L_f(x)|x-y|,\quad
x\in \mathbb{R}^{n},\quad |y-x|_{\infty}\leq 1,
\]
i.e.,\ the function
$D^{\alpha}f$ is locally Lipschitz with the Lipschitz function $L_f$ with
respect to the norm $|\cdot|$ on $\mathbb{R}^{n}.$ Let us make the following assumptions.
\begin{enumerate}
\item[(A3) ] The function $\,u(t,\cdot)$ is $(q+1,L_u)$-smooth with
\[
\int_{\mathbb{R}^n} L^2_u(x) \,\varrho(dx)\leq C^2_u<\infty
\]
for some constant \(C_u>0.\)
\item[(A4) ] It holds
\[
\int_{\{|z|_{\infty}>R\}} u^2(t,z)\,\varrho(dz)\leq B_{\nu}R^{-\nu}
\]
for some \(\nu>0\) all $R>0.$
\end{enumerate}
The following result holds.
\begin{lemma}
\label{regr:th:2104a1}
Suppose that (A3) and (A4) hold, then
\begin{eqnarray}
\label{regr:eq:2104a2}
\Vert u(t,\cdot)-\underline{u}(t,\cdot)\Vert_{L^{2}(\varrho)}^{2}\lesssim \frac{C_u^2}{[(q+1)!]^{2}n}\left(  \frac{Rn}{S}\right)
^{2(q+1)}+B_{\nu}R^{-\nu},
\end{eqnarray}
where \(\lesssim\) stands for inequality up to an absolute constant.
\end{lemma}
\begin{remark}
\label{regr:rem_poly} Notice that the terms on the right-hand-side
of~\eqref{regr:eq:2104a2} are of order
\begin{equation}
\left(  \frac{R}{S}\right)  ^{2(q+1)}+R^{-\nu}, \label{regr_order}%
\end{equation}
provided that we only track $R$ and $S$ and ignore the remaining parameters,
such as $q$ and $\kappa_{p}(\delta,\mathbf{W})$. Let us assume that both terms
in~\eqref{regr_order} are of the same order. Then we get $R=O(S^{\frac
{2(q+1)}{\nu+2(q+1)}})$ and thus $R^{-\nu}=O(S^{-{\frac{2\nu(q+1)}{\nu
+2(q+1)}}})$. Together with the fact that the overall number of basis
functions $K$ is of order $S^{n}$, we have $R^{-\nu}=O(K^{-{\frac{2\nu
(q+1)}{n(\nu+2(q+1))}}})$. Thus there is a constant $D>0$ such that
\[
\Vert u(t,\cdot)-\underline{u}(t,\cdot)\Vert_{L^{2}(\varrho)}^{2}\leq\frac
{D}{K^{\kappa}}.
\]
with $\kappa=\frac{2\nu(q+1)}{n(\nu+2(q+1))}.$
\end{remark}
The following result is based on Corollary~\ref{cor:main} and gives sufficient conditions for (A3) and (A4) to hold.
\begin{corollary}
\label{cor:suff}
Let $u(t,x,\mathbf{W})$ be as above. Assume that $g$ is $q+1$-times
differentiable (in $x$) and that there are constants $\zeta_{1},\zeta_2\geq0$ such
that
\begin{eqnarray}
\label{eq:Dlg}
|D^{l}g(x)|\leq\zeta_1 e^{\zeta_2 |x|}
\end{eqnarray}
for all $x\in \mathbb{R}^n$ and $l=1,\ldots,q+1$. Assume that $\sigma$ is
bounded, $(4+q)$-times differentiable with bounded derivatives, $b$ and
$\beta$ are bounded, $(3+q)$-times differentiable with bounded derivatives,
and let $K_{1}>0$ be a bound for their norms, i.e.
\begin{eqnarray}
\label{eq:bounds_coeff}
\Vert\sigma\Vert_{\mathcal{C}_{b}^{3+q}}\vee\Vert \hat{b} \Vert_{\mathcal{C}%
_{b}^{3+q}}\vee\Vert\beta\Vert_{\mathcal{C}_{b}^{3+q}}\leq K_{1}
\end{eqnarray}
with $\hat{b}$ denoting the Stratonovich corrected drift as given in~\eqref{eq:rough_Stra-SDE}.
Suppose that
\begin{eqnarray*}
\int e^{2\zeta_2|x|}\varrho(dx)< \infty,
\end{eqnarray*}
then  (A3) holds with
\[
C_u\leq D_{1}\exp\left( \left(
\f{\log\left(2 \kappa_p(\delta, \mathbf{W}) \right)}{\delta}\right)
^{p/2} \right)+D_{2}.
\]
for some constants $D_{1}=D_{1}(q,K_{1},\zeta_{1},\zeta_{2})$ and $D_{2}%
=D_{2}(q,K_{1},\zeta_{1},\zeta_{2}).$ Moreover, (A4) holds for
 some \(\nu>0\) and \(B_\nu\) depending on \(K_{1},T,\zeta_{1},\zeta_{2}.\)
\end{corollary}

Using to the parameter allocations in Remark~\ref{regr:rem_poly} we end up
with the following convergence rates for the regression procedures proposed in
Section~\ref{pseudo-res} and Section~\ref{seq:temp-spat}, respectively.

\begin{corollary}
Suppose that the conditions \eqref{eq:Dlg} and \eqref{eq:bounds_coeff} are satisfied. Moreover assume that
\begin{eqnarray*}
\int e^{2\zeta_2|x|}\mu(dx)< \infty,
\end{eqnarray*}
then
under assumptions of Theorem 4.1,  the latter reads,
\[
E\int_{\mathcal{D}}\left\vert \tilde{v}(z)-v(z)\right\vert ^{2}\mu(dz)\leq
D_{3}\frac{K}{M}+\frac{D_4}{K^{\kappa}},
\]
for some constants $D_3,D_{4}>0.$
\end{corollary}
\begin{corollary}
Suppose that the conditions \eqref{eq:Dlg} and \eqref{eq:bounds_coeff} are satisfied. Moreover assume that
\begin{eqnarray*}
\int e^{2\zeta_2|x|}\mathrm{P}_{X_{t_{l}}}(dx)< \infty, \quad l=1,\ldots,L,
\end{eqnarray*}
then under assumptions (A1) and (A2)
\[
E\Vert\overline{u}(t_{l},\cdot)-u(t_{l},\cdot)\Vert_{L^{2}(\mathrm{P}%
_{X_{t_{l}}})}^{2}\leq D_{5}(\log M+1)\frac{K}{M}+\frac{D_{6}}{K^{\kappa}},
\]
for some constants $D_{5},D_{6}>0.$
\end{corollary}


\section{Simplified Euler scheme for rough differential equations}
\label{sec:Euler}

For the computation of the optimal coefficients $\gamma$ in (\ref{reg}) and (\ref{gass}) it is required to construct the vector $\mathcal Y$ with
components $\mathcal{Y}_m$ defined in (\ref{rany}) that depends on paths of the solution to equation (\ref{eq:SDE}). For that reason, we introduce an Euler scheme
which allows us to numerically solve (\ref{eq:SDE}). \smallskip

As in Section~\ref{sec:regularity-solution}, we consider the
hybrid Stratonovich-rough differential equation with $0\leq t\leq r \leq
T$: \begin{equation}
  \label{eq:Stra-SDE}
  dX_r^{t,x} = \left[b(X_r^{t,x})-a(X_r^{t,x}) \right] dt +\sigma(X_r^{t,x})\circ dB_r  + \beta(X_r^{t,x}) d\mathbf{W}_r,\;\;\;X_t^{t,x}=x,
\end{equation}
where $a(\cdot)=\frac{1}{2} \sum_{i=1}^m D\sigma_i(\cdot)\cdot\sigma_i(\cdot)$ is the It\^o-Stratonovic correction and $\sigma_i$ is the $i$th column
of $\sigma$.
Again, the above hybrid equation is defined as an RDE driven by the joint rough path of $B$ and $\mathbf{W}$. This geometric joined rough path
$\mathbf{Z}^{\text{g}}$ is given as in (\ref{eq:joined_rp}) but $\mathbb{B}^{\text{It\^o}}$ is substituted by the Stratonovich integral
$\mathbb{B}^{\text{Str}}$. Below, we set
$d(\cdot):=b(\cdot)-a(\cdot)$ and $V(\cdot):=\begin{bmatrix}\sigma(\cdot)& \beta(\cdot)\end{bmatrix}$ for the simplicity of the notation.\smallskip

Firt of all, let $t=r_1<r_2<\ldots< r_{\bar n}=T$ be an equidistant time grid with step size $h$. In the
numerical experiments later on the path $W$ will be specified as a trajectory of a fractional Brownian motion with Hurst index $\frac{1}{3}< H
\leq\frac{1}{2}$. For this situation the following scheme provides a meaningful approximation $\bar X_{r_k}$ of $X_{r_k}^{t,x}$:
\begin{align}\label{eq:simple_euler_scheme}
    {\bar X}_{r_{k+1}}= {\bar X}_{r_{k}}+d({\bar X}_{r_{k}})h+ \sum_{\ell=1}^3 \frac{1}{\ell!} V_{i_1}\dots V_{i_\ell} I({\bar
 X}_{r_{k}})\Delta_k  Z^{i_1}  \dots \Delta_k Z^{i_\ell},\end{align}
 where $V_i$ is the $i$th column of $V$, $I(x)=x$, $V_i V_j(x)=DV_j(x)\; V_i(x)$ and $\Delta_k Z^{i}={Z}^i_{r_{k+1}}-{Z}^i_{r_{k}}$. Notice that we
use Einstein's summation convention in (\ref{eq:simple_euler_scheme}) which we indicate by the upper indices for the components of $Z$.\smallskip

This simplified Euler scheme was first introduced in~\cite{DNT12} and also investigated in~\cite{BFRS16}. In the following, we state a result
from~\cite{BFRS16} on the strong order of convergence to (\ref{eq:simple_euler_scheme}).
\begin{theorem}\label{thr:Euler_rate}
Let $W$ be a $d$-dimensional, continuous, centered Gaussian process with independent components. Moreover, we assume that each component $W^i$, $i\in\left\{1, \ldots, d\right\}$, has 
stationary increments with a concave variance function\begin{align*}
                     \sigma_i^2(\tau):= E\left\vert W^i_{t+\tau}-W^i_t\right\vert^2,\quad t,\tau \geq 0,                                  
                                                      \end{align*}
where $\sigma_i^2(\tau)=\mathcal O\left(\tau^{\frac{1}{\rho}}\right)$ as $\tau\rightarrow 0$ for some $\rho\in[1, 2)$. 
Let $X$ be the solution to (\ref{eq:Stra-SDE}) and $\bar X$ be its approximation based on (\ref{eq:simple_euler_scheme}), where $Z_t=\left(\begin{smallmatrix} W_t(\omega)\\B_t \end{smallmatrix}\right)$
for fixed $\omega\in\Omega$. Then, for almost all paths of $W$ and for any $1\leq p<\infty$, there is a constant $\tilde C$ such that  \begin{align*}
  \left\vert E\max_{k=1,\ldots,\bar n}\left\vert X_{r_k}-\bar X_{r_k}\right\vert^p\right\vert^{\frac{1}{p}}\leq \tilde C h^{\frac{1}{\rho}-0.5-\delta},
                                                              \end{align*}
where $h$ is the time step of the Euler method and $\delta>0$ is arbitrary small.
\begin{proof}
This theorem is a consequence of ~\cite[Theorem 1]{BFRS16} together with \cite[Theorem 1.1]{DNT12}.
\end{proof}
                                                              \end{theorem}
\begin{remark}
\leavevmode
\begin{itemize}
 \item Theorem \ref{thr:Euler_rate} covers the case of $W$ being a fractional Brownian motion with Hurst index $\frac{1}{4}<H\leq \frac{1}{2}$ ($\frac{1}{\rho}= 2H$).
 \item An almost sure rate for the scheme in (\ref{eq:simple_euler_scheme}) is proved in~\cite[Theorem 1.1]{DNT12} in case $Z$ is a fractional Brownian motion.
\end{itemize}

\end{remark}

\section{Numerical examples}
\label{sec:numerical-examples}

We illustrate the methods by some numerical examples. First we study examples involving linear vector fields, for which the rough differential equation has an explicit
solution. This allows for easy comparison with a reliable reference solution. Later on, we consider an example with non-linear vector fields without readily available
reference values. All examples take place in a two- or three-dimensional state 
space, and we assume that the driving Brownian motion is one-dimensional (i.e., the PDE fails to be elliptic), whereas the rough driver is 
two-dimensional in order to rule out trivial cases.

\subsection{Numerical examples with linear vector fields}\label{sec:linear-example}

Let us investigate a particular example for the RPDE (\ref{eq:Cauchy-RPDE}). We set $c, \gamma\equiv 0$ such that by Theorem 
\ref{thr:existence+uniqueness+cauchy}
the corresponding regular solution is simply represented by \begin{equation}\label{eq:special_feynman_kac}
    u(t,x; \mathbf{W}) = E\left[ g(X_T^{t,x})\right], \quad   (t,x) \in [0,T] \times \R^n,
  \end{equation}
where $X_{\cdot}^{t,x}$ is the solution to (\ref{eq:Stra-SDE}) with initial time $t$ and initial value $x$. Below, we from now on assume that
\begin{equation}\label{lin_coef}
   b(x)=Ax,\;\;\;\sigma_i(x)=C_i x\;\;\;\text{and}\;\;\;  \beta_j(x)=N_j x,    
\end{equation}
for $i=1, \ldots, m$, $j=1, \ldots, d$, $x\in\R^n$ and where all coefficients $A$, $C_i$, $N_j$ are $n\times n$ matrices.\smallskip

\subsubsection{Explicit solutions to linear RDEs}

We can find an explicit representation for the resulting linear RDE (compare
\ref{eq:Stra-SDE}) by introducing the fundamental solution $\Phi$ to the
linear system. Using the Einstein convention, we formally define $\Phi$ as the
$\R^{n\times n}$-valued process satisfying
\begin{equation}\label{eq:fundamental_sol}
 \Phi_r = I+\int_0^r \left(A-\frac{1}{2}\sum_{i=1}^m C_i^2\right) \Phi_s ds +\int_0^r C_i \Phi_s\circ dB^i_s  + \int_0^r N_j\Phi_s d\mathbf{W}^j_s.
\end{equation}
For $t\leq r$ we can easily see that the following identity holds:
\begin{equation*}
 \Phi_r \Phi_t^{-1} = I+\int_t^r \left(A-\frac{1}{2}\sum_{i=1}^m C_i^2\right) \Phi_s\Phi_t^{-1} ds +\int_t^r C_i \Phi_s\Phi_t^{-1}\circ dB^i_s 
 + \int_t^r N_j\Phi_s\Phi_t^{-1} d\mathbf{W}^j_s.
\end{equation*}
 Consequently, equation (\ref{eq:Stra-SDE}) with the linear coefficients (\ref{lin_coef}) is represented as 
 \begin{equation}\label{eq:expressolwithfund}
  X_r^{t,x} = \Phi_r \Phi_t^{-1} x,\;\;\; 0\leq t\leq r \leq T.
\end{equation}
\paragraph{\textit{Case of commuting matrices}}
We now point out a case, in which $\Phi$ is given explicitly. Let all matrices $A$, $C_i$ and $N_j$ commute, then we have 
\begin{equation}\label{eq:pihcommute}
     \Phi_r=f(r, B^i_r, W^j_r):=\exp\left((A-\frac{1}{2}\sum_{i=1}^m C_i^2)r+C_i B^i_r+ N_jW^j_r\right).      
 \end{equation}
 Using the classical chain rule for geometric rough paths
\begin{equation*}
     df(r, B^i_r, W^j_r)=\frac{\partial}{\partial t} f(r, B^i_r, W^j_r) dr + \frac{\partial}{\partial b_i} f(r, B^i_r, W^j_r) \circ dB^i_r
     + \frac{\partial}{\partial w_j} f(r, B^i_r, W^j_r) d\mathbf{W}^j_r,
 \end{equation*}
 we indeed see that $f$ solves (\ref{eq:fundamental_sol}) taking into account that 
 \begin{equation}
  \exp\left(\sum_{i=1}^q A_i \right)=\prod_{i=1}^q \exp\left(A_i\right)\;\;\;\text{and}\;\;\;A_j\exp\left(A_i\right)=\exp\left(A_i\right)A_j
 \end{equation}
for commuting matrices $A_1, \ldots, A_q$.

\paragraph{\textit{Case of nilpotent matrices}}

We know from the above considerations that the fundamental matrix $\Phi$ is given by (\ref{eq:pihcommute}) if all matrices commute, i.e., the rough 
path structure does not enter the solution at all. For that reason, we investigate another case with an explicit solution. Let us again 
look at the linear RDE which is of the form:
\begin{equation}
  \label{eq:linear-rde}
  dX_r = \sum_{i=1}^{d_1} A_i X_r d\mathbf{Z}^{g, i}_r, \quad r \in [t,T],
\end{equation}
where $d_1=m+d$, $\mathbf{Z}^g$ is the geometric joint rough path of $B$ and $W$, $A_i=C_i$, $A_{j+m}=N_j$ for $i=1, \ldots, m$ and $j=1, \ldots, 
d$. For simplicity, we assume to have a zero drift, i.e., $A-\frac{1}{2}\sum_{i=1}^m C_i^2=0$.\smallskip

The Chen-Strichartz formula, see~\cite{S87}, 
provides a general solution formula in terms of a infinite series in the general case, involving higher order iterated integrals of the driving rough 
path.\footnote{The Chen-Strichartz formula is usually given for the smooth case, but one can repeat the proof for the rough case, see, for 
instance,~\cite{B04} for the Brownian case in a free setting.} For simplicity, we shall only provide the solution in the step-2 nilpotent case, i.e., 
we assume that
\begin{equation}\label{eq:2nilpotent}
  \forall i, j, l:\ [[A_i,A_j],A_l] = 0,
\end{equation}
where $[A,B] \coloneqq AB - BA$ denotes the usual commutator of matrices.
\begin{lemma}
  \label{lem:linear-rde-solution-nilpotent}
  For $1 \le i \neq j \le d_1$ let
  \begin{equation*}
    a^{ij}_{t,r} \coloneqq \half \int_t^r Z^i_{t,s} dZ^j_s - \half \int_t^r
    Z^j_{t,s} dZ^i_s
  \end{equation*}
  denote the \emph{area} swapped by the paths $Z^i$ and $Z^j$, where the
  integrals are, of course, understood in the sense of the rough path
  $\mathbf{Z}^g$. Then, we have
  \begin{equation*}
    X_r^{t, x} = \exp\left( \sum_{i=1}^{d_1} A_i Z^i_{t,r} - \sum_{1 \le i < j \le d_1}
      [A_i,A_j] a^{ij}_{t,r} \right) x.
  \end{equation*}
\end{lemma}
\begin{remark}
  The unusual minus sign in Lemma~\ref{lem:linear-rde-solution-nilpotent}
  comes from the fact that the linear vector field $y \mapsto -[A,B]y$ is the
  Lie bracket of the linear vector fields $y \mapsto Ay$ and $y \mapsto
  By$. In the more general formulation involving general vector fields, the
  minus sign above, therefore, turns into a plus sign.
\end{remark}
\begin{proof}[Sketch of proof of Lemma~\ref{lem:linear-rde-solution-nilpotent}]
  Formally, suppose that the paths $t \mapsto Z^i_t$ are actually smooth, so
  that~\eqref{eq:linear-rde} can be replaced by the non-autonomous ODE
  \begin{equation*}
    \dot{X}_t = A(t) X_t, \quad A(t) \coloneqq \sum_{i=1}^{d_1} A_i \dot{Z}^i_t.
  \end{equation*}
  (Here, $A(t)$ is considered a time dependent vector field.)
  The Chen-Strichartz formula (also known as ``generalized
  Baker-Campbell-Hausdorff-Dynkin formula'')~\cite[formula (G.C-B-H-D)]{S87}
  involves $n$-fold Lie brackets of the vector fields $A(s_j)$ for different
  times $s_j$, $j=1, \ldots, n$. Note that
  \begin{equation*}
    [A(s_1), A(s_2)](x) = - \sum_{1 \le i < j \le d_1} \left( \dot{Z}^i_{s_1}
      \dot{Z}^j_{s_2} - \dot{Z}^j_{s_1} \dot{Z}^i_{s_2} \right)[A_i,A_j] \cdot x, 
  \end{equation*}
  while all Lie brackets of terms involving two or more Lie brackets
  vanish. The result is then obtained by inserting into the formula.
\end{proof}

\subsubsection{Numerical example with commuting matrices}\label{sec:experiment_commute}

For the following numerical considerations we now assume that $m=1$, $n=d=2$, $T=1$ and 
$g(x)=\exp\left(-0.5 \left\|x\right\|^2\right)$.\medskip

We specify the matrices in (\ref{lin_coef}) of the linear system. We introduce a matrix $V$ which satisfies the property $V^{-1}=V$: 
\begin{equation*}
 V= \begin{pmatrix}
    -\frac{1}{\sqrt{2}}&\frac{1}{\sqrt{2}}\\
    \frac{1}{\sqrt{2}}&\frac{1}{\sqrt{2}}
  \end{pmatrix}.
  \end{equation*}
Using $V$, we then set
 \begin{align*}
 A&=V \begin{pmatrix}
    0.5& 0\\
    0&4.5
  \end{pmatrix}V= \begin{pmatrix}
    2.5&2\\
    2&2.5
  \end{pmatrix},\;\;\;C_1=C=V \begin{pmatrix}
    1& 0\\
    0&3
  \end{pmatrix}V= \begin{pmatrix}
    2&1\\
    1&2
  \end{pmatrix},\\
 N_1&= V \begin{pmatrix}
    0.5& 0\\
    0&1.5
  \end{pmatrix}V=\begin{pmatrix}
    1&0.5\\
    0.5&1
  \end{pmatrix}\;\;\;\text{and}\;\;\;N_2= V \begin{pmatrix}
    3& 0\\
    0&1
  \end{pmatrix}V=\begin{pmatrix}
    2&-1\\
    -1&2
  \end{pmatrix}.
 \end{align*}
Due to their special structure, all these matrices commute. Furthermore, we observe that $A=\frac{1}{2}C^2$ such that the drift is zero. Hence, the 
corresponding fundamental solution is of the following simple form \begin{align*}
       \Phi_r&=V \exp\left(\begin{pmatrix}
    1& 0\\
    0&3
  \end{pmatrix}B_r+\begin{pmatrix}
    0.5& 0\\
    0&1.5
  \end{pmatrix}W_r^1+\begin{pmatrix}
    3& 0\\
    0&1
  \end{pmatrix}W_r^2\right) V\\
  &=V \begin{pmatrix}
    \exp\left(B_r+0.5 W_r^1+3W_r^2\right)& 0\\
    0&\exp\left(3 B_r+1.5 W_r^1+W_r^2\right)
  \end{pmatrix} V.
            \end{align*}
Inserting (\ref{eq:expressolwithfund}) into (\ref{eq:special_feynman_kac}) with the above fundamental matrix and using numerical integration to 
estimate the expected value delivers an ``exact'' solution $u$ of the underlying RPDE. The goal of this section is to compare the exact solution with 
the solution that is obtained from the pseudo-regression procedure from Section~\ref{sec:regression}.

We conduct the experiments for two different paths of $W$ (see Figure \ref{fig:paths_lin}). In both cases we choose $W^1$ and $W^2$ to be fixed paths 
of independent scalar fractional Brownian motions with Hurst index $H=0.4$.
\begin{figure}[h]
\centering
\includegraphics[width=0.85\textwidth,height=120px]{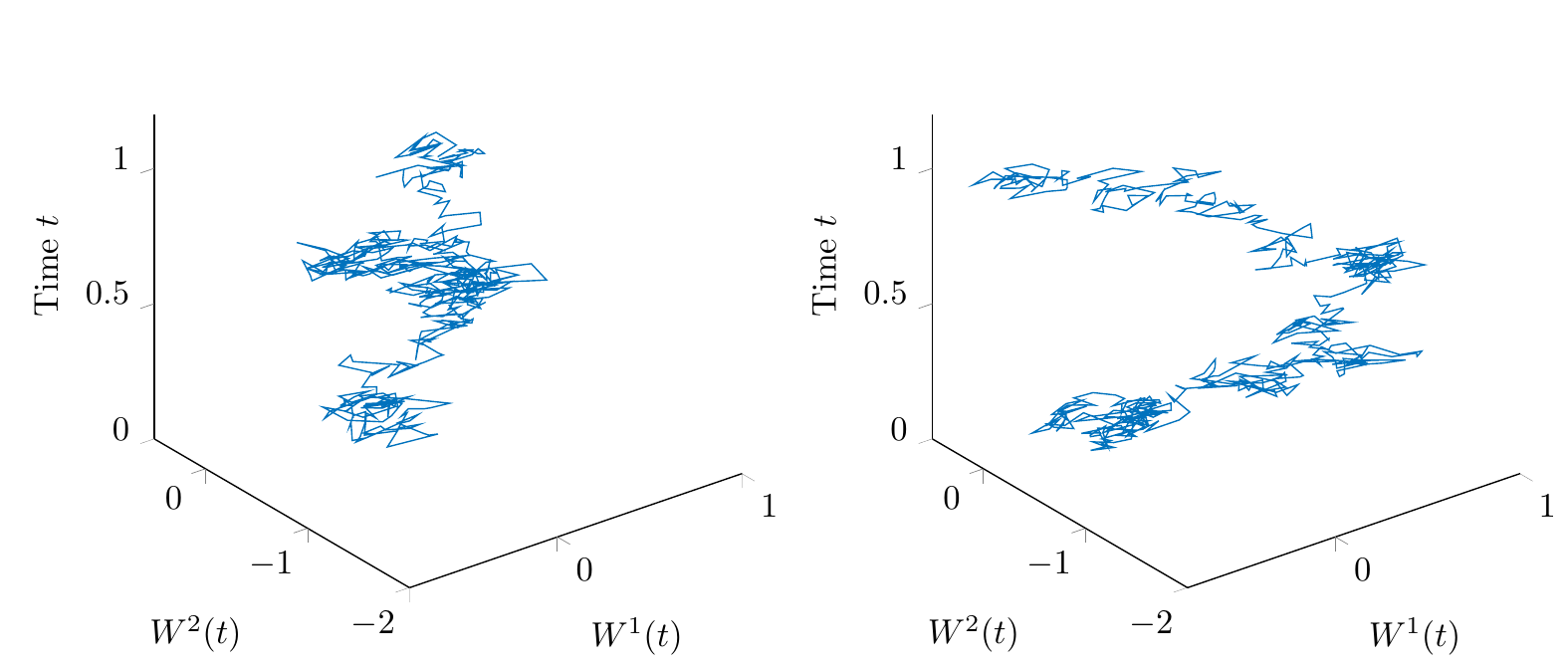}
\caption{Two paths of a fractional Brownian motion with Hurst index $H=0.4$.}
\label{fig:paths_lin}
\end{figure}

\paragraph{\textit{Simulations for the first path (left picture in Figure \ref{fig:paths_lin})}}
We compute a numerical approximation $\tilde u$ of the RPDE solution $u$ based
on the pseudo-regression procedure, see Theorem~\ref{psth}, where for every fixed $t\in[0, 1]$ the approximation $\tilde u(t, \cdot)$ is derived according to (\ref{semstoch}) and (\ref{gass}). We start with the 
left driver in Figure \ref{fig:paths_lin}. Within the numerical approximation we encounter three different errors. The regression error itself 
depends on the number of basis functions $K$ and the number of samples $M$ that we use to approximate the expected value with respect to the 
probability measure $\mu$ of the initial data. In order to generate the paths of (\ref{eq:Stra-SDE}) that we require for the regression approach, we 
need to apply the Euler scheme from Section \ref{sec:Euler}. The error in this discretization depends on the step size $h$ which is our third 
parameter.

We choose $\mu$ to be the uniform measure on $[0,1]^2$ and zero elsewhere.
Based on this we choose Legendre polynomials on $[0, 1]^2$ as an ONB 
$\left(\psi_i\right)_{i=1, \ldots, K}$ of $L^2(\mathbb R^2, \mu)$. To be more precise, we consider the Legendre polynomials up to a certain fixed order $p$ for 
every spatial direction and then take into account the total tensor product between the basis functions of different spatial variables, such that $K=(p+1)^n$.

In Figure \ref{fig:lin_regression_W12} we plot the regression solution $\tilde u$ on $[0, 1]^2$ for three different time points. Here, we use $K=36$ 
Legendre polynomials, $M=10^6$ samples and a step size $h=2^{-9}$ of the Euler scheme (\ref{eq:simple_euler_scheme}). We observe in our simulations 
that $\tilde u$ is a very good approximation for the reference solution $u$ for these fixed parameters. The plots for $u$ look exactly as in Figure 
\ref{fig:lin_regression_W12}. Since there is no visible difference between $u$ and $\tilde u$, we omit the pictures for $u$.\smallskip

\begin{figure}[h!]
\centering
\includegraphics[width=0.32\textwidth,height=120px]{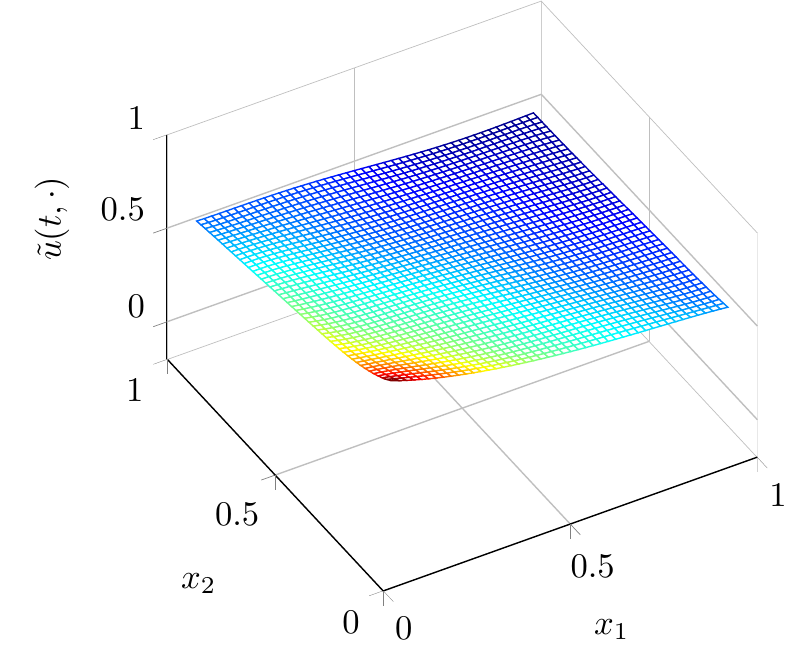}
\includegraphics[width=0.32\textwidth,height=120px]{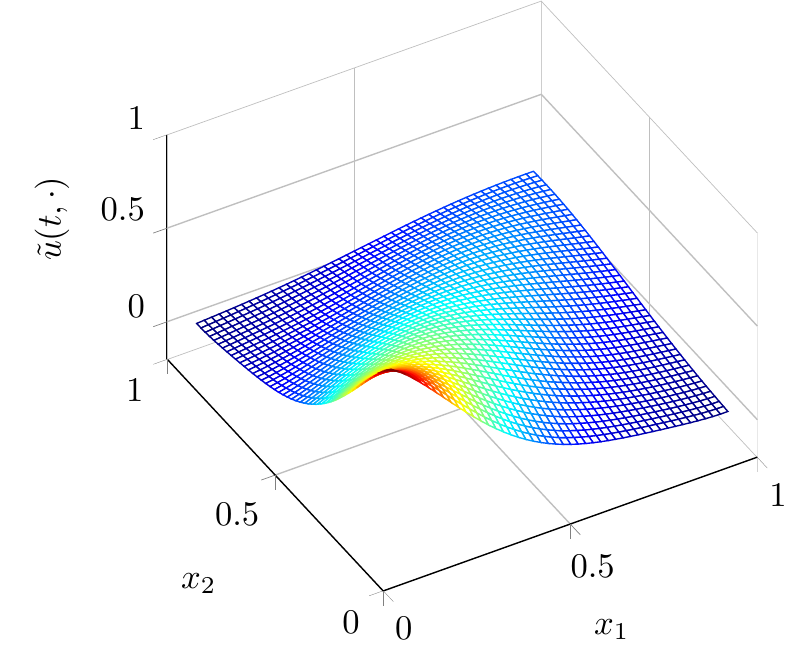}
\includegraphics[width=0.32\textwidth,height=120px]{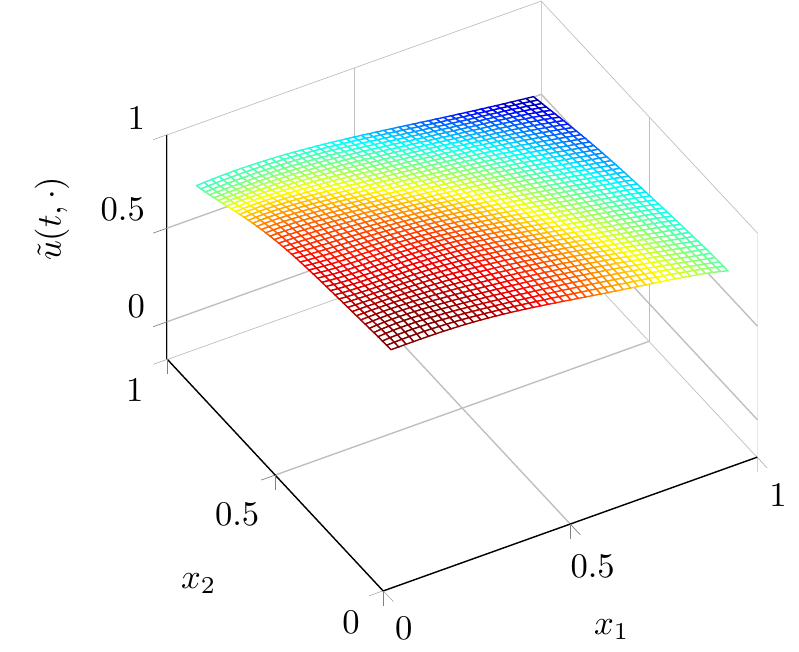}
\caption{Pseudo-regression solution $\tilde u(t,\cdot)$, $t=0, 0.84, 0.99$, of the RPDE driven by the left path in Figure \ref{fig:paths_lin}. The 
parameters are $K=36$, $h=2^{-9}$, $M=10^6$.}
\label{fig:lin_regression_W12}
\end{figure}
Below, we investigate how sensitive the pseudo-regression approach is in every single parameter. Therefore, we always fix two parameters and vary the 
remaining third one. All the errors are measured in $L^2([0, 1]^2)$, i.e., we compute $\left\|u(t, \cdot)-\tilde u(t, \cdot) \right\|_{L^2([0, 
1]^2)}$, $t\in[0, 1]$, or the corresponding relative error. In Figure \ref{fig:error_dif_h}, the absolute and relative errors are shown for different 
step sizes $h=2^{-7}, 2^{-8}, 2^{-9}$. If we compare the curves with the largest step size with the ones having the smallest step size, we can see 
that there is a remarkable difference. We observe that the error is most of the times twice and sometimes up to three time larger when using a four 
times larger step size. This can lead to an relative approximation error of more than $10\%$. This implies that a small step size $h$ is recommended 
in order to ensure a small error. This is not surprising since the order of convergence in $h$ is worst out of all parameters. \smallskip

\begin{figure}[h]
\centering
\includegraphics[width=0.85\textwidth,height=120px]{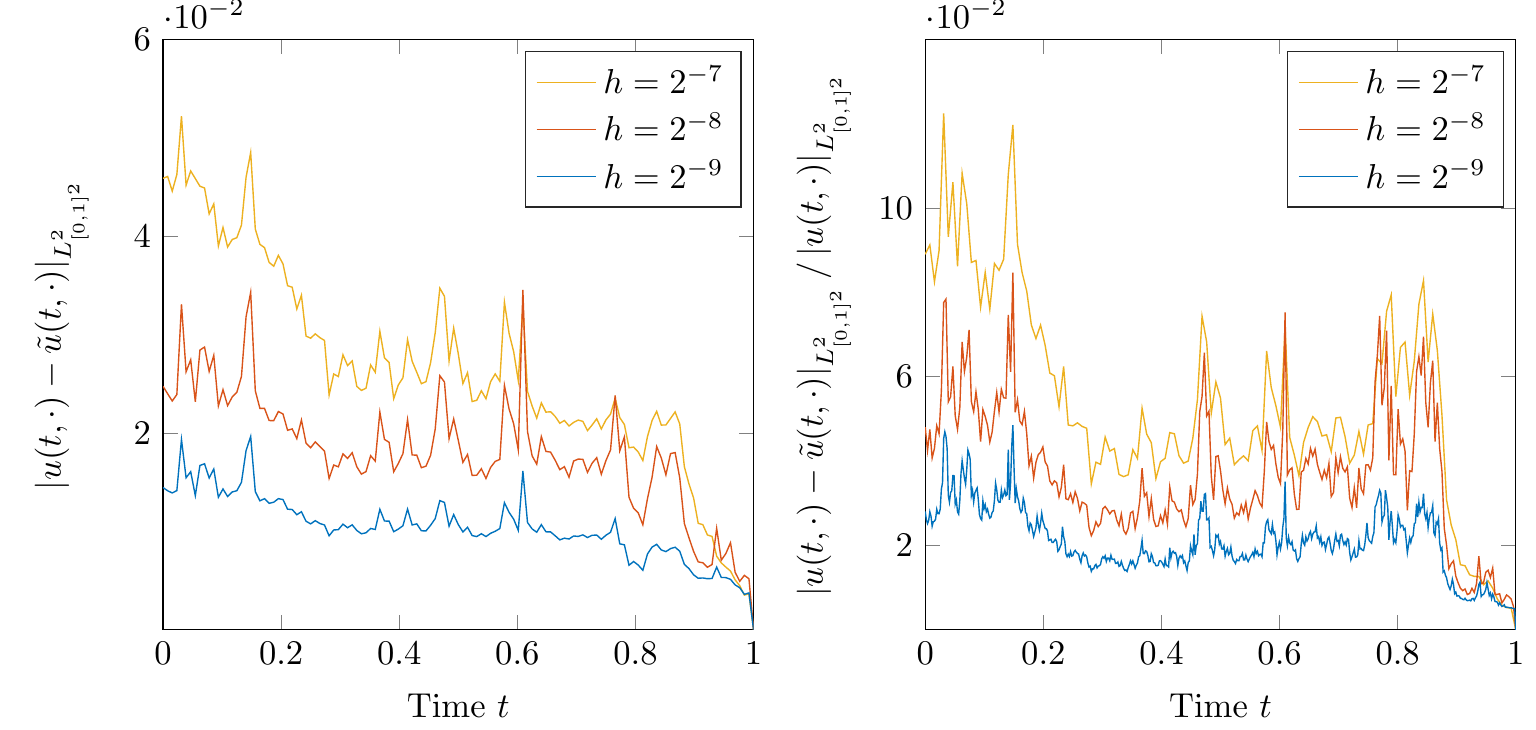}
\caption{Absolute and relative error between RPDE and pseudo-regression solution. The parameters are $K=36$, $M=10^6$ and $h=2^{-7}, 2^{-8}, 2^{-9}$.}
\label{fig:error_dif_h}
\end{figure}
Now we fix the number of basis functions and the step size of the Euler method. For different numbers of samples $M=10^4, 10^5, 10^6$, we find the 
errors in Figure \ref{fig:error_dif_M}. We see that it does not really matter whether $10^6$ or $10^5$ samples are used, whereas $10^4$ samples are 
probably too few, since the relative error can be up to $9\%$.\smallskip

\begin{figure}[h]
\centering
\includegraphics[width=0.85\textwidth,height=120px]{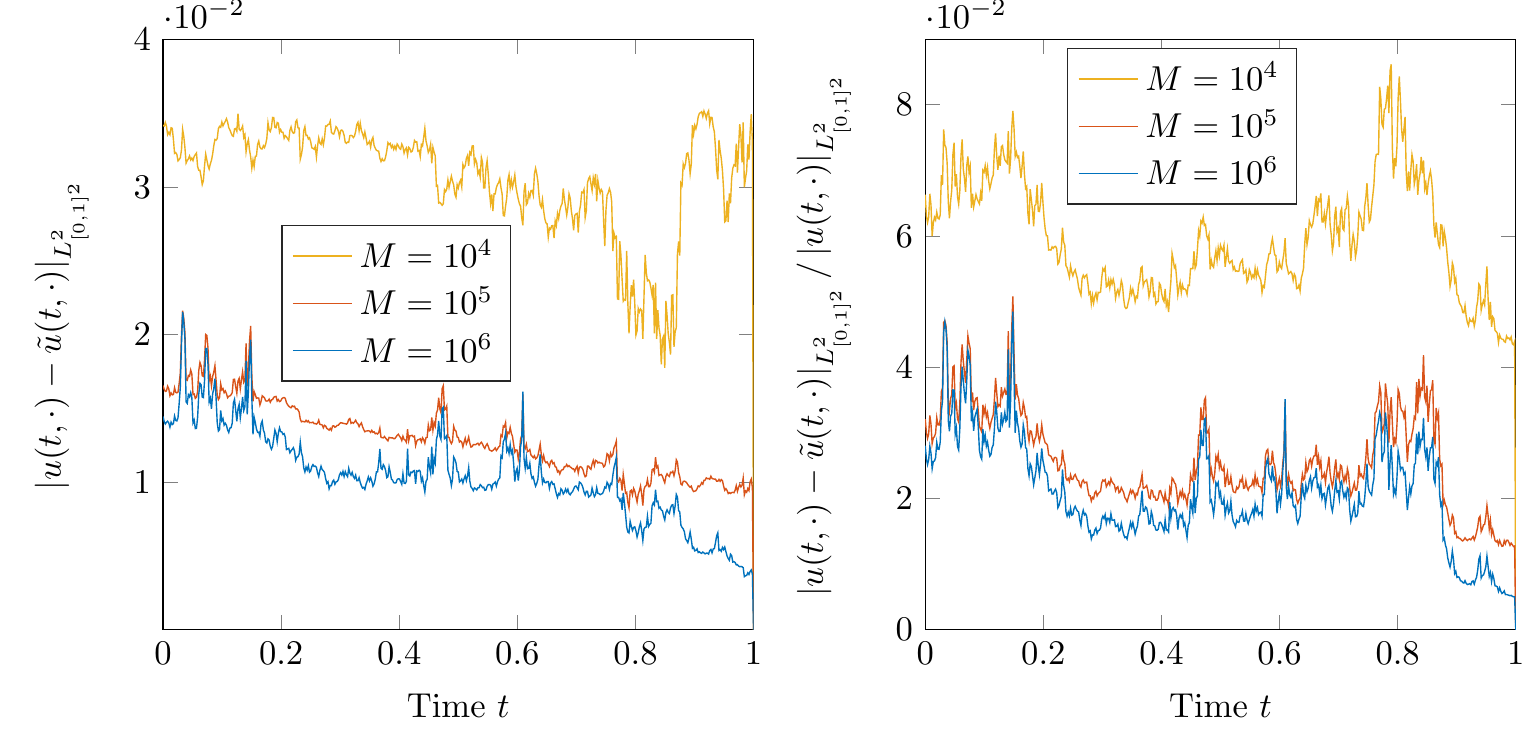}
\caption{Absolute and relative error between RPDE and pseudo-regression solution. The parameters are $K=36$, $M=10^4, 10^5, 10^6$ and $h=2^{-9}$.}
\label{fig:error_dif_M}
\end{figure}
It remains to analyze the error in the number of basis functions. In Figure \ref{fig:lin_regression_W12}, the solution looks relatively flat, such 
that it is not surprising that the parameter $K$ only plays a minor role. Since there is barely a difference when varying $K$, we state the 
logarithmic errors in Figure \ref{fig:error_dif_K} for $K=16, 81$. The error for $K=36$ lies between the curves in Figure \ref{fig:error_dif_K} and 
is omitted because it would have been hard to distinguish between the plots if it would have been included. Even in the logarithmic scale there is 
almost no difference in the errors. Moreover, we observe that for most of the time points, an additional error is caused by taking too many 
polynomials. Thus, the approach is not at all sensitive in the parameter $K$ for this problem with the left driving path in Figure  
\ref{fig:paths_lin}. This is not true for every driving path as the following experiment will show. Using the right path in Figure  
\ref{fig:paths_lin} as the driver instead will lead to a much larger error if we choose the same parameters as before.
\begin{figure}[h]
\centering
\includegraphics[width=0.85\textwidth,height=120px]{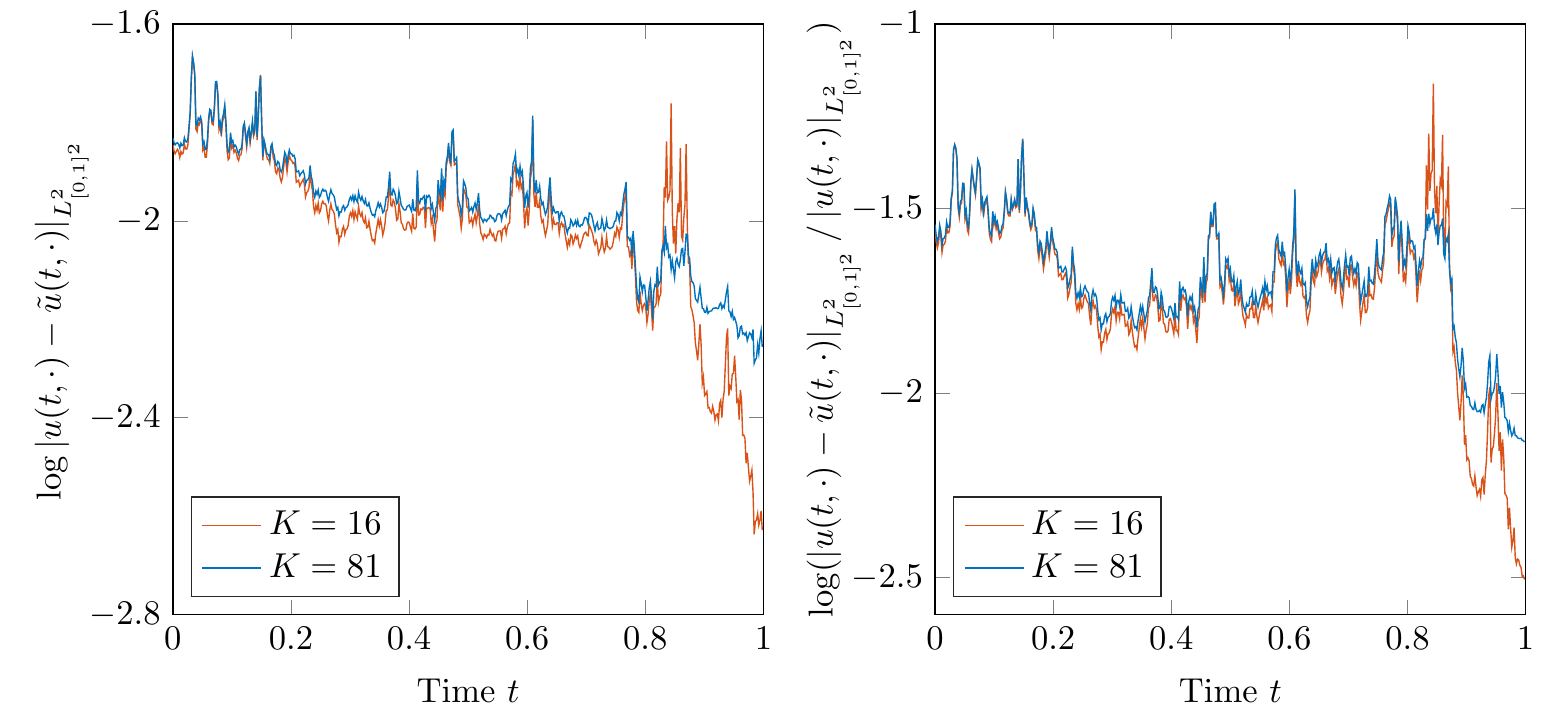}
\caption{Logarithmic absolute and relative error between RPDE and pseudo-regression solution. The parameters are $K=16, 81$, $M=10^6$ and $h=2^{-9}$.}
\label{fig:error_dif_K}
\end{figure}

\paragraph{\textit{Simulations for the second path (right picture in Figure \ref{fig:paths_lin})}}

We conduct a second experiment with the same example as above. We only change the driving path, i.e., the left path in Figure \ref{fig:paths_lin} is 
replaced by the right one. This leads to a very large relative $L^2$-error for $M=10^6$, $h=2^{-9}$ and $K=36$, see Figure \ref{fig:error_dif_K_W1}. 
In the worst case ($t=0.48$) the relative error is almost $80\%$. The reason for this can be seen in Figure \ref{fig:exact_W1}, where $u(0.48, 
\cdot)$ (left) is compared with $\tilde u(0.48, \cdot)$ (right). The exact solution in this worst case is close to be a delta function which is 
generally hard to approximate. The pseudo-regression solution clearly looks differently which shows that the parameter $K$ depends on the underlying driving 
path. If we increase the number of polynomials to $K=121$, we can reduce the error in Figure \ref{fig:error_dif_K_W1} but still many more basis 
functions would be required to obtain a small relative error which is still large at every time point, where $u$ is close to be a delta function.
\begin{figure}[h]
\centering
\includegraphics[width=0.425\textwidth,height=150px]{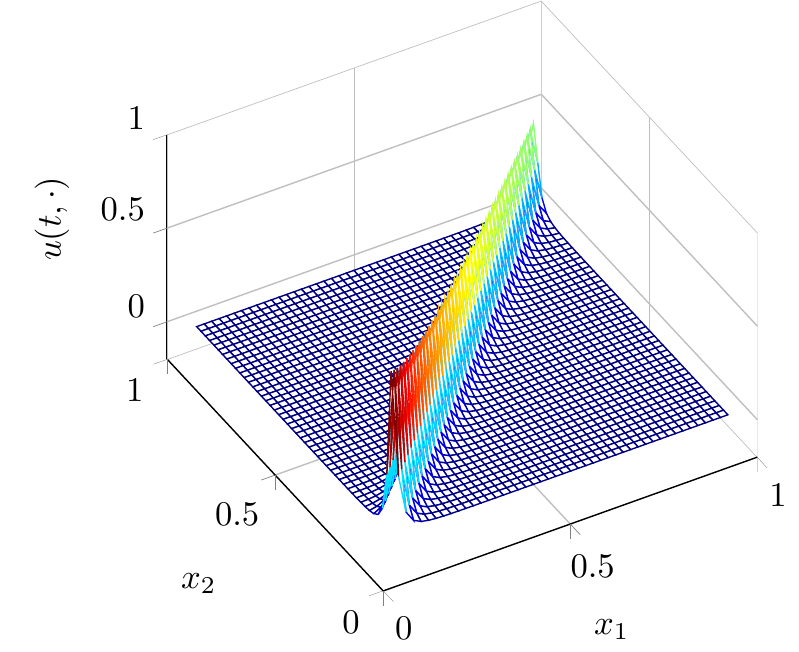}
\includegraphics[width=0.425\textwidth,height=150px]{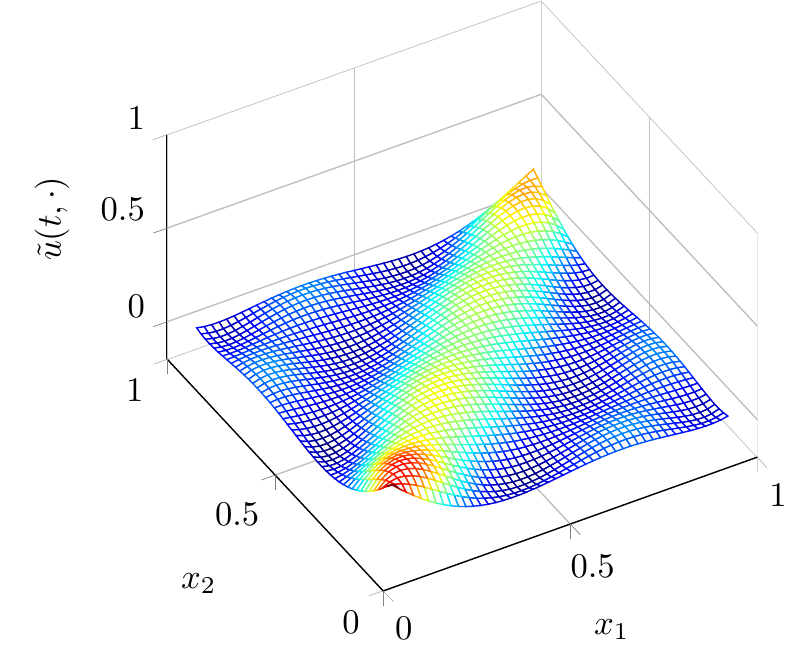}
\caption{Exact solution $u(t,\cdot)$ (left) and pseudo-regression solution $\tilde u(t,\cdot)$ (right), $t=0.48$, of the RPDE driven by the right 
path in Figure \ref{fig:paths_lin} (coefficients as in (\ref{lin_coef})). The regression parameters are $K=36$, $h=2^{-9}$, $M=10^6$.}
\label{fig:exact_W1}
\end{figure}
\begin{figure}[h]
\centering
\includegraphics[width=0.425\textwidth,height=120px]{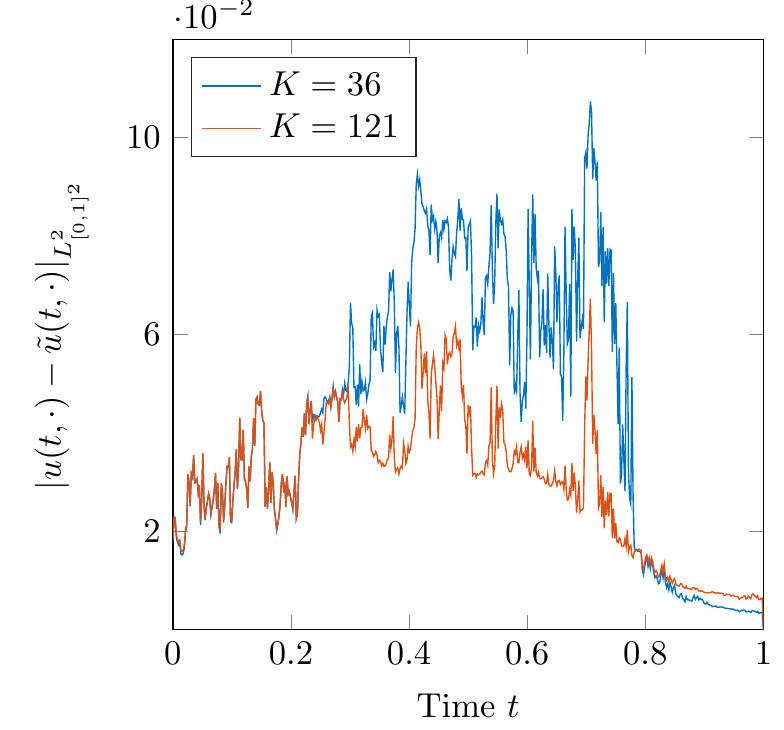}
\includegraphics[width=0.425\textwidth,height=120px]{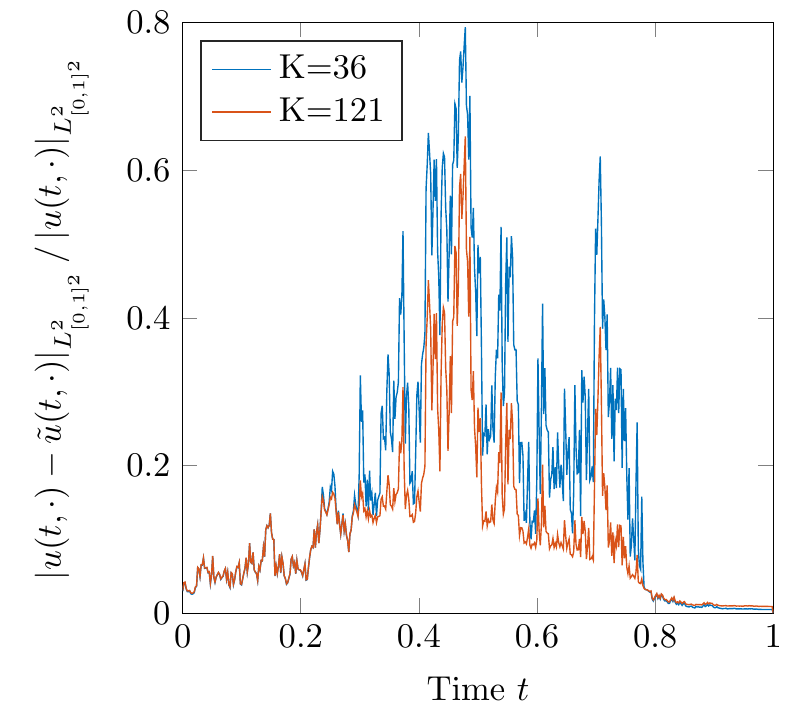}
\caption{Absolute and relative error between RPDE and pseudo-regression solution with parameters $M=10^6$, $h=2^{-9}$ and $K=36, 121$.}
\label{fig:error_dif_K_W1}
\end{figure}

\subsubsection{Numerical example with nilpotent matrices}

In Subsection \ref{sec:experiment_commute} an example has been considered that does not depend on the complete rough path but only on the path. 
Therefore, we investigate another case, where still a reference solution can be derived but this time it depends on the full rough path. \smallskip

Let us consider an scenario that fits the framework (\ref{eq:2nilpotent}). We set $b, c, \gamma\equiv 0$ in system (\ref{eq:Cauchy-RPDE}) with 
terminal time $T=1$ and terminal value $g(x)=\exp\left(-0.5 \left\|x\right\|^2\right)$. Moreover, we define 
\begin{equation}\label{lin_coef_3D}
   \sigma(x)=A_1 x,\;\;\; \beta_1(x)=A_2 x\;\;\;\text{and}\;\;\;  \beta_2(x)=A_3 x,  
\end{equation}
where we assume $m=1$, $d=2$ and a three-dimensional space variable $x\in\mathbb R^3$. Again, by Theorem \ref{thr:existence+uniqueness+cauchy},
the solution to (\ref{eq:Cauchy-RPDE}) has the following stochastic representation: \begin{equation}\label{eq:special_feynman_kac2}
    u(t,x; \mathbf{W}) = E\left[ g(X_1^{t,x})\right], \quad   (t,x) \in [0,1] \times \R^3,
  \end{equation}
where $X_1^{t,x}$ satisfies the RDE \begin{equation}
  \label{eq:linear-rde_example}
  dX_r = A_1 X_r dB_r+A_2 X_r\mathbf{W}^1_r+A_3 X_r\mathbf{W}^2_r,\quad X_t=x, \quad r \in [t,1].
\end{equation}
Now we fix the coefficients such that (\ref{eq:2nilpotent}) is fulfilled: \begin{align*}
  A_1=\left(\begin{matrix}
       0&0&1\\0&0&0\\0&0&0
      \end{matrix}\right),\quad A_2=\left(\begin{matrix}
       0&1&0\\0&0&0\\0&0&0
      \end{matrix}\right),\quad A_3=\left(\begin{matrix}
       0&0&0\\0&0&1\\0&0&0
      \end{matrix}\right).                                                                     
                              \end{align*}
Notice that $A_1^2=0$ implies that the It\^o-Stratonovich correction term is zero in (\ref{eq:linear-rde_example}) such that we automatically obtain 
a geometric driver in the equation. The driving path is again a fractional Brownian motion with Hurst index $H=0.4$.
Now, since $A_1$ commutes with $A_2$ and $A_3$ it can be seen that $\ [A_1,A_2] = \ [A_1,A_3]= 0$. Furthermore, we observe that $\ [A_2,A_3]= A_1$ 
which by Lemma \ref{lem:linear-rde-solution-nilpotent} leads to the following solution representation:
\begin{align*}
  X_r^{t,x} = \exp\left( A_1 (B_{t,r}-a^{12}_{t,r})+A_2 W^1_{t,r}+A_3 W^2_{t,r}\right) x,
\end{align*}
where the term\begin{equation*}
    a^{12}_{t,r} = \half \left(\int_t^r W^1_{t,s} dW^2_s - \int_t^r W^2_{t,s} dW^1_s\right)
  \end{equation*}
is approximated numerically by using piece-wise linear approximations to $W^1$ and $W^2$ on a very fine time grid. Consequently, we have 
\begin{align*}
  X_r^{t,x} = \exp\left( \begin{smallmatrix}
       0& W^1_{t,r}&B_{t,r}-a^{12}_{t,r}\\0&0& W^2_{t,r}\\0&0&0
      \end{smallmatrix}\right) x=:\exp\left( D_{t,r}\right) x.
\end{align*}
The matrix $D_{t,r}$ is nilpotent with index $3$, so that $\exp\left( D_{t,r}\right)=I+D_{t,r}+\frac{1}{2}D_{t,r}^2$ which then leads to 
\begin{align*}
  X_r^{t,x} = \left(\begin{smallmatrix}
       1& W^1_{t,r}&B_{t,r}-a^{12}_{t,r}+0.5 W^1_{t,r}W^2_{t,r} \\0&1& W^2_{t,r}\\0&0&1
      \end{smallmatrix}\right) x.
\end{align*}
Inserting this into (\ref{eq:special_feynman_kac2}) and estimating the expected value with numerical integration provides the exact solution $u$ of 
the underlying RPDE.\smallskip

The probability measure within the regression approach has the same structure as before but it is now defined on 
$\mathbb R^3$, i.e., we choose $\mu$ to be the uniform measure on $[0,1]^3$ and zero elsewhere.  
So, the ONB of $L^2(\mathbb R^3, \mu)$ is given by Legendre polynomials on $[0, 1]^3$. In Figure \ref{fig:error_regress_3D}, the 
absolute and the relative error in $L^2([0, 1]^3)$ between the exact and the pseudo-regression solution is considered. The algorithm also works very 
well in this case since the relative error is less than $1\%$ for $K=64$ basis polynomials, $M=10^6$ samples and a Euler step size of 
$h=2^{-9}$.\smallskip

\begin{figure}[h]
\centering
\includegraphics[width=0.85\textwidth,height=120px]{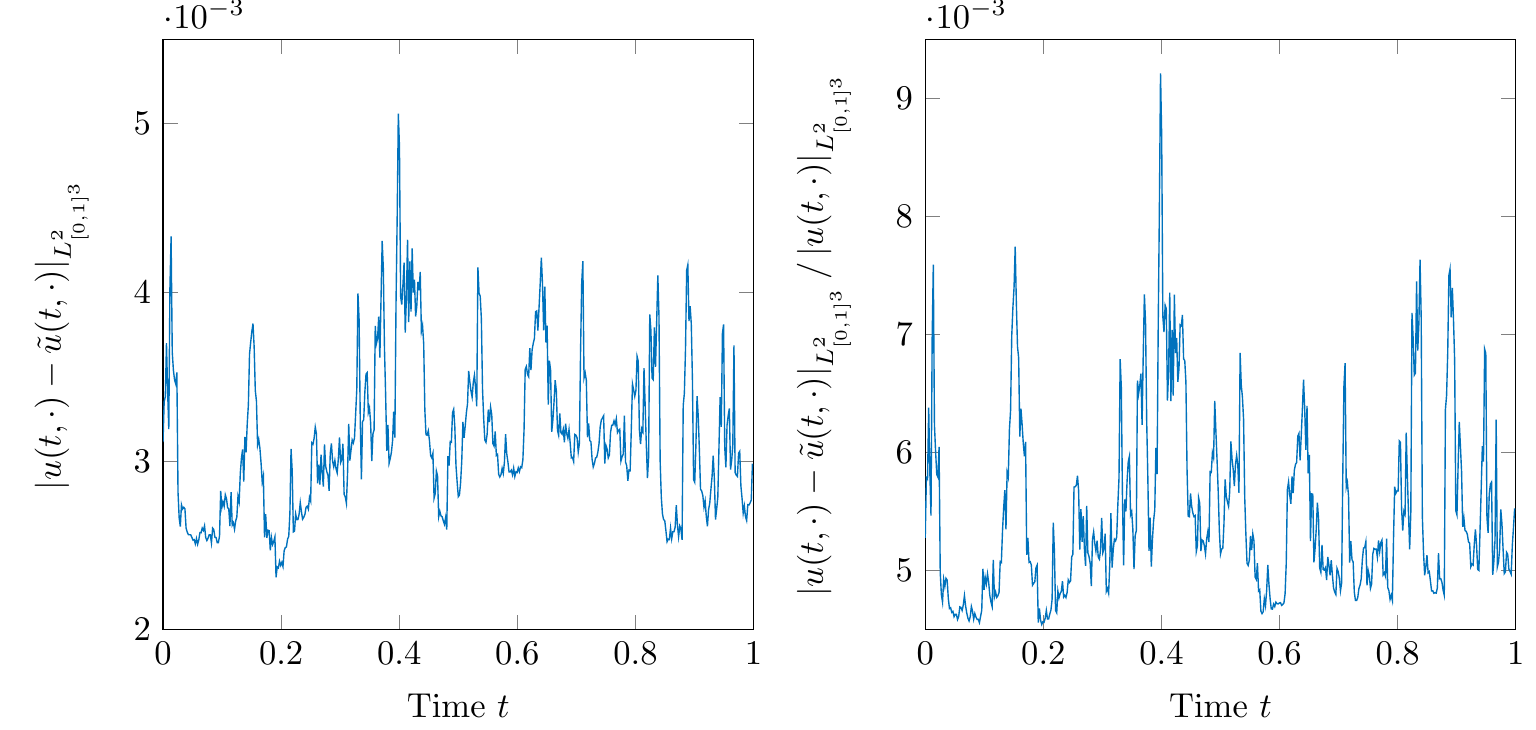}
\caption{Absolute and relative error between the RPDE with coefficients as in (\ref{lin_coef_3D}) and the pseudo-regression solution. The parameters are 
$K=64$, $M=10^6$ and $h=2^{-9}$.}
\label{fig:error_regress_3D}
\end{figure}
The pseudo-regression $\tilde u$ for several fixed time points is illustrated in Figure \ref{fig:regression_3D}. In the cubic domain the color represents 
the corresponding function value of $\tilde u(t, \cdot)$ which is relatively large if the color is red and relatively small if the color is blue. We 
omit the plots for the exact solution since there is no visible difference to the regression solution. \smallskip

\begin{figure}[h]
\centering
\includegraphics[width=0.32\textwidth,height=120px]{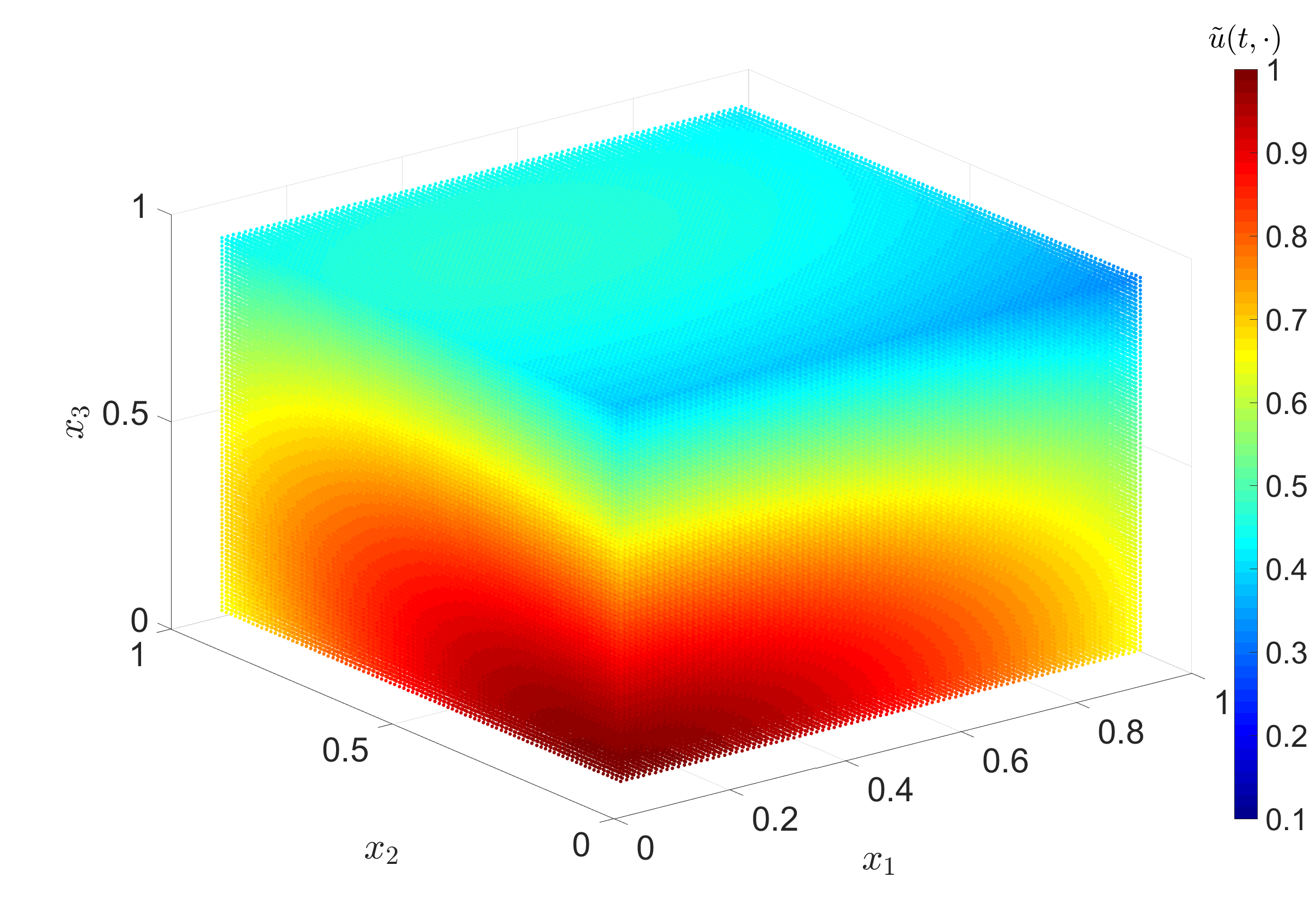}
\includegraphics[width=0.32\textwidth,height=120px]{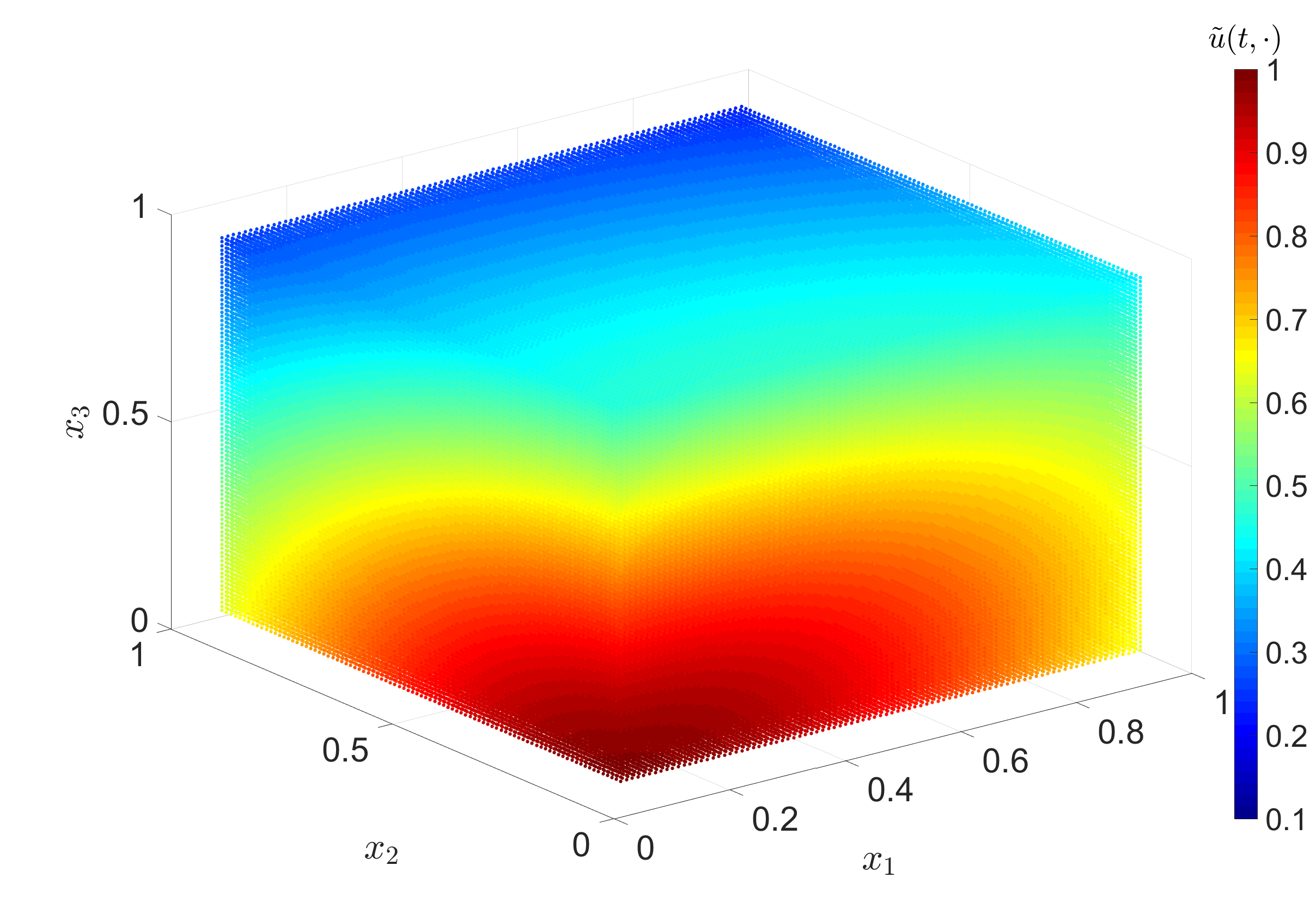}
\includegraphics[width=0.32\textwidth,height=120px]{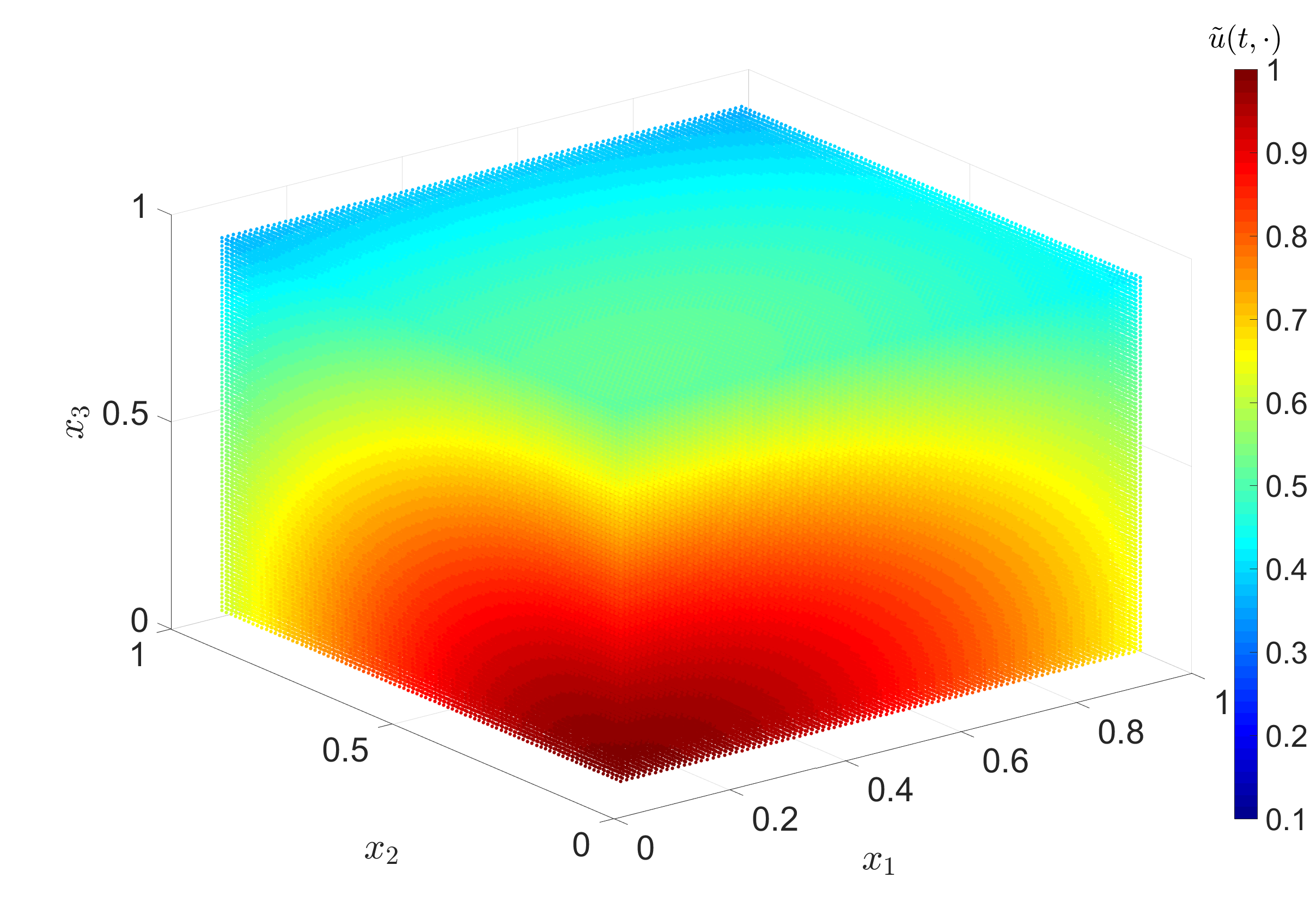}
\includegraphics[width=0.32\textwidth,height=120px]{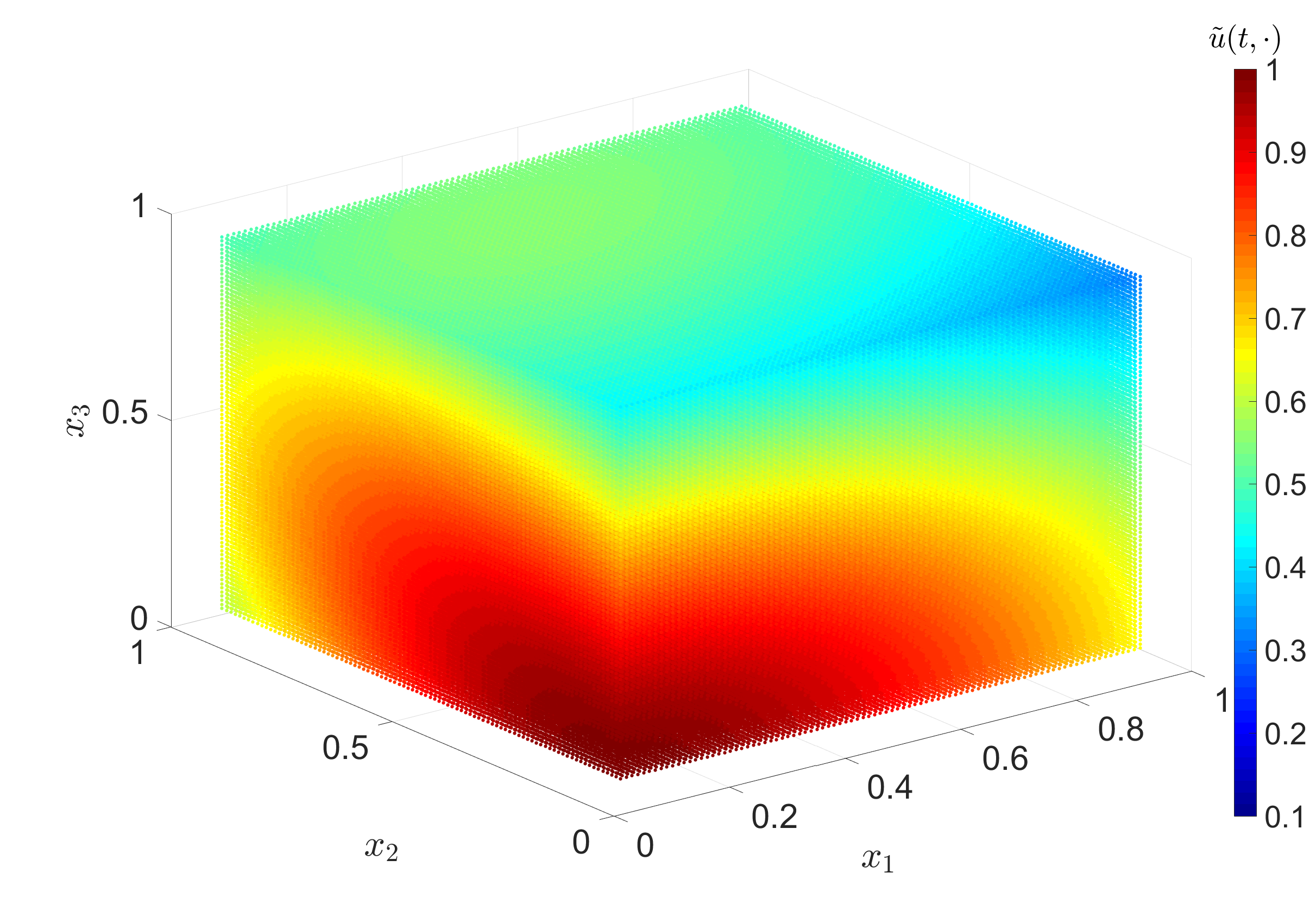}
\includegraphics[width=0.32\textwidth,height=120px]{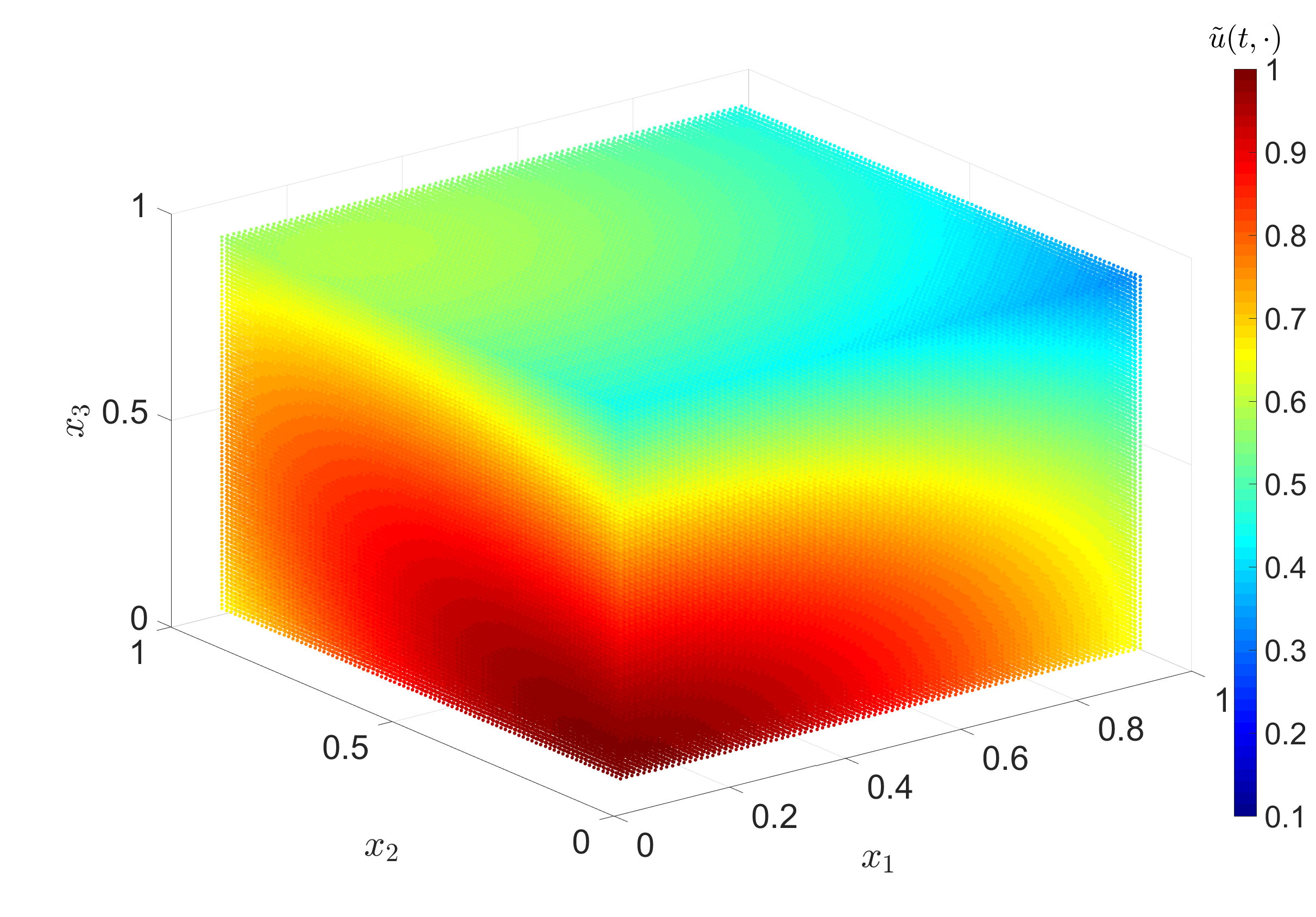}
\includegraphics[width=0.32\textwidth,height=120px]{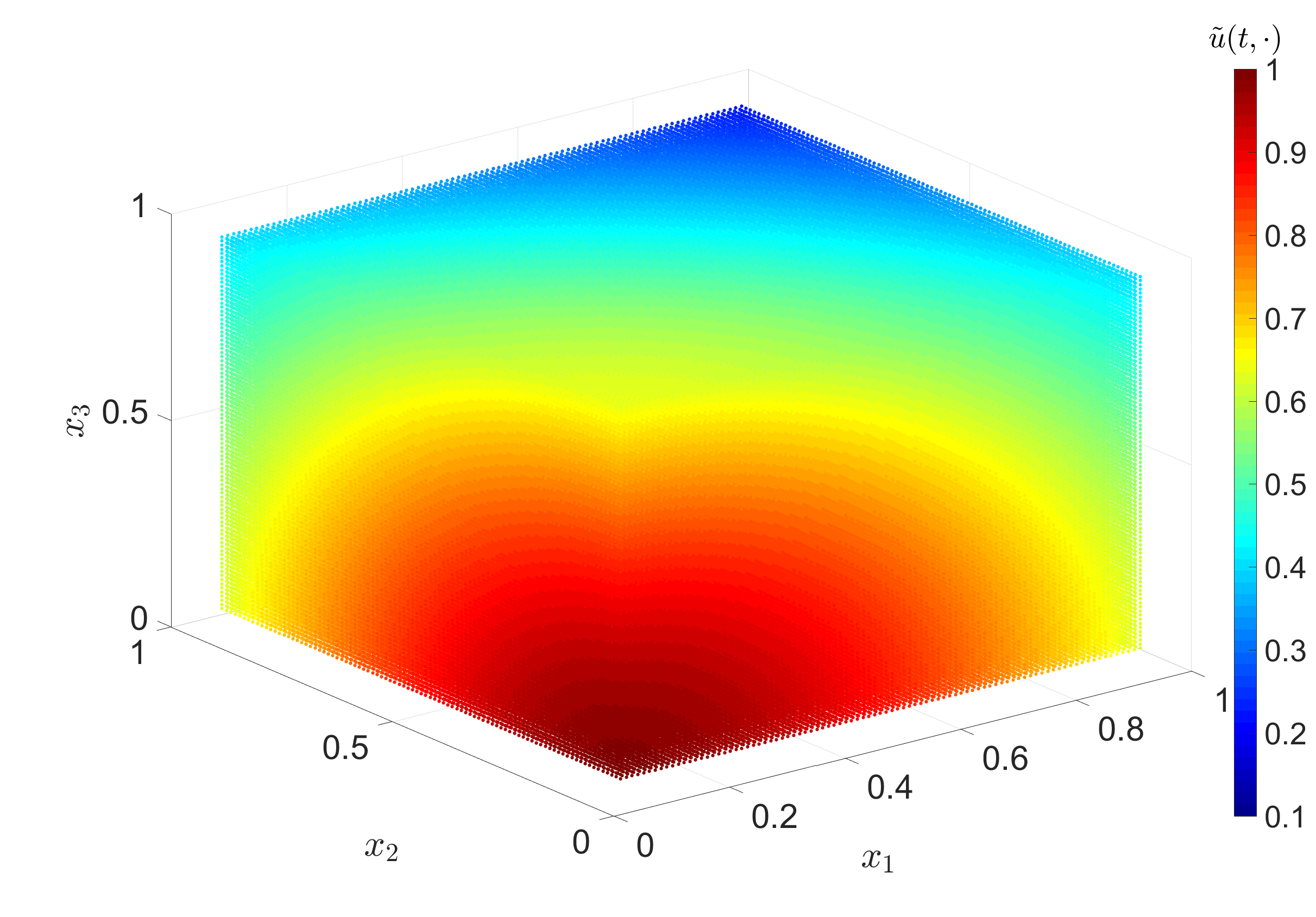}
\caption{The pseudo-regression solution $\tilde u(t,\cdot)$, $t=0, 0.14, 0.43, 0.61, 0.78, 0.99$, of the RPDE driven a path of a fBm. The parameters 
are $K=64$, $h=2^{-9}$, $M=10^6$.}
\label{fig:regression_3D}
\end{figure}
We conclude this section by discussing an alternative to the pseudo-regression, that is the stochastic regression, where an approximation $\hat u(t, \cdot)$ to $u(t, \cdot)$ is derived 
based on (\ref{reg}) instead of (\ref{gass}). Within the stochastic regression the Euler scheme (\ref{eq:simple_euler_scheme}) has to be used only once, whereas we run (\ref{eq:simple_euler_scheme}) 
for every fixed $t$ when computing the pseudo-regression solution $\tilde u(t, \cdot)$. Consequently, $\hat u$ can be computationally cheaper than $\tilde u$ if the Euler method is very expensive in terms of 
 computational time. We determine the solution of the stochastic regression $\hat u$ with the same paprameters as before and compare it with  the exact solution in Figure \ref{fig:error_stoch_regress_3D}.
In this context, we modify our basis, i.e., we use $\tilde\psi_i=g \psi_i$, where $\psi_i$ are again Legendre polynomials ($i=1, \ldots, K$). This compensation is required since $\psi_i$ takes very large values 
outside $[0, 1]^3$. Now, we evaluate the basis functions at samples of the solution to (\ref{eq:Stra-SDE}) in order to compute $\mathcal{M}^{(t)}$ in (\ref{reg}).
Since the paths of the solution to (\ref{eq:Stra-SDE}) leave $[0, 1]^3$ quite frequently, we would encounter a very large variance and hence a large error when using the non-modified basis. 
With the basis $(\tilde\psi_i)_{i=1, \ldots, K}$ we see that the error in Figure \ref{fig:error_stoch_regress_3D} is relatively small but it is clearly larger than the error of the pseudo-regression in
Figure \ref{fig:error_regress_3D}.
\begin{figure}[h]
\centering
\includegraphics[width=0.85\textwidth,height=120px]{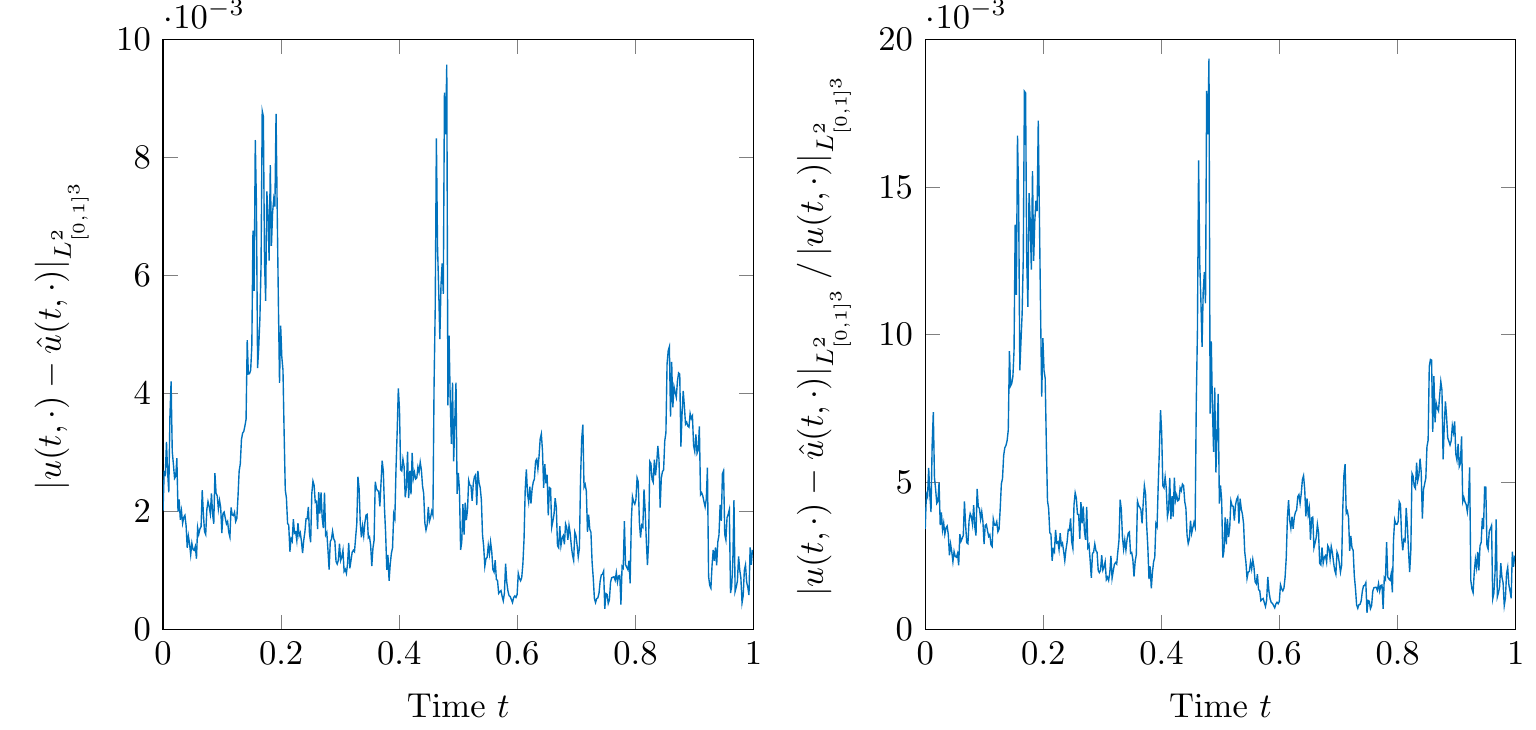}
\caption{Absolute and relative error between the RPDE with coefficients as in (\ref{lin_coef_3D}) and the stochastic regression solution $\hat u$. 
The parameters are as in Figure \ref{fig:error_regress_3D}, i.e., $K=64$, $M=10^6$ and $h=2^{-9}$.}
\label{fig:error_stoch_regress_3D}
\end{figure}

\subsection{Numerical example with non-linear vector fields}
\label{sec:non-linear-example}
We conclude the numerical section with an example which has no reference solution. We start with a similar setting as in Section \ref{sec:linear-example}, i.e.,
we assume that $c, \gamma\equiv 0$. Hence, the solution of the underlying RPDE is given by the expected value in (\ref{eq:special_feynman_kac}) but here non-linear vector fields
enter equation (\ref{eq:Stra-SDE}). We define them as follows: \begin{equation}\label{non-lin_coef}
   b(x)=\left(\begin{smallmatrix} \sin(x_1+2x_2)\\ \sin(2x_1+x_2) \end{smallmatrix}\right),\;\;\;\sigma(x)=\left(\begin{smallmatrix} x_1+\sin(x_2)\\ 2x_1+0.5\cos(x_2) \end{smallmatrix}\right),
  \;\;\;  \beta(x)=\left(\begin{smallmatrix} 0.3\cos(x_1+x_2)&0.2(x_1+x_2)\\ \sin(x_2)-0.5 x_1& \sin(x_1 x_2)-0.5x_2\end{smallmatrix}\right),    
\end{equation}
where we suppose to have a scalar Brownian motion $B$ ($m=1$), a two-dimensional spaces variable $x=\left(\begin{smallmatrix} x_1\\ x_2\end{smallmatrix}\right)\in\mathbb R^2$
as well as a two dimensional rough path $\mathbf W$ ($n=d=2$). Moreover, the terminal
time and the terminal value of the RPDE are $T=1$ and $g(x)=\exp\left(-0.5 \left\|x\right\|^2\right)$, respectively. We fix the probability measure $\mu$ like in 
Subsection \ref{sec:experiment_commute} such that the ONB is again represented by Legendre polynomials on $[0, 1]^2$. We apply the pseudo-regression approach to this case and illustrate 
the resulting solution $\tilde u$ in Figure \ref{fig:non_lin_regres} for $K=36, h=2^{-8}$ and $M=10^6$. Although there is no reference solution to determine the exact error, 
we expect the approximation to be good because $\tilde u$ is relatively flat in space and does not show an extreme behavior like in Figure \ref{fig:exact_W1}.
\begin{figure}[h]
\centering
\includegraphics[width=0.85\textwidth,height=120px]{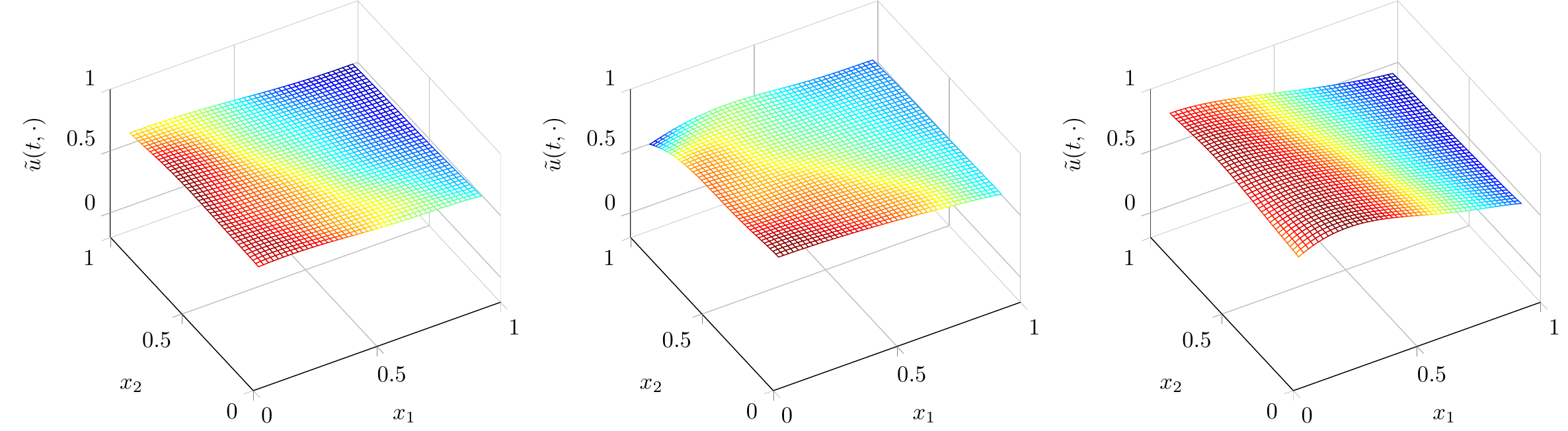}
\includegraphics[width=0.85\textwidth,height=120px]{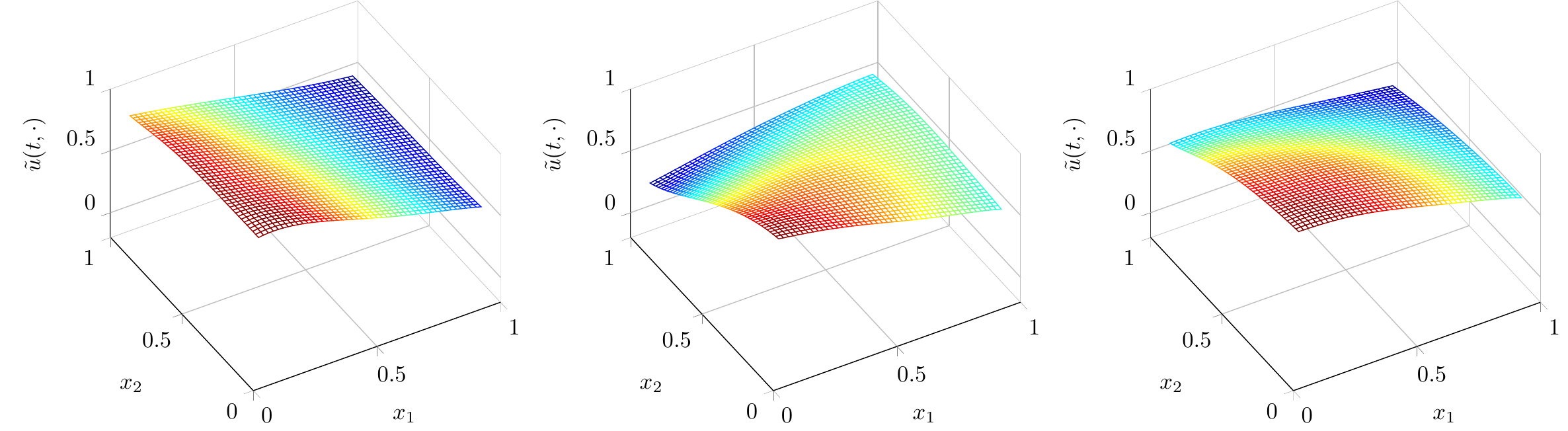}
\caption{The pseudo-regression solution $\tilde u(t,\cdot)$, $t=0, 0.33, 0.57, 0.73, 0.87, 0.99$, of the RPDE (vector fields as in (\ref{non-lin_coef})) driven by a path of a fBm with $H=0.4$.
The regression parameters are $K=36, h=2^{-8}$ and $M=10^6$.}
\label{fig:non_lin_regres}
\end{figure}


\appendix

  \appendix
  
  \section{Proof of Theorem~\ref{thm:bounds_flow_deriv}}
\label{sec:proof-theor-refthm:b}
  
  \subsection{Controlled $p$-variation paths}
  
Recall that a function 
\begin{align*}
 \omega \colon \{0 \leq s \leq t \leq T\} \to [0,\infty)
\end{align*}
is called $\emph{control}$ if it is continuous, $\omega(t,t) = 0$ for all $t \in [0,T]$ and if it is superadditive, i.e.
\begin{align*}
 \omega(s,u) + \omega(u,t) \leq \omega(s,t)
\end{align*}
for all $s \leq u \leq t$. Examples of a control include $\omega(s,t) = |t - s|$ or $\omega(s,t) = \| \mathbf{Z} \|_{p-\text{var};[s,t]}^p$ provided $\mathbf{Z}$ is a $p$-rough path. We say that the \emph{$p$-variation of $\mathbf{Z}$ is controlled by a control function $\omega$} if $| \mathbf{Z}_{s,t} |^p \leq \omega(s,t)$ holds for all $s \leq t$. For arbitrary control functions, the quantity $N_{\alpha}(\omega;[s,t])$ is defined exactly as in \eqref{eq:N} by replacing $\| \mathbf{Z} \|_{p-\text{var};[u,v]}^p$ by $\omega(u,v)$.

The following definition generalizes the notion of a \emph{controlled path} from H\"older- to $p$-variation rough paths.
\begin{definition}
 Let $U$ and $W$ be normed spaces. Let $Z \colon [0,T] \to U$ be a path whose $p$-variation is controlled by a control function $\omega$. We say that a path $y \colon [0,T] \to W$ is \emph{controlled by $Z$ and $\omega$} if there exists a path $y' \colon [0,T] \to L(U,W)$ whose $p$-variation is controlled by $\omega$ so that for $R^y$ given implicitly by the relation
 \begin{align*}
  y_{s,t} = y'_s Z_{s,t} + R^y_{s,t}, 
 \end{align*}
 we have
 \begin{align*}
  \| R^y \|_{p/2-\omega; [0,T]} \coloneqq \sup_{0 \leq s < t \leq T} \frac{| R^{y}_{s,t} |}{\omega(s,t)^{\frac{2}{p}}} < \infty.
 \end{align*}
 We will usually not explicitly mention the control $\omega$ and just say that $y$ is controlled by $Z$. We denote by $\mathscr{D}_{Z}^p([0,T],W)$ the space of controlled $p$-rough paths. We will call a function $y'$ with the given property a \emph{Gubinelli-derivative} of $y$ (with respect to $Z$).

\end{definition}
It is an (admittedly lengthy) exercise to show that all classical estimates proven for H\"older rough paths can be generalized to $p$-rough paths and their controlled functions in the sense above for $p \in [2,3)$ when replacing $|t-s|$ by $\omega(s,t)$ in these estimates. Indeed, the results follow by using an appropriate version of the \emph{Sewing Lemma} for control functions which was proven recently, even for discontinuous control functions, in \cite[Theorem 2.2]{FZ17}. For instance, the corresponding results for rough integrals are summarized in the following theorem.

\begin{theorem}\label{thm:rough_integral}
  Let $U,W,\hat{W}$ be finite dimensional vector spaces and $\mathbf{Z} = (Z, \mathbb{Z})$ be a $p$-rough path with $p$-variation controlled by $\omega$, $p \in [2,3)$. Let $y \in \mathscr{D}_{Z}^p([0,T],L(W,\hat{W}))$ and $z \in \mathscr{D}_{Z}^p([0,T],W)$. Then
  \begin{align*}
   \int_s^t y_u\, dz_u = (\mathcal{I} \Xi)_{s,t}\ ,  \qquad \Xi_{u,v} = y_u z_{u,v} + y'_u z'_u \mathbb{Z}_{u,v}
  \end{align*}
  exists as an abstract integral (cf. \cite[Lemma 4.2 and p. 49 eq. (4.6)]{FH14}), and there is a constant $C$ depending only on $p$ such that the estimate
  \begin{align*}
   &\left| \int_s^t y_u\, dz_u - y_s z_{s,t} - y'_s z'_s \mathbb{Z}_{s,t} \right| \\
   &\quad \leq C \left(\|x\|_{p-\omega;[s,t]} \|R^z\|_{p/2-\omega;[s,t]} + \|\mathbb{Z}\|_{p/2-\omega;[s,t]} \|y' z' \|_{p-\omega;[s,t]} \right) \omega(s,t)^{3/p}
  \end{align*}
  holds for every $s < t$. In particular, the map $t \mapsto \int_0^t y_u\, dz_u$ is itself a controlled $p$-rough path, both controlled by $z$ with derivative $y$, and by $Z$ with derivative $y z'$. 
  
  \end{theorem}

  \begin{proof}
   A combination of \cite[Theorem 4.10 and Remark 4.11]{FH14} generalized to $p$-rough paths.
  \end{proof}

  In the next lemmas, we prepare some estimates for rough integrals and
  solutions to rough differential equations we are going to use at the end of
  this section. In the following, $U,W,\hat{W},W_1,W_2, \ldots$ will denote
  finite dimensional vector spaces, and $\mathbf{Z}$ will be a fixed weakly geometric $p$-rough
  path, $p \in [2,3)$, with values in $U \oplus (U \otimes U)$, 
 controlled by a control function $\omega$. $C \geq 0$ will denote a generic constant whose actual value may change from line to line and which might depend on the parameters specified before.
  \begin{lemma}\label{lemma:rem_RDE_sol}
   Let $V \colon W \to L(U,W)$ and let $y \colon [0,T] \to W$ be a solution to
   \begin{align}\label{eqn:mult_RDE}
    y_t = x + \int_0^t V(y_s)\, d\mathbf{Z}_s.
   \end{align}
   Assume that 
   \begin{align*}
    \|y \|_{p-\omega} \vee \|\mathbf{Z} \|_{p-\omega} \leq 1 \quad \text{and} \quad \| V \|_{\mathcal{C}^2_b} \leq 1.
   \end{align*}
   Then there are constants $C$ and $\alpha$ depending only on $p$ such that
   \begin{align*}
    \| R^y \|_{p/2-\omega} \leq C (1 + N_{\alpha}( \omega ; [0,T])).
   \end{align*}

  \end{lemma}
  
  \begin{proof}
   For $\alpha > 0$, set 
   \begin{align*}
    \| R^y \|_{p/2-\omega; \alpha} := \sup_{0 \leq s < t \leq T; |t-s| \leq \alpha} \frac{|R^y_{s,t}|}{\omega(s,t)^{2/p}}.
   \end{align*}
   Choose $s < t$ such that $\omega(s,t) \leq \alpha$. Then we have, using the estimate in Theorem \ref{thm:rough_integral},
   \begin{align*}
    |R^y_{s,t}| &= \left| \int_s^t  V(y_u)\, d\mathbf{Z}_u - V(y_s) Z_{s,t} \right| \\
    &\leq \left| \int_s^t  V(y_u)\, d\mathbf{Z}_u - V(y_s) Z_{s,t} -  D V(y_s) V(y_s) \mathbb{Z}_{s,t}  \right| + |D V(y_s) V(y_s) \mathbb{Z}_{s,t}| \\
    &\leq C \alpha^{1/p} \omega(s,t)^{2/p} ( \| R^y \|_{p/2-\omega; \alpha} + \| D V(y_{\cdot}) V(y_{\cdot}) \|_{p-\omega} ) +  \omega(s,t)^{2/p}.
   \end{align*}
   Using boundedness of $V$ and its derivatives, one can check that
   \begin{align*}
    \| D V(y_{\cdot}) V(y_{\cdot}) \|_{p-\omega} \leq C \| y \|_{p-\omega} \leq C.
   \end{align*}
   Hence we obtain
   \begin{align*}
    \| R^y \|_{p/2-\omega; \alpha} \leq  C \alpha^{1/p}\| R^y \|_{p/2-\omega; \alpha} + C.
   \end{align*}
   Choosing $\alpha$ such that $C \alpha^{1/p} \leq 1/2$, we obtain
   \begin{align*}
    \| R^y \|_{p/2-\omega; \alpha} \leq 2C.
   \end{align*}
   Now choose $(\tau_n)$ such that $0 = \tau_0 < \tau_1 < \ldots < \tau_N < \tau_{N+1} = T$ with $\omega(\tau_i, \tau_{i+1}) \leq \alpha$ and $N =  N_{\alpha}( \omega ; [0,T])$. Let $s < t$ be arbitrary. Choose $i$ and $j$ such that $s \in [\tau_{i-1 }, \tau_{i})$ and $ t\in (\tau_{j }, \tau_{j+1}]$. Then
   \begin{align*}
    \frac{|R^y_{s,t}|}{\omega(s,t)^{2/p}} &\leq \omega(s,t)^{- 2/p} \big( |R^y_{s,\tau_{i}}| + |R^y_{\tau_i,\tau_{i+1}}| + \ldots + |R^y_{\tau_j,t}| \\
    &\quad + |V(y_{\tau_i}) - V(y_s)||Z_{\tau_i,\tau_{i+1}}| + \ldots + |V(y_{\tau_j}) - V(y_s)||Z_{\tau_j,t}|    \big) \\
    &\leq \frac{|R^y_{s,\tau_{i}}|}{\omega(s,\tau_i)^{2/p}} + \ldots + \frac{|R^y_{\tau_j,t}|}{\omega(\tau_j,t)^{2/p}} + \frac{|V(y_{\tau_i}) - V(y_s)||Z_{\tau_i,\tau_{i+1}}|}{\omega(s,\tau_i)^{1/p} \omega(\tau_i,\tau_{i+1})^{1/p}} + \ldots + \frac{|V(y_{\tau_j}) - V(y_s)||Z_{\tau_j,t}|}{\omega(s,\tau_j)^{1/p} \omega(\tau_j,t)^{1/p}} \\
    &\leq 2C( N_{\alpha}( \omega ; [0,T]) + 1) + C ( N_{\alpha}( \omega ; [0,T]) + 1).\qedhere
   \end{align*}
  \end{proof}

  \begin{lemma}\label{lemma:est_rough_integral}
   Let $y$ be a solution to \eqref{eqn:mult_RDE}. Consider
   \begin{align*}
    \zeta_t = \zeta_0 + \int_0^t \nu(y_s)(d\mathbf{Z}_s)z_s \in L(\hat{W}, W)
   \end{align*}
   where $\nu \colon W \to L(U, L(W_1, W))$ is bounded, twice differentiable with bounded derivatives and $z \colon [0,T] \to L(\hat{W}, W_1)$ is controlled by $Z$. Assume that 
   \begin{align*}
    \|y \|_{p-\omega} \vee \|\mathbf{Z} \|_{p-\omega} \leq 1 \quad \text{and} \quad \| \nu \|_{\mathcal{C}^2_b} \leq 1.
   \end{align*}
   Then there is a constant $C$ and some $\alpha > 0$ depending on $p$ such that
   \begin{align*}
    \|\zeta \|_{p-\omega;[s,t]} &\leq C \omega(s,t)^{2/p} \Big( \|z\|_{\infty;[s,t]}(1 + N_{\alpha}( \omega ; [0,T])) + \|z\|_{\infty;[s,t]} + \|z'\|_{\infty;[s,t]} \\
   &\quad + \|z\|_{p-\omega;[s,t]} + \|z'\|_{p-\omega;[s,t]} + \|R^z\|_{p/2-\omega;[s,t]} \Big)\\
   &\quad + C\omega(s,t)^{1/p}(\|z\|_{\infty;[s,t]} + \|z'\|_{\infty;[s,t]}) + C\|z\|_{\infty;[s,t]}
   \end{align*}
   and
   \begin{align*}
    \| R^{\zeta} \|_{p/2-\omega;[s,t]} &\leq C \omega(s,t)^{1/p} \Big( \|z\|_{\infty;[s,t]}(1 + N_{\alpha}( \omega ; [0,T])) + \|z\|_{\infty;[s,t]} + \|z'\|_{\infty;[s,t]} \\
   &\quad + \|z\|_{p-\omega;[s,t]} + \|z'\|_{p-\omega;[s,t]} + \|R^z\|_{p/2-\omega;[s,t]} \Big) + C(\|z\|_{\infty;[s,t]} + \|z'\|_{\infty;[s,t]})
   \end{align*}
    for all $s \leq t$. 
  \end{lemma}
  
  \begin{proof}
  Note first that the path $t \mapsto y_t$ is controlled by $Z$, and the path $t \mapsto \nu(y_t)z_t \in L(U, L(\hat{W}, W))$ is controlled by $Z$ as composition with a smooth function \cite[Lemma 7.3]{FH14}. Moreover, its Gubinelli derivative is given by
  \begin{align*}
   (\nu(y_t)z_t)' = (\nu(y_t))'z_t + \nu(y_t)z_t' = D\nu(y_t)(y'_t)z_t + \nu(y_t)z_t' = D\nu(y_t)V(y_t)z_t + \nu(y_t)z_t'
  \end{align*}
  where we used ($p$-variation versions of) \cite[Lemma 7.3]{FH14} in the first and second equality and \cite[Theorem 8.4]{FH14} in the third. 
   We start to prove the claimed estimate for $R^{\zeta}$. The Gubinelli derivative of $\zeta$ is given by $\zeta_t' = \nu(y_t)z_t$ and
  \begin{align*}
   R^{\zeta}_{s,t} = \int_s^t \nu(y_u)(d\mathbf{Z}_u) z_u \,  - \nu(y_s)(z_s) Z_{s,t}.
  \end{align*}
  Hence we can estimate 
  \begin{align*}
   | R^{\zeta}_{s,t} | &\leq \left|\int_s^t \nu(y_u)(d\mathbf{Z}_u )z_u  - \nu(y_s)(z_s) Z_{s,t} - (\nu(y_s)z_s)' \mathbb{Z}_{s,t} \right| + \left|D \nu(y_s) V(y_s)z_s \mathbb{Z}_{s,t} + \nu(y_s)z_s' \mathbb{Z}_{s,t} \right| \\
   &\leq C \left( \|R^{\nu(y_{\cdot})z_{\cdot}}\|_{p/2-\omega;[s,t]} + \| D \nu(y_{\cdot}) V(y_{\cdot}) z_{\cdot} \|_{p-\omega;[s,t]} + \| D \nu(y_{\cdot})  z'_{\cdot} \|_{p-\omega;[s,t]} \right) \omega(s,t)^{3/p} \\
   &\quad + C(\|z\|_{\infty;[s,t]} + \|z'\|_{\infty;[s,t]}) \omega(s,t)^{2/p}
  \end{align*}
  where we used Theorem \ref{thm:rough_integral} and that $V$, $\nu$ and all its derivatives are bounded. We have
  \begin{align*}
   |R^{\nu(y_{\cdot})z_{\cdot}}_{s,t}| &= |\nu(y_t)z_t - \nu(y_s)z_s - (\nu(y_s))'z_s Z_{s,t} - \nu(y_s)z'_s Z_{s,t} | \\
   &\leq |(\nu(y_t) - \nu(y_s) - (\nu(y_s))'Z_{s,t})z_t| + |\nu(y_s)(z_t - z_s - z'_s Z_{s,t})| + |(\nu(y_s))' (z_t - z_s) Z_{s,t}| \\
   &\leq C( \|R^{\nu(y_{\cdot})}\|_{p/2-\omega;[s,t]} \|z\|_{\infty} + \|R^{z}\|_{p/2-\omega;[s,t]} + \|z\|_{p-\omega;[s,t]}) \omega(s,t)^{2/p}
  \end{align*}
  using $(\nu(y_s))' = D\nu(y_s)V(y_s)$ and boundedness of the vector fields an their derivatives.
  As in \cite[Lemma 7.3]{FH14}, we can see that
  \begin{align*}
   \|R^{\nu(y_{\cdot})}\|_{p/2-\omega;[s,t]} &\leq C \left( \|y\|^2_{p-\omega;[s,t]} + \|R^{y}\|_{p/2-\omega;[s,t]} \right) \\
   &\leq C (1 + N_{\alpha}( \omega ; [0,T]))
  \end{align*}
  where the second estimate follows from Lemma  \ref{lemma:rem_RDE_sol}. Thus
  \begin{align*}
   \|R^{\nu(y_{\cdot})z_{\cdot}}\|_{p-\omega;[s,t} \leq C(\|z\|_{\infty}(1 + N_{\alpha}( \omega ; [0,T])) + \|R^{z}\|_{p/2-\omega;[s,t]} + \|z\|_{p-\omega;[s,t]}). 
  \end{align*}
  Using the Lipschitz bounds for $V$, $\nu$ and its derivatives, we can easily see that
  \begin{align*}
   \| D \nu(y_{\cdot}) V(y_{\cdot}) z_{\cdot} \|_{p-\omega;[s,t]} &\leq C( \|z\|_{\infty;[s,t]} + \|z\|_{p-\omega;[s,t]} )\quad \text{and} \\
   \| D \nu(y_{\cdot})  z'_{\cdot} \|_{p-\omega;[s,t]} &\leq C( \|z'\|_{\infty;[s,t]} + \|z'\|_{p-\omega;[s,t]} ).
  \end{align*}
  Therefore, we find that
  \begin{align*}
   \| R^{\zeta} \|_{p/2-\omega;[s,t]} &\leq C \omega(s,t)^{1/p} \big( \|z\|_{\infty;[s,t]}(1 + N_{\alpha}( \omega ; [0,T])) + \|z\|_{\infty;[s,t]} + \|z'\|_{\infty;[s,t]} \\
   &\quad + \|z\|_{p-\omega;[s,t]} + \|z'\|_{p-\omega;[s,t]} + \|R^z\|_{p/2-\omega;[s,t]} \big) + C(\|z\|_{\infty;[s,t]} + \|z'\|_{\infty;[s,t]}).
  \end{align*}
  For $\zeta$, we have
  \begin{align*}
   \frac{|\zeta_{s,t}|}{\omega(s,t)^{1/p}} &\leq \omega(s,t)^{1/p} \frac{|R^{\zeta}_{s,t}|}{\omega(s,t)^{2/p}} + \frac{|\nu(y_s)(z_s) Z_{s,t}|}{\omega(s,t)^{1/p}} \\
   &\leq \omega(s,t)^{1/p} \| R^{\zeta} \|_{p/2-\omega;[s,t]}  + C\|z\|_{\infty;[s,t]}
  \end{align*}
  for all $s < t$ and the claim follows.
  \end{proof}
  
  \begin{lemma}
   Let $A \colon [0,T] \to L(U,L(W,W))$ be controlled by $Z$, and let $\mathbf{Z}$ be weakly geometric. Consider a solution $\Phi \colon [0,T] \to L(W, W)$ to
   \begin{align*}
    \Phi_t = \Phi_0 + \int_0^t A(s)(d\mathbf{Z}_s) \Phi_s \in L(W, W).
   \end{align*}
   Then Liouville's formula holds:
\begin{align*}
 \det(\Phi_t) = \det(\Phi_0) \exp\left(  \operatorname{Tr} \int_0^t A(s)\, d \mathbf{Z}_s \right).
\end{align*}
In particular, if $\det(\Phi_0) \neq 0$, $\Phi_t$ is invertible for every $t \geq 0$. In this case, the inverse $\Psi_t := \Phi_t^{-1}$ solves the equation
   \begin{align*}
    \Psi_t = \Phi_0^{-1} - \int_0^t \Psi_s A(s)(d\mathbf{Z}_s)  \in L(W, W).
   \end{align*}
  \end{lemma}
  
  \begin{proof}
   Assume first that $Z$ is smooth. In this case, Liouville's formula is well-known, cf. \cite[(11.4) Proposition]{Ama90}. The statement about $\Psi$ follows from the identity
\begin{align*}
 \frac{d \Phi_t^{-1}}{dt} = - \Phi_t^{-1}  \frac{ d\Phi_t}{dt} \Phi_t^{-1},
\end{align*}
which is true for matrices depending smoothly on $t$. The general case follows by approximation of $\mathbf{Z}$ with smooth rough paths and continuity of the rough integral.
  \end{proof}

  
  \begin{lemma}\label{lemma:jabobian_estimates}
   Let $y$ be a solution to \eqref{eqn:mult_RDE}. Consider a solution $\Phi \colon [0,T] \to L(W, W)$ to
   \begin{align*}
    \Phi_t = \Phi_0 + \int_0^t \nu(y_u)(d\mathbf{Z}_u) \Phi_u \in L(W, W)
   \end{align*}
   where $\nu \colon W \to L(U, L(W, W))$ is bounded, twice differentiable and has bounded derivatives. Assume that  
   \begin{align*}
    \|y \|_{p-\omega} \vee \|\mathbf{Z} \|_{p-\omega} \leq 1 \quad \text{and} \quad \| \nu \|_{\mathcal{C}^2_b} \leq 1.
   \end{align*}
   Then $\Phi$ is controlled by $Z$ and there are constants $C$ and $\alpha > 0$ depending on $p$ such that
   \begin{align*}
    \| R^{\Phi} \|_{p/2-\omega} &\leq C(1 + |\Phi_0|)\exp(C N_{\alpha}( \omega ; [0,T])) \quad \text{and} \\
    \| \Phi \|_{\infty} + \| \Phi \|_{p-\omega} + \| \Phi' \|_{\infty} + \| \Phi' \|_{p-\omega} &\leq C(1 + |\Phi_0|)\exp(C N_{\alpha}( \omega ; [0,T])).
   \end{align*}
   The same estimate holds true for any solution $\Psi \colon [0,T] \to L(W, W)$ to
   \begin{align*}
    \Psi_t = \Psi_0 - \int_0^t \Psi_u \nu(y_u)(d\mathbf{Z}_u)  \in L(W, W)
   \end{align*}
   when we replace $\Phi_0$ by $\Psi_0$.
  \end{lemma}

  \begin{proof}
   Note that $\Phi$ is controlled by $Z$ with derivative
   \begin{align}
    (\Phi_t)' = \nu(y_t)\Phi_t.
   \end{align}
   Using boundedness of $\nu$ and its derivative and our assumptions on $\omega$, this implies that
   \begin{align*}
    \| \Phi' \|_{\infty;[s,t]} \leq C \| \Phi \|_{\infty;[s,t]} \quad \text{and} \quad  \| \Phi' \|_{p-\omega;[s,t]} \leq C(\| \Phi \|_{p-\omega;[s,t]} + \| \Phi \|_{\infty;[s,t]})
   \end{align*}
    for all $s < t$, therefore it is enough to bound $\Phi$ to obtain bounds for $\Phi'$. Let $K$ be a constant such that
    \begin{align*}
     \| \Phi \|_{p-\omega;[0,T]} + \| \Phi \|_{\infty;[0,T]} \leq K.
    \end{align*}
    Let $\alpha > 0$ and choose $s < t$ such that $\omega(s,t) \leq \alpha$. Using Lemma \ref{lemma:est_rough_integral}, we have for sufficiently small $\alpha$
    \begin{align*}
     \| R^{\Phi} \|_{p/2-\omega;[s,t]} \leq C \alpha^{1/p} \| R^{\Phi} \|_{p/2-\omega;[s,t]} + C \alpha^{1/p} K(1 + N_{\alpha}( \omega ; [0,T])) + CK.
    \end{align*}
    Choosing $\alpha$ smaller if necessary, we may assume that $C \alpha^{1/p} \leq 1/2$ and we therefore obtain
    \begin{align*}
      \| R^{\Phi} \|_{p/2-\omega;[s,t]} \leq 2C \alpha^{1/p} K(1 + N_{\alpha}( \omega ; [0,T])) + 2CK \leq CK(1 + N_{\alpha}( \omega ; [0,T])).
    \end{align*}
    Using the same strategy as at the end of the proof of Lemma \ref{lemma:rem_RDE_sol}, we can conclude that
    \begin{align*}
     \| R^{\Phi} \|_{p/2-\omega;[0,T]} \leq CK(1 + N_{\alpha}( \omega ; [0,T]))^2.
    \end{align*}
    Using the results about linear rough differential equations in \cite[Section 5]{FR13}, we see that we can choose 
    \begin{align*}
     K = C(1 + |\Phi_0|)\exp(C N_{\alpha}( \omega ; [0,T]))
    \end{align*}
    and the claim follows for $R^{\Phi}$. The estimates for $\Phi$ can either be obtained by a direct calculation similar to the one performed in Lemma \ref{lemma:est_rough_integral}, but also follow from the results proven for linear rough differential equations in \cite[Section 5]{FR13}. The estimates for $\Psi$ can be obtained in exactly the same way.
\end{proof}

  \begin{lemma}\label{lemma:var_of_const_est}
   Let $y$ be a solution to \eqref{eqn:mult_RDE}. Let $\zeta \colon [0,T] \to L(\hat{W}, W)$ be of the form 
   \begin{align*}
    \zeta_t = \int_0^t \hat{\nu} (y_u)(d \mathbf{Z}_u) \hat{z}_u 
   \end{align*}
   for some $\hat{\nu} \colon W \to L(U, L(W_1, W))$ and $\hat{z} \colon [0,T] \to L(\hat{W}, W_1)$. Assume that $\hat{z}$ is controlled by $Z$. Consider
   \begin{align*}
    z_t = \Phi_t \left( z_0 + \int_0^t \Psi_u\, d\zeta_u \right), \quad z_0 \in L(\hat{W}, W)
   \end{align*}
   with $\Phi, \Psi$ as in Lemma \ref{lemma:jabobian_estimates} where we assume in addition that $\Phi_0 = \Psi_0 = \operatorname{Id}$. Assume that  
   \begin{align*}
    \|y \|_{p-\omega} \vee \|\mathbf{Z} \|_{p-\omega} \leq 1 \quad \text{and} \quad \| \hat{\nu} \|_{\mathcal{C}^2_b} \leq 1.
   \end{align*}
   Let $\kappa \geq 1$ be a constant such that
   \begin{align*}
      \|\hat{z}\|_{\infty} + \|\hat{z}\|_{p-\omega} + \|\hat{z}'\|_{\infty} + \|\hat{z}'\|_{p-\omega} + \|R^{\hat{z}}\|_{p/2-\omega} \leq \kappa. 
   \end{align*}
   Then $z$ is controlled by $Z$ and there are constants $C > 0$ and $\alpha > 0$ depending only on $p$ such that
    \begin{align*}
    &\| z \|_{\infty} + \| z \|_{p-\omega} + \| z' \|_{\infty} + \| z' \|_{p-\omega} + \| R^{z} \|_{p/2-\omega} \\
    &\quad \leq C \kappa (1 + |z_0|)(1 + \omega(0,T)^{1/p} + \omega(0,T)^{4/p} )\exp(C N_{\alpha}( \omega ; [0,T])).
    \end{align*}

  \end{lemma}

  \begin{proof}
   It is clear that $z$ is controlled by $Z$, and the Gubinelli derivative is given by
   \begin{align}
    \begin{split}\label{eqn:Gub_deriv_sol}
    z'_t &= \Phi'_t \left( z_0 + \int_0^t \Psi_u\, d\zeta_u \right) + \Phi_t( \Psi_t \zeta'_t) = \nu(y_t)\Phi_t \left( z_0 + \int_0^t \Psi_u\, d\zeta_u \right) + \hat{\nu} (y_t) \hat{z}_t \\
    &= \nu(y_t) z_t + \hat{\nu}(y_t) \hat{z}_t
      \end{split}
   \end{align}
   where we used Theorem \ref{thm:rough_integral}, $\zeta'_t = \hat{\nu}(y_t)\hat{z}_t$ and the fact that $\Psi_t = (\Phi_t)^{-1}$. For $s < t$, we therefore obtain
   \begin{align*}
    R^z_{s,t} &= z_{s,t} - z'_s Z_{s,t} \\
    &= \Phi_t \int_s^t \Psi_u\, d \zeta_u + \Phi_{s,t} z_s - \nu(y_s)(Z_{s,t}) z_s - \hat{\nu}(y_s)(Z_{s,t}) \hat{z}_s \\
    &= \Phi_t \left( \int_s^t \Psi_u\, d \zeta_u - \Psi_s \zeta_{s,t} + \Psi_s R^{\zeta}_{s,t} \right) + \Phi_t \Psi_s \zeta'_s Z_{s,t} + \Phi_{s,t} z_s - \Phi'_s Z_{s,t} z_s - \hat{\nu}(y_s)(Z_{s,t}) \hat{z}_s  \\
    &= \Phi_t \left( \int_s^t \Psi_u\, d \zeta_u - \Psi_s \zeta_{s,t} + \Psi_s R^{\zeta}_{s,t} \right) + R^{\Phi}_{s,t} z_s +  (\Phi_t - \Phi_s) \Psi_s \hat{\nu}(y_s)(Z_{s,t}) \hat{z}_s
   \end{align*}
    and by the triangle inequality,
   \begin{align}
    \begin{split}\label{eqn:rough_int_inhom}
    |R^z_{s,t}| &\leq \left| \Phi_t\left( \int_s^t \Psi_u\, d\zeta_u - \Psi_s \zeta_{s,t} + \Psi_s R_{s,t}^{\zeta} \right) \right| + \left| R^{\Phi}_{s,t} \left(z_0 + \int_0^s \Psi_u\, d\zeta_u \right) \right| \\
    &\quad + |(\Phi_t - \Phi_s) \Psi_s \hat{\nu}(y_s)(Z_{s,t})\hat{z}_s|.
      \end{split}
   \end{align}
   For the first term on the right hand side in \eqref{eqn:rough_int_inhom}, we can use the estimate in Theorem \ref{thm:rough_integral} to see that
   \begin{align*}
    &\left| \Phi_t\left( \int_s^t \Psi_u\, d\zeta_u - \Psi_s \zeta_{s,t} + \Psi_s R_{s,t}^{\zeta} \right) \right| \\
    \leq\ &C \|\Phi\|_{\infty;[s,t]} \omega(s,t)^{2/p} \Big( (\|R^{\zeta}\|_{p/2-\omega;[s,t]} + \|\Psi' \zeta'\|_{p-\omega;[s,t]})\omega(s,t)^{1/p} + \|\Psi' \|_{\infty;[s,t]} \|\zeta' \|_{\infty;[s,t]} \\
    &\qquad  + \|\Psi \|_{\infty;[s,t]} \|R^{\zeta}\|_{p/2-\omega;[s,t]} \Big).
    \end{align*}
    Note first that
    \begin{align*}
     \|\Psi' \zeta'\|_{p-\omega;[s,t]} \leq \|\Psi'\|_{p-\omega;[s,t]} \|\zeta'\|_{\infty;[s,t]} + \|\Psi' \|_{\infty;[s,t]} \|\zeta'\|_{p-\omega;[s,t]}
    \end{align*}
    and 
    \begin{align*}
     \|\zeta'\|_{\infty} + \|\zeta'\|_{p-\omega} \leq C \kappa.
    \end{align*}
    From Lemma \ref{lemma:est_rough_integral}, it follows that
    \begin{align*}
     \|R^{\zeta}\|_{p/2-\omega;[s,t]} \leq C \kappa \left( \omega(s,t)^{1/p} (1 + N_{\alpha}( \omega ; [0,T])) + 1 \right).
    \end{align*}
    Using the estimates for $\Phi$ and $\Psi$ in Lemma \ref{lemma:jabobian_estimates}, we therefore obtain
    \begin{align*}
     \left| \Phi_t\left( \int_s^t \Psi_u\, d\zeta_u - \Psi_s \zeta_{s,t} + \Psi_s R_{s,t}^{\zeta} \right) \right| \leq C \kappa \omega(s,t)^{2/p} (1 + \omega(0,T)^{1/p} + \omega(0,T)^{2/p} ) \exp(C N_{\alpha}( \omega ; [0,T])) 
    \end{align*}
    for all $s < t$. For the second summand in \eqref{eqn:rough_int_inhom}, we can again use Theorem \ref{thm:rough_integral} to estimate
    \begin{align*}
     \left| R^{\Phi}_{s,t} \left(z_0 + \int_0^s \Psi_u\, d\zeta_u \right) \right| &\leq C \omega(s,t)^{2/p} \| R^{\Phi} \|_{p/2-\omega;[s,t]} \\
     &\quad \times \Big( |z_0| + \omega(0,s)^{3/p}( \|R^{\zeta}\|_{p/2-\omega;[0,s]} + \|\Psi' \zeta' \|_{p-\omega;[0,s]}) \\
     &\qquad + \|\zeta\|_{\infty;[0,s]} + \omega(0,s)^{2/p} |\hat{z}_0|  \Big)
    \end{align*}
    using $\Psi_0 = \operatorname{Id}$, $\Psi'_0 = \nu(y_0)$ and $\zeta'_0 = \hat{\nu}(y_0)\hat{z}_0$. With Lemma \ref{lemma:est_rough_integral}, we obtain
    \begin{align*}
     \|\zeta\|_{\infty;[0,s]} \leq \omega(0,T)^{1/p} \|\zeta\|_{p-\omega} \leq C \kappa (\omega(0,T)^{1/p} + \omega(0,T)^{3/p} )(1 + N_{\alpha}( \omega ; [0,T])).
    \end{align*}
    As above, we obtain the bound
    \begin{align*}
     \left| R^{\Phi}_{s,t} \left(z_0 + \int_0^s \Psi_u\, d\zeta_u \right) \right| \leq C \kappa \omega(s,t)^{2/p} (1 + |z_0|)(1 + \omega(0,T)^{1/p} + \omega(0,T)^{4/p} ) \exp(C N_{\alpha}( \omega ; [0,T])) 
    \end{align*}
    for all $s < t$. For the third term in \eqref{eqn:rough_int_inhom}, we have
    \begin{align*}
     |(\Phi_t - \Phi_s) \Psi_s \hat{\nu} (y_s)(Z_{s,t}) \hat{z}_s| &\leq \kappa \omega(s,t)^{2/p} \| \Phi\|_{p-\omega;[s,t]} \|\Psi\|_{\infty;[0,s]} \\
     &\leq C \kappa \omega(s,t)^{2/p} \exp(C N_{\alpha}( \omega ; [0,T])).
    \end{align*}
    Using all these estimates in \eqref{eqn:rough_int_inhom}, we can conclude that
    \begin{align*}
     \| R^z \|_{p/2-\omega} \leq C \kappa (1 + |z_0|)(1 + \omega(0,T)^{1/p} + \omega(0,T)^{4/p} )\exp(C N_{\alpha}( \omega ; [0,T])).
    \end{align*}
    We proceed with $z$. For $s <t$,
    \begin{align*}
     |z_t - z_s| \leq |\Phi_t| \left| \int_s^t \Psi_u\, d\zeta_u \right| + |\Phi_t - \Phi_s| \left| z_0 + \int_0^s \Psi_u\, d\zeta_u \right|
    \end{align*}
    and as before, we obtain the estimate
    \begin{align*}
     \| z \|_{p-\omega} \leq C \kappa (1 + |z_0|)(1 + \omega(0,T)^{1/p} + \omega(0,T)^{4/p} )\exp(C N_{\alpha}( \omega ; [0,T])).
    \end{align*}
    Similarly,
    \begin{align*}
     \| z \|_{\infty} \leq C \kappa (1 + |z_0|)(1 + \omega(0,T)^{1/p} + \omega(0,T)^{4/p} )\exp(C N_{\alpha}( \omega ; [0,T])).
    \end{align*}
    From \eqref{eqn:Gub_deriv_sol}, we see that
   \begin{align*}
    \|z'\|_{\infty;[s,t]} \leq C(\|z\|_{\infty;[s,t]} + \|\hat{z}\|_{\infty;[s,t]})
   \end{align*}
   and
   \begin{align*}
    \|z'\|_{p-\omega;[s,t]} \leq C(\|z\|_{\infty;[s,t]} + \|z\|_{p-\omega;[s,t]} + \|\hat{z}\|_{\infty;[s,t]} + \|\hat{z}\|_{p-\omega;[s,t]})
   \end{align*}
   for all $s<t$, therefore the same estimates hold for $z'$. This proves the claim.
  \end{proof}

   \begin{lemma}\label{lemma:controlled_tensor}
   Let $z \in \mathscr{D}_{Z}^p([0,T],L(W,W_1))$ and $\hat{z} \in \mathscr{D}_{Z}^p([0,T],L(\hat{W},\hat{W}_1))$. Then $z \otimes \hat{z} \in \mathscr{D}_{Z}^p([0,T],L(W \otimes \hat{W}, W_1 \otimes \hat{W}_1))$ with derivative
   \begin{align*}
    (z \otimes \hat{z})_s'(u) = (z'_s u) \otimes \hat{z}_s + z_s \otimes (\hat{z}'_s u),\quad u \in U
   \end{align*}
  and remainder given by
  \begin{align*}
   R^{z \otimes \hat{z}}_{s,t} = z_{s,t} \otimes \hat{z}_{s,t} + R_{s,t}^z \otimes \hat{z}_s + z_s \otimes R_{s,t}^{\hat{z}}.
  \end{align*}

  \end{lemma} 

  \begin{proof}
   Follows readily from a short calculation.
  \end{proof}

    \begin{proof}[Proof of Theorem \ref{thm:bounds_flow_deriv}]\label{proof:flow_deriv_est}
  W.l.o.g, we may assume $s = 0$, otherwise we may replace $\mathbf{Z}$ by the time-shifted rough path $\mathbf{Z}_{s + \cdot}$ and solve the corresponding equation. Existence of the derivatives and their characterization as solutions to rough differential equations is a classical result, cf. \cite[Section 8.9]{FH14} and \cite[Section 11.2]{FV10}. It remains to prove the claimed bounds for $X^{(k)}$. Let us first assume that $\|V \|_{\mathcal{C}^{2 + k}_b} \leq 1/\kappa$ for some $\kappa \geq 1$ and that $\hat{\omega}$ is some control function for which
  \begin{align*}
   \|\mathbf{Z} \|_{p-\hat{\omega}} \vee \| X^{x}\|_{p-\hat{\omega}} \leq 1
  \end{align*}
  holds (the precise choice of $\kappa$ and $\hat{\omega}$ will be made later). We claim that in this case, there are constants $C$, $\alpha$ and $M$ depending on $p$ and $k$ such that
  \begin{align}
    \begin{split}\label{eqn:interm_bounds}
   &\| X^{(k)} \|_{\infty} + \| X^{(k)} \|_{p-\hat{\omega}} + \| (X^{(k)})' \|_{\infty} + \| (X^{(k)})' \|_{p-\hat{\omega}} + \| R^{X^{(k)}} \|_{p/2-\hat{\omega}} \\
    &\quad \leq C(1 + \hat{\omega}(0,T)^{1/p} + \hat{\omega}(0,T)^{M/p}) \exp(C N_{\alpha}( \hat{\omega} ; [0,T]))
      \end{split}
  \end{align} 
  holds. We prove the claim by induction. For $k = 1$, $X^{(1)}_t =: \Phi_t$ solves
  \begin{align*}
   X^{(1)}_t = \operatorname{Id} + \int_0^t DV(X^x_s)(d \mathbf{Z}_s) X^{(1)}_s
  \end{align*}
  and the bound \eqref{eqn:interm_bounds} follows from Lemma \eqref{lemma:jabobian_estimates}. Let $k \geq 2$ and assume that our claim holds for all $l = 1, \ldots, k-1$. It is easy to see that $X^{(k)}$ solves an inhomogeneous equation of the form
  \begin{align}\label{eqn:inhom_RDE_main_thm}
   X^{(k)}_t = \zeta_t + \int_0^t D V(X^{x}_s) X^{(k)}_s \, d \mathbf{Z}_s \in L((\R^n)^{\otimes k}, \R^n)
  \end{align}
  where $\zeta \colon [0,T] \to L((\R^n)^{\otimes k}, \R^n))$ can be written as
  \begin{align*}
   \zeta_t = \sum_{\stackrel{2 \leq l \leq k}{i_1 + \ldots + i_l = k}} \lambda_{i_1, \ldots, i_l} \int_0^t D^l V(X^x_s)(d\mathbf{Z}_s)(X^{(i_1)}_s \otimes \cdots \otimes X^{(i_l)}_s) =: \int_0^t \hat{\nu}(X^x_s)(d \mathbf{Z}_s) \hat{z}^k_s,
  \end{align*}
  the $\lambda_{i_1, \ldots, i_l}$ being integers which can be explicitly calculated using the Leibniz rule. Note that $\hat{z}^k$ is controlled by $Z$ by the induction hypothesis and Lemma \ref{lemma:controlled_tensor}, therefore the integrals are well defined. Moreover, the estimate \eqref{eqn:interm_bounds} holds for $\hat{z}^k$ instead of $X^{(k)}$ again by the induction hypothesis and Lemma \ref{lemma:controlled_tensor}. We also see that we can choose $\kappa \geq 1$ depending only on $k$ to obtain $\| \hat{\nu} \|_{\mathcal{C}^2_b} \leq 1$. The equation \eqref{eqn:inhom_RDE_main_thm} can be solved with the variation of constants method: making the ansatz $X^{(k)}_t = \Phi_t C_t$, $C_t \in L((\R^n)^{\otimes k}, \R^n))$,
we can conclude that $X^{(k)}$ can be written as   
\begin{align*}
 X^{(k)}_t = \Phi_t   \int_0^t \Phi_s^{-1} d\zeta_s .
\end{align*}
  The claim \eqref{eqn:interm_bounds} for $X^{(k)}$ now follows from Lemma \ref{lemma:var_of_const_est}. We proceed with deducing the bound \eqref{eqn:bounds_derivatives} from \eqref{eqn:interm_bounds}. Note first that $X^x$ also solves the equation 
   \begin{align*}
    X_t^{x} = x + \int_0^t \tilde{V}(X^{x}_s)\, d\tilde{\mathbf{Z}}_s
   \end{align*}  
  where $\tilde{V} = V / (\kappa \|V\|_{\mathcal{C}^{2 + k}_b})$ and $\tilde{\mathbf{Z}} = (\kappa \|V\|_{\mathcal{C}^{2 + k}_b} Z, \kappa^2 \|V\|_{\mathcal{C}^{2 + k}_b}^2 \mathbb{Z})$. Clearly $\|\tilde{V}\|_{\mathcal{C}^{2 + k}_b} \leq 1/\kappa$, and
  \begin{align*}
   \hat{\omega}(s,t) := \| \tilde{\mathbf{Z}} \|_{p-\var;[s,t]}^p + \| X^x \|_{p-\var;[s,t]}^p =  \kappa^p \|V\|_{\mathcal{C}^{2 + k}_b}^p \| \mathbf{Z} \|_{p-\var;[s,t]}^p + \| X^x \|_{p-\var;[s,t]}^p
  \end{align*}
  is a valid choice for $\hat{\omega}$. Therefore, \eqref{eqn:interm_bounds} holds for $X^{(k)}$ with this $\hat{\omega}$. Next, \cite[Corollary 3]{FR13} implies that there is a constant depending on $p$ such that
  \begin{align*}
   N_{1}(X^x ;[0,T]) \leq C(N_{1}(\tilde{\mathbf{Z}} ;[0,T]) + 1).
  \end{align*} 
  Together with \cite[Lemma 4]{FR13}, this implies that
  \begin{align*}
    \hat{\omega}(0,T)^{1/p} &\leq \| \tilde{\mathbf{Z}} \|_{p-\var;[0,T]} + \| X^x \|_{p-\var;[0,T]} \leq N_{1}(\tilde{\mathbf{Z}} ;[0,T]) + N_{1}(X^x ;[0,T]) + 2 \\
    &\leq C(N_{1}(\tilde{\mathbf{Z}} ;[0,T]) + 1) \leq C \exp(N_{1}(\tilde{\mathbf{Z}} ;[0,T])),
  \end{align*}
  therefore also
  \begin{align*}
   \hat{\omega}(0,T)^{1/p} + \hat{\omega}(0,T)^{M/p} \leq C \exp(C N_{1}(\tilde{\mathbf{Z}} ;[0,T])).
  \end{align*}
  From \cite[Lemma 3]{FR13} and \cite[Lemma 5]{BFRS16}, we see that
  \begin{align*}
   N_{\alpha}(\hat{\omega};[0,T]) &\leq 2 N_{\alpha}(\tilde{\mathbf{Z}};[0,T]) +  2  N_{\alpha}(X^x ;[0,T]) + 2 \leq C(N_{1}(\tilde{\mathbf{Z}} ;[0,T]) + 1)
  \end{align*}
  for a constant $C$ depending on $\alpha$ and $p$, and therefore on $p$ only. Using these estimates, \eqref{eqn:interm_bounds} implies that there is a constant $C$ depending on $p$ and $k$ such that
  \begin{align*}
    \| X^{(k)} \|_{p-\hat{\omega}} \leq C \exp \left( C N_{1}( \tilde{\mathbf{Z}} ; [0,T]) \right)
  \end{align*}
  holds. Using \cite[Lemma 1 and Lemma 3]{FR13}, we see that
  \begin{align*}
    N_{1}( \mathbf{\tilde{Z}} ; [0,T])  \leq \kappa^p \|V\|_{\mathcal{C}^{2 + k}_b}^p (2 N_{1}( \mathbf{Z} ; [0,T]) + 1)
  \end{align*}
  which shows that
  \begin{align*}
    \| X^{(k)} \|_{p-\hat{\omega}} \leq C \exp \left( C \|V\|_{\mathcal{C}^{2 + k}_b}^p( N_{1}( \mathbf{Z} ; [0,T]) + 1) \right).
  \end{align*}
%
%
  Note that for every $s \leq t$, the estimate for $X^{x}$ in \cite[Theorem 10.14]{FV10} implies that
  \begin{align*}
    \hat{\omega}(s,t) &= \| \tilde{\mathbf{Z}} \|_{p-\var;[s,t]}^p + \| X^x \|_{p-\var;[s,t]}^p \leq  \| \tilde{\mathbf{Z}} \|_{p-\var;[s,t]}^p + C  \| \tilde{\mathbf{Z}} \|_{p-\var;[s,t]}^p(1 +  \| \tilde{\mathbf{Z}} \|_{p-\var;[0,T]}^p) \\
    &\leq C \kappa^p \|V\|_{\mathcal{C}^{2 + k}_b}^p \left( 1 + \kappa^p \|V\|_{\mathcal{C}^{2 + k}_b}^p \omega(0,T) \right)  \omega(s,t),
  \end{align*}
  therefore
  \begin{align*}
   \| X^{(k)} \|_{p-\omega} \leq C \|V\|_{\mathcal{C}^{2 + k}_b} \left( 1 + \|V\|_{\mathcal{C}^{2 + k}_b} \omega(0,T)^{1/p} \right) \| X^{(k)} \|_{p-\hat{\omega}}
  \end{align*}
  and we can use again the estimates above to conclude \eqref{eqn:bounds_derivatives}. The estimate \eqref{eqn:bounds_derivatives_sup} just follows from
  \begin{equation*}
   \|X^{(k)} \|_{\infty} \leq  \| X^{(k)} \|_{p-\omega}   \| \mathbf{Z}
   \|_{p-\text{var}} + |X^{(k)}_0|. \qedhere
  \end{equation*}
\end{proof}


\bibliographystyle{plain}
\bibliography{rpde}

\end{document}